\documentclass[11pt]{article}
\usepackage[margin=1in]{geometry}
\usepackage[utf8]{inputenc} 
\usepackage[T1]{fontenc}    
\RequirePackage[colorlinks,citecolor=blue,urlcolor=blue]{hyperref}
\usepackage{hyperref}       
\usepackage{url}            
\usepackage{booktabs}       
\usepackage{amsfonts}       
\usepackage{nicefrac}       
\usepackage{microtype}      
\usepackage[square,numbers]{natbib}
\usepackage{times}
\usepackage{bbm}
\usepackage{amsmath}
\usepackage{amsthm}
\usepackage{xcolor}
\usepackage{graphicx}
\usepackage{subcaption}
\usepackage{amsthm,amsmath,natbib}
\usepackage{makecell}
\usepackage{accents}
\usepackage[capitalise]{cleveref}

\usepackage[titletoc,title]{appendix}

\usepackage{xpatch}
\makeatletter
\AtBeginDocument{\xpatchcmd{\@thm}{\thm@headpunct{.}}{\thm@headpunct{}}{}{}}
\makeatother

\usepackage{titlesec}
\titlelabel{\thetitle.\quad}

\let\hat\widehat
\let\tilde\widetilde

\renewcommand{\P}{\mbox{$\mathbb{P}$}}
\newcommand{\E}{\mbox{$\mathbb{E}$}}
\newcommand{\F}{\mbox{$\mathcal{F}$}}
\newcommand{\D}{\mbox{$\mathcal{D}$}}
\newcommand{\Tau}{\mathcal{T}}
\newcommand{\eff}{\mathrm{eff}}

\newcommand{\K}{k^*}
\newcommand{\dbtilde}[1]{\accentset{\approx}{#1}}

\DeclareMathOperator*{\argmax}{argmax}
\DeclareMathOperator*{\argmin}{argmin}
\newtheorem{theorem}{Theorem}[section]
\newtheorem{lemma}[theorem]{Lemma}
\newtheorem{claim}[theorem]{Claim}
\newtheorem{example}[theorem]{Example}
\newtheorem{corollary}[theorem]{Corollary}

\newtheorem{fact}[theorem]{Fact}
\newtheorem{proposition}[theorem]{Proposition}

\newtheorem{remark}[theorem]{Remark}

\title{On the bias, risk and consistency of sample means\\ in multi-armed bandits}

\author{%
	Jaehyeok Shin, Aaditya Ramdas and Alessandro Rinaldo \\
	Department of Statistics and Data Science \\
	Carnegie Mellon University \\
	\texttt{\{shinjaehyeok, aramdas, arinaldo\}@cmu.edu}
}

\begin{document}
	
	\maketitle
	
	\begin{abstract}%
	The sample mean is among the most well studied estimators in statistics, having many desirable properties such as unbiasedness and consistency. However, when analyzing data collected using a multi-armed bandit (MAB) experiment, the sample mean is biased and much remains to be understood about its properties. For example, when is it consistent, how large is its bias, and can we bound its mean squared error? This paper delivers a thorough and systematic treatment of the bias, risk and consistency of MAB sample means. Specifically, we identify four distinct sources of selection bias (sampling, stopping, choosing and rewinding) and analyze them both separately and together. We further demonstrate that a new notion of \emph{effective sample size} can be used to bound the risk of the sample mean under suitable loss functions. We present several carefully designed examples to provide intuition on the different sources of selection bias we study. Our treatment is nonparametric and algorithm-agnostic, meaning that it is not tied to a specific algorithm or goal. In a nutshell, our proofs combine variational representations of information-theoretic divergences with new martingale concentration inequalities. 
	\end{abstract}
	

\section{Introduction} 
In many real-world settings, data are collected in an adaptive manner from several distributions (arms), as captured by the classic stochastic multi-armed bandits (MAB) framework \cite{robbins1952some}. The data collection procedure (henceforth, algorithm) may have been primarily designed for purposes such as testing a hypothesis, minimizing regret or identifying the best arm. 
In each round, the algorithm draws a sample from one of the arms based on the previously observed data (\emph{adaptive sampling}). The algorithm may also be terminated based on a data-driven stopping rule rather than at a fixed time (\emph{adaptive stopping}). 
Even though mean estimation may not have been the primary objective,  the sample means of arms might nevertheless be calculated later on. 
For example, after identifying the best arm, it is natural to want an estimate of its mean reward.
In ``off-policy evaluation'' \cite{li2015toward}, mean reward estimates from a current policy are used to evaluate the performance of a different policy before actually implementing the latter. 
An analyst can choose a specific target arm based on the collected data (\emph{adaptive choosing}), for example focusing on certain ``promising'' arms. Furthermore, the analyst may wish to analyze the data at some past times, as if the experiment had stopped earlier (\emph{adaptive rewinding}).
 
 Among several possible mean estimators, we focus on the sample mean, which is arguably the simplest and most commonly used in practice.   In the classical nonadaptive setting, the sample mean has favorable properties. In particular, it is unbiased, consistent, and converges almost surely to the true mean, $\mu$. Additionally, under tail assumptions such as sub-Gaussian or sub-exponential conditions, the sample mean is tightly concentrated around $\mu$. Lastly, the sample mean has minimax optimal risk with respect to suitable loss functions such as the $\ell_2$ loss for distributions with a finite variance and the  Kullback-Leibler (KL) loss for distributions in a natural exponential family. 
 
The adaptive nature of the data acquisition process complicates significantly the analysis of the sample mean. In particular, it is well-known that the sample mean is biased under adaptive sampling and characterizing the bias has been a recent topic of interest. Significant progress has been made in characterizing the \emph{sign} of the bias. While estimating MAB ad revenues, \cite{xu2013estimation} gave an informal argument of why the sample mean is negatively biased for ``optimistic'' algorithms. Later, \cite{villar2015multi} encountered this negative bias in a simulation study motivated by using MAB for clinical trials. Recently, \cite{bowden2017unbiased} provided an expression for the bias, and \cite{nie2018adaptively} derived some sufficient conditions under which the bias is negative. In our recent work \cite{shin2019bias}, we extend both results, which hold only at a predetermined time and for a fixed arm, to more general adaptive setting that include adaptive sampling, stopping and choosing. There, we describe a simple monotonicity condition that determines the sign of the bias, including natural examples where it can be positively or negatively biased.
Despite this progress, it is still obscure how large the bias is, and more generally, how the sample mean estimator behaves around the true mean.
 
In this paper, we derive sufficient conditions under which the sample mean is consistent under all four aforementioned notions of adaptivity (sampling, stopping, choosing and rewinding, henceforth called the ``fully adaptive setting''). Then, we study the magnitude of its bias and risk under general  moment/tail conditions.  Adaptive mean estimation, in each of the four senses described above, has received significant attention in both recent and older literature (only studied one at a time, not together). Below, we briefly discuss how our work relates to these past works, proceeding one notion at a time in approximate historical order.

We begin by noting that a single-armed bandit is simply a random walk, a setting where adaptive stopping has been extensively studied, since even the simplest of asymptotic questions are often nontrivial. For example, if a random walk is stopped at an increasing sequence of stopping times, the corresponding sequence of stopped sample means does \emph{not} necessarily converge to $\mu$, even in probability, without regularity conditions on the distribution and stopping rules (see Ch.1 of \cite{gut2009stopped}). The book by \cite{gut2009stopped} on stopped random walks is an excellent reference, beginning from the seminal paper of \cite{wald1948optimum}, and summarizing decades of advances in sequential analysis. Some relevant authors include \cite{anscombe1952large,richter1965limit,starr1966asymptotic,starr1972further}, since they discuss inferential questions for stopped random walks or stopped tests, often in parametric and asymptotic settings. As far as we know, most of these results have not been extended to the MAB setting, which naturally involves adaptive sampling and choosing. Motivated by this, we provide new consistency results that hold in the fully adaptive setting.

For the problem of estimation following a sequential test, \cite{cox1952note} and \cite{siegmund1978estimation} developed an asymptotic expression for the size of the bias of the sample mean. Further, the moment bounds derived in \cite{de2004self, pena2008self} for self-normalized processes can be converted into bounds,  on the $\ell_p$-risk of the sample mean. However, both sets of results are only relevant to the case of a fixed arm (since they work in the one-armed setting),  a specific stopping rule (a sample mean crossing a boundary) and  a restricted class of arms (in our context, sub-Gaussian arms), and do not directly apply to adaptively chosen arms in MAB settings, unlike the results obtained in this paper. 

The recent literature on best-arm identification in MABs has often used anytime uniform concentration bounds for the sample mean of each arm around its true mean \cite{jamieson_lil_2014,kaufmann_complexity_2014} also known as finite-LIL bounds.  Historically, these were called confidence sequences~\cite{darling_confidence_1967,darling_further_1968,lai_confidence_1976}, though both theoretical and practical advances outside the MAB literature have been made recently \cite{balsubramani_sharp_2014,balsubramani_sequential_2016,howard2021time}. While these results yield high probability deviation inequalities that allow for adaptive rewinding, they cannot be immediately converted into bias and risk bounds. Below, we develop variants of these bounds and incorporate them into our risk analysis to to cover all the cases in the fully adaptive setting.

Recently \cite{russo2016controlling}  derived information-theoretic bounds for the selection bias introduced by adaptive choosing. This work, further extended by
\cite{jiao2017dependence}, showed that if a fixed number of samples is collected from each distribution, then the bias (or expected $\ell_2$ loss) of the sample mean of an adaptively chosen arm can be bounded using the mutual information between the arm index and the observed data. From our MAB perspective, these bounds only hold for a deterministic sampling strategy that is stopped at a fixed time. Our paper derives new bias and risk bounds based on the mutual information for the fully adaptive setting.

In sum, characterizing the risk and bias under all four notions of adaptivity simultaneously is an interesting and challenging  problem. Below, we summarize our contributions and describe the organization of the paper:
\begin{enumerate}
	\item We formulate sufficient conditions for consistency of a sequence of sample means in the fully adaptive setting which only require the existence of a finite mean for each arm (Proposition~\ref{prop::consistency}). 
	
	\item For the $\ell_2$ loss and for arms with finite moments, we derive risk bounds for the sample mean in the fully adaptive setting that includes an adaptive arm choice and adaptive rewinding (Theorem~\ref{thm::bounds_in_fully_adap_finite_moment} and Corollary~\ref{cor::bounds_in_finite_moment}). 
		\item 
By considering certain Bregman divergences between the sample and true mean as loss functions and for arms with exponentially decaying tails, we derive sharp risk bounds for a fixed target at a stopping time (Theorem~\ref{thm::bound_on_Bregman}) which are in turn used to derive quantitative upper and lower bounds for the bias under  adaptive sampling and stopping (Corollary~\ref{cor::bound_on_bias_l1_risk}).


		\item Under the fully adaptive setting including adaptive arm choice and adaptive rewinding, we show that by inducing a small “adaptive normalizing factor” in a log-log scale, we can extend the above results to derive bounds on the normalized risk of the sample mean to the fully adaptive setting (Theorem~\ref{thm::bounds_in_fully_adap} and Corollary~\ref{cor::bounds_in_fully_adap}).  

\end{enumerate}

The rest of the paper is organized as follows. In Section~\ref{sec::acda_setting}, we briefly formalize the four notions of adaptivity in the stochastic MAB framework. In Section~\ref{sec::consistency}, we provide a sufficient condition for the consistency of the sample mean. Section~\ref{sec::risk_under_finite_moment} provides a bound for the $\ell_2$ risk and bias of the sample mean for arms with finite moments and Section~\ref{sec:risk_under_sub_psi} extends the previous result to arms with exponentially decaying tails by introducing the Bregman divergence as a loss function. In Section~\ref{sec::proof_sketch}, we present proof techniques for main theorems. We end with a brief discussion in Section~\ref{sec::disc}, and for reasons of space, we defer all proofs to the Appendix. 

\section{The stochastic MAB framework} \label{sec::acda_setting}
For a fixed integer $\K \geq 1$, let  $P_1, \dots, P_{\K}$ be $\K$ distributions of interest (also called arms)  with finite means $\mu_k =  \E_{Y\sim P_k}[Y]$. Throughout this paper, every (in)equality between random variables is meant in the almost sure sense. 
\subsection{The data collection process}
The data are collected according to the standard MAB protocol, described as follows. See, e.g., \cite{lattimore2018bandit} for a good reference on MABs.
\begin{itemize}
    \item Let $W$ be a random variable capturing all external sources of randomness that are independent of everything else. Set $t=1$.
	\item At time $t$, let $\D_{t-1}$ be the data we have so far which is given by
	\[
	 \D_{t-1}  :=\{W,A_1, Y_1, \dots, A_{t-1}, Y_{t-1}\},
	\]
	 where $A_s$ is the (random) index of the arm sampled at time $s$ and $Y_s$ is the observation from the arm $A_s$. Based on the previous data (and possibly the  external  randomness $W$), let  
	$\nu_t(k \mid \D_{t-1})  \in [0,1]$ be the conditional probability of sampling the $k$-th arm for all $k \in [\K] := \{1,\ldots,\K\}$ with $\sum_{k=1}^{\K}\nu_t(k \mid \D_{t-1}) =1$.
	 Different choices for $\nu_t$ capture commonly used methods, such as random allocation, $\epsilon$-greedy, upper confidence bound algorithms and Thompson sampling. 
	\item
 	Sample $A_t$ from a multinomial distribution with probabilities $\{\nu_t(k \mid \D_{t-1})\}_{k=1}^{\K}$.  Let $Y_t$ be an independent draw from $P_{A_t}$. This yields a natural filtration $\left\{ \mathcal{F}_t \right\}_{t \geq 0}$ which is defined, starting with $\mathcal{F}_0 = \sigma\left(W, A_1\right)$, as
\[
\mathcal{F}_t := \sigma\left(W, A_1, Y_1,A_2, \dots,Y_t, A_{t+1}  \right),~~\forall t \geq 1.
\]
Here, $\mathcal{F}_t$ corresponds to ``the information known before observing $Y_{t+1}$''.
Then, $\{Y_t\}$ is adapted to $\left\{ \mathcal{F}_t \right\}$, and $\{A_t\}, \{\nu_t\}$ are predictable with respect to $\left\{ \mathcal{F}_t \right\}$, meaning that they are adapted to $\left\{ \mathcal{F}_{t-1} \right\}$.


\item For each $k\in[\K]$ and $t\geq 1$, define the running sum  and number of draws for arm $k$ as
$S_k(t) := \sum_{s=1}^{t} \mathbbm{1}(A_s = k) Y_s, ~~ N_k(t) := \sum_{s=1}^{t} \mathbbm{1}(A_s = k).$ Assuming that arm $k$ is sampled at least once, we define the  sample mean for arm $k$ as
\[
\hat{\mu}_k(t) := \frac{S_k(t)}{N_k(t)}.
\]
Then, $\{S_k(t)\}$, $\{\hat{\mu}_k(t)\}$ are adapted to $\{\F_t\}$ while $\{N_k(t)\}$ is predictable. 

\item Let $\Tau \geq 1$ be a stopping time with respect to $\{\F_t\}$, meaning that the event $\{t < \Tau\}$ is measurable with respect to the filtration $\F_t$. 
	If $t < \Tau$, then increment $t$. Else return the collected data \[\mathcal{D}_\Tau = \{W,A_1, Y_1, \dots, A_\Tau, Y_\Tau\}.
	\]
	We denote with $\F_\Tau$ the stopping time $\sigma$-field associated with $\Tau$ and the filtration  $\left\{ \mathcal{F}_t \right\}$, corresponding to "the information known at time $\Tau$" (see, e.g., \cite[page 367]{durrett2019probability}). In particular, $\mathcal{D}_\Tau$ is $\F_{\Tau}$-measurable.
\end{itemize}

With a slight abuse of notation, we denote with $\nu =\{\nu_t\}$ and $\Tau$ the sampling algorithm and the stopping rule, respectively. 
If, for each $t$, the sampling algorithm $\nu_t$ (or stopping rule $\Tau$) is independent of the arm realization observed so far, namely  $Y_1,\dots, Y_{t-1}$ (but not necessarily of the sampling history $A_1,\dots, A_{t-1}$), we call it a \emph{nonadaptive} sampling algorithm (or stopping rule). In particular, we denote with $T$ a random nonadaptive stopping rule (e.g., $T = 1 + Z$, where $Z$ is a Poisson random variable independent of everything else) and with $t^*$ a deterministic stopping time. Accordingly, we say that the data are  collected {\it nonadaptively} when a nonadaptive sampling algorithm and a nonadaptive stopping rule are deployed.
\begin{remark}
Throughout this paper, we use $t$ to denote the ``global'' time index, shared by all arms.  The data collection process can also be represented in terms of  ``local times" $t \mapsto \{ N_k(t), k \in [k^*]\}$, one for each arm, where $N_k(t)$ records the number of samples from arm $k$ collected by (global) time~$t$. 
Since $A_t$ is predictable with respect to $\F_t$, the optional skipping theorem for i.i.d.\ random variables  \citep[Theorem 5.2, page 145]{doob:90} implies that observations from each arm are i.i.d.\ samples under the local time indexing \cite[also see][Section 4.1]{garivier2018kl}. Consequently, the sample mean of each arm at a fixed \emph{local} time is unbiased, and all existing theories of consistency and risk of sample means can be directly applied, but this does not apply even at fixed \emph{global} times (since the corresponding local time at each arm is then random). In many bandit algorithms and in settings considered here, the algorithm is stopped at a data-dependent global time, and --- like for fixed global times --- the optional skipping argument does not apply; the stopping rule of the data collection process depends on the history of all the arms and on a global fixed horizon at which the sample mean is biased in general. 
\end{remark}

\subsection{Target for inference}
The selection of the arm whose mean is the target for estimation is made as follows:

\begin{itemize}
    \item Based on the collected data, the analyst chooses a data-dependent arm index based on a possibly randomized rule $\kappa:\mathcal{D}_\Tau~\mapsto~[\K]$. We denote the (random) index $\kappa(\mathcal{D}_\Tau)$ as just $\kappa$ for short, so that the selected mean is $\mu_\kappa$, a random parameter.  Note that $\kappa$ is $\F_{\Tau}$-measurable, but when $\kappa$ is nonadaptively chosen (i.e. it is independent of $\F_{\Tau}$) we denote it as $K$ when it is a random variable and simply as $k$ when it corresponds to a fixed arm. 
	\item Optionally, we may adaptively rewind the clock to focus on a previous random time $\tau \leq \Tau$ to characterize the past behavior of a chosen sample mean $\hat{\mu}_\kappa(\tau)$. The rewound time $\tau$ is assumed to be measurable with respect to $\F_{\Tau}$; in particular, it is {\it not} a stopping time. For instance, $\tau = \argmax_{t \leq \Tau}\hat{\mu}_{\kappa}(t)$ is a rewound time. We may care about the bias of the sample mean at this ``extreme'' time $\tau$. If we do not rewind, then $\tau=\Tau$.
\end{itemize}





The phrase ``fully adaptive setting'' refers to the scenario of running an adaptive sampling algorithm until an adaptive stopping time $\Tau$, and asking about the sample mean of an adaptively chosen arm $\kappa$ at an adaptively rewound time $\tau$. When we are not in the fully adaptive setting, we explicitly mention what aspects are adaptive. Table \ref{tab::notation} below summarizes the various types of notation for the selected arm and the chosen time.

\begin{table}[h]
\centering
\setlength\tabcolsep{1mm}
\captionsetup{labelfont=bf}
\def\arraystretch{2}
\caption{Summary of notations for time and arm indices.}
\label{tab::notation}
\begin{tabular}{cccc}
\hline
Index & Deterministic & Random but nonadaptive & Adaptive \\ \hline\hline
Time & $t, t^*$ & $T$ & $\Tau$ (stopping time), $\tau$\\ \hline
Arm & $k, \K$ & $K$    & $\kappa$ \\ \hline
\end{tabular}
\end{table}

\section{Consistency of the sample mean}\label{sec::consistency}

In sequential data analysis we often estimate the mean not just once but many times as new data become available. Let $\tau_1 \leq \tau_2 \leq \cdots$ be a sequence of non-decreasing random times, and thus $N_k(\tau_1) \leq N_k(\tau_2) \leq \cdots$. A natural question is to identify conditions under which the sample mean $\hat{\mu}_{\kappa_t}(\tau_t)$ is consistent,  in the sense that the sequence $\hat{\mu}_{\kappa_t}(\tau_t) - \mu_{\kappa_t}$ converges to zero, almost surely or in probability, as $t \rightarrow \infty$.

It is well known that the condition $\mathbb{E}\left[N_{k}(\tau_t)\right] \to \infty~~\text{as}~~t \to \infty$ is not sufficient to guarantee consistency of the sample mean even for a fixed target arm $k$, as demonstrated in the next example. 
\begin{example} \label{eg::counter_eg-consistency}
	Let $P_1$ and $P_2$ be standard normal distributions. Set $\nu_1 (1) = 1$, that is, the algorithm always picks the first distribution at $t=1$. For $t \geq 2$, set  $A_t = \mathbbm{1}(|Y_1| > z_{\alpha/2})  + 1$ where $z_\alpha$ is the $\alpha$-upper quantile of the standard normal which means that we pick a single (random) arm forever based on the first observation $Y_1$. Finally, let $t^* \geq 2$ be a deterministic stopping time. Then,  we have
	\[
	\mathbb{E}N_1(t^*) = 1 + (t^*-1)(1-\alpha)\longrightarrow \infty ~~\text{as}~~t^* \longrightarrow \infty.
	\]
	Note however that 
	\begin{align*}
		\mathbb{P} \left( \hat{\mu}_1(t^*)  > z_{\alpha/2}   \right) \geq \  \mathbb{P}\left(|Y_1| > z_{\alpha/2}  \right) = \alpha,~~ \forall t^* \geq 2.
	\end{align*} 
	Therefore $\mathbb{P} \left( \hat{\mu}_1(t^*)  > z_{\alpha/2}   \right)$ does not converge to zero even if $t^*$ and $\mathbb{E}N_1(t^*)$ approach infinity, and hence $\hat{\mu}_1(t^*)$ does not converge to the true mean $\mu_1=0$ in probability. 
\end{example}

For each fixed $k\in [\K]$, Theorem 2.1  in \cite{gut2009stopped} immediately yields that
\begin{equation} \label{eq::consistency_fixed_k_main}
\text{if}~~ N_k(\tau_t) \overset{a.s. }{\to} \infty~~\text{as}~~t \to \infty,~~\text{then} ~~\hat{\mu}_k(\tau_t) \overset{a.s.}{\to} \mu_k~~\text{as}~~t \to \infty. 
\end{equation}
Theorem 2.2 in \cite{gut2009stopped}  further implies that, in the previous display~\eqref{eq::consistency_fixed_k_main}, we can replace almost sure convergence  with the convergence in probability in both the condition and conclusion (See Appendix~\ref{Appen::prop::consistency} for formal descriptions).
In our next result, we generalize these claims to the multi-armed setting with an adaptively chosen arm. Note that here we only need the underlying distributions to have finite first moments.  
\begin{proposition} \label{prop::consistency}
	The following statements hold for any sequence of choice functions 
	
	$\kappa_t:\D(\tau_t)\to[\K]$ that are based on data up to time $\tau_t$:
	\begin{align} 
	&\text{if}~~ N_{\kappa_t}(\tau_t) \overset{a.s. }{\to} \infty~~\text{as}~~t \to \infty,~~\text{then}~~\hat{\mu}_{\kappa_t}(\tau_t) - \mu_{\kappa_t} \overset{a.s.}{\to} 0~~\text{as}~~t \to \infty, \text{ and} \label{eq::consistency_chosen_k_as}\\
	&\text{if}~~ N_{\kappa_t}(\tau_t) \overset{p}{\to} \infty~~\text{as}~~t \to \infty,~~\text{then}~~\hat{\mu}_{\kappa_t}(\tau_t)  - \mu_{\kappa_t}\overset{p}{\to} 0~~\text{as}~~t \to \infty.  \label{eq::consistency_chosen_k_in_p}
	\end{align}
\end{proposition}

\noindent The proof of the proposition is deferred to  Appendix~\ref{Appen::prop::consistency}.  

Since $\left|\hat{\mu}_{\kappa_t}(\tau_t)  - \mu_{\kappa_t} \right| = \sum_{k =1}^{\K} \mathbbm{1}\left(\kappa_t= k\right)\left|\hat{\mu}_{k}(\tau_t)  - \mu_{k}\right|$, if we  instead assume the stronger condition that $N_k(\tau_t)\rightarrow\infty$ for all $k \in [\K]$, almost surely or in probability, Theorem 2.1 and 2.2 in \cite{gut2009stopped} immediately imply consistency of the sample means. However, the following example demonstrates that even if the number of draws for each \emph{fixed} arm does not converge to infinity, Proposition~\ref{prop::consistency} can guarantee the consistency of the chosen sample mean which, in contrast, cannot be directly implied by Theorem 2.1 and 2.2 in \cite{gut2009stopped}.

\begin{example}\label{eg::consistency-full}
	Let $P_1$ and $P_2$ be two identical continuous distributions with finite means. Set $\nu_1 (1) = 1$ and $\nu_2(2) = 1$, meaning that we begin by sampling each arm once. For all times $t \geq 3$, we set  $A_t = \mathbbm{1}(Y_1 >Y_2)  + 1$, meaning that we pick a single (random) arm forever. Finally, let $t^* \geq 3$ be a deterministic stopping time. The number of draws from each arm does not diverge to infinity either almost surely or in probability as $t^* \to \infty$ since for $k=1,2$, we have
	\[
	\mathbb{P}\left(N_k(t^*) \leq 1\right) = \frac{1}{2},~~\forall t^* \geq 3.
	\]	
	Now, let $\kappa = \mathbbm{1}\left(N_1(t^*)\leq N_2(t^*)\right)  + 1$, that is, we choose the arm with more data when we stop. Then, the number of draws from the chosen arm is always equal to $t^*-1$ and the sufficient condition in Proposition~\ref{prop::consistency} is trivially satisfied. Thus, even though $N_k(t^*)$ does not diverge to $\infty$ (almost surely or in probability) for any fixed $k$, our proposition still guarantees that $\hat{\mu}_{\kappa}(t^*) - \mu_{\kappa} \overset{a.s.}{\longrightarrow} 0 \text{ as}~~t^* \to \infty.$
\end{example}
The above example demonstrates the additional subtlety in the conditions for consistency when moving from a fixed arm to an adaptively chosen one.

\section{Risk of sample mean under arms with finite moments} \label{sec::risk_under_finite_moment}
Having established that the sample means are consistent estimators of the true means of the arms, in the subsequent sections we will turn to the more challenging tasks of deriving finite sample bounds on the magnitude of both the bias and the risk under different nonparametric assumptions on the arms. 
Specifically, we will be concerned with two notions of $\ell_2$ risk for the sample mean estimator: the classic or {\it unnormalized} one, corresponding to the squared error loss and taking the form
\begin{equation}
\text{[Unnormalized $\ell_2$ risk]} \quad \mathbb{E} \left[ \left(\hat{\mu}_{\kappa}(\tau) - \mu_{\kappa}\right)^2 \right],
\end{equation}
and a weighted or {\it normalized} variant, defined as  
\begin{equation}
\text{[Normalized $\ell_2$ risk]} \quad \mathbb{E}\left[ N_{\kappa}(\tau) \left(\hat{\mu}_{\kappa}(\tau) - \mu_{\kappa}\right)^2 \right].
\end{equation}
As we will see shortly, the unnormalized risk is a function of both the sampling algorithm and the stopping rule, while the normalized risk is upper bounded by a term that only depends on the choosing rule. The two types of risk bounds are rather different in both their form and interpretability, and each elucidate complementary aspects of the problems. In addition to normalized and unnormalized $\ell_2$ bounds, we will also give analogous $\ell_1$ bounds.

For each $p \geq 1$ and $k \in [\K]$, we define the centered $p$-norm of arm $k$ as 
\begin{equation}
    \sigma_k^{(p)} : = \left(\int \left|x -\mu_k\right|^{p} \mathrm{d}P_k(x)\right)^{1/p}.
\end{equation}
From Jensen's inequality, we know that if $p_1 \leq p_2$ then $\sigma_k^{(p_1)} \leq \sigma_k^{(p_2)}$ for each $k \in [\K]$. With a slight abuse of notation, below we will denote the standard deviation of the $k$-th arm with $\sigma_k$ instead of $\sigma_k^{(2)}$.  



\subsection{Unnormalized and normalized \texorpdfstring{$\ell_2$}{l2} risks under nonadaptive sampling and stopping}

Recall that if the sampling algorithm (or stopping rule) is independent of the realizations of the arms, we call it a nonadaptive sampling algorithm (or stopping rule, respectively).  
Under nonadaptive sampling and stopping, the unnormalized $\ell_2$-risk of the sample mean for arm $k$ is given by
\begin{equation}
\begin{aligned}
\mathbb{E}\left(\hat{\mu}_k(T) - \mu_k\right)^2
 &= \mathbb{E}\left[\mathbb{E} \left[\left(\hat{\mu}_k(T) - \mu\right)^2 \mid \{A_t\}_{t=1}^T\right]\right]\\
&= \mathbb{E} \left[\frac{\sigma_k^2}{N_k(T)}\right] ,
\end{aligned}
\end{equation}
where the second equality comes from the independence assumption on the sampling algorithm and stopping rule and the fact that $\mathbb{E} \left(\hat{\mu}_k(n) -\mu_k\right)^2 = \frac{\sigma_k^2}{n}$ where $\hat{\mu}_k(n)$ is a sample average of $n$ i.i.d.\ observations from a distribution with mean $\mu_k$ and variance $\sigma_k^2$.
Next, define the effective sample size for arm $k$ as
\begin{equation}
 n_k^\eff:= \left[\mathbb{E}\left(1/ N_k(T)\right)\right]^{-1}.  
\end{equation}
Then, under nonadaptive sampling and stopping, the $\ell_2$ risk of the sample mean for arm $k$ can immediately be seen to be equal to 
\[
\mathbb{E}\left(\hat{\mu}_k(T) - \mu_k\right)^2  = \frac{\sigma_k^2}{n_k^\eff}.
\]
Clearly, the effective sample size $n_k^\eff$ depends on both the nonadaptive sampling algorithm and the stopping rule, as it quantifies the combined effects of these rules on the $\ell_2$ risk of $\hat{\mu}_k(T)$. In contrast, the normalized risk is agnostic to the choices of such rules. In detail, we show next that the minimax normalized $\ell_2$ risk for estimating the mean of the $k$th arm over all nonadaptive data collection procedures is $\sigma^2_k$, and this risk value is achieved by the sample mean.  
 

\begin{proposition} 
\label{prop::minimax_nonadaptive}
For any fixed $k \in [\K]$, let $\mathbb{P}_k(\mu_k,\sigma_k)$ be the class of distributions on an arm $k$ with mean $\mu_k$ and variance $\sigma_k^2$. Let $\mathbb{V}$ and $\mathbb{T}$ be classes of nonadaptive sampling algorithms and stopping rules satisfying $N_k(T) \geq 1$. Finally, let $Q = Q(P_k, \nu, T)$ be the induced distribution on the observations from the arm $k$ with  $P_k$ the distribution under a nonadaptive sampling algorithm $\nu$ and stopping $T$.  Then, the minimax normalized $\ell_2$ risk is given by
	\begin{equation} \label{eq::minimax_nonadaptive}
	    \inf_{\tilde{\mu}_k}\sup_{\substack{\mu_k \in \mathbb{R}\\P_k \in \mathbb{P}_k(\mu_k,\sigma_k) \\ \nu \in \mathbb{V}, T \in \mathbb{T}}}\mathbb{E}_{Q}\left[N_k(T)\left(\tilde{\mu}_k(T) - \mu_k\right)^2\right] = \sigma_k^2,
	\end{equation}
	where the infimum is over all estimators. Furthermore, for any given $P_k \in \mathbb{P}_k(\mu_, \sigma_k^2)$, $\nu \in \mathbb{V}$ and $T \in \mathbb{T}$, the sample mean estimator achieves the minimax risk.
\end{proposition}
The proof of the proposition is based on standard decision-theoretic arguments and can be found in Appendix~\ref{Appen::subSec::minimax_nonadap}.

 In practice, we often do not know ahead of time which arm $k$ would be the most interesting to study before looking at the data. For instance, we may want to estimate the mean for the arm with the largest observed empirical mean, or the second largest, or even the smallest. In this case, the target of inference is $\mu_\kappa$, where $\kappa$ is an adaptive choice which possibly depends on the collected data $\D_{\Tau}$. 

Following \cite{jiao2017dependence}, to quantify dependence between $\kappa$ and $\D_{\Tau}$, we deploy the information-theoretic dependence measure (an $f_q$-divergence)
\begin{equation}
    I_{q}(\kappa ; \D_{\Tau}):= D_{f_q}\left(P_{\left(\kappa, \D_{\Tau}\right)}| P_{\kappa}\otimes P_{\D_{\Tau}} \right),
\end{equation}
where $q \geq 1$, $f_{q}(x):=|x-1|^q$ and
$D_{f_q}(Q' |Q):=  \int f_q\left(\frac{\mathrm{d}Q'}{\mathrm{d}Q}\right) \mathrm{d} Q$, assuming that $Q' \ll Q$. It can be easily checked that $I_q(\kappa; \D_{\Tau}) \geq 0$ and that $I_q(\kappa; \D_{\Tau}) = 0$ if and only if $\kappa$ and $\D_{\Tau}$ are independent.  It can also be  showed that $I_q(\kappa, \D_{\Tau})$ can be upper bounded as
 \[
 I_q(\kappa, \D_{\Tau}) \leq 1 + \sum_{k=1}^{\K} p_k^2\left(\left|\frac{1}{p_k}-1\right|^q -1\right),
 \]
 where $p_k := \mathbb{P}\left(\kappa = k\right),~~\forall k \in [\K]$. In particular, 
 \[
 I_q(\kappa, \D_{\Tau}) \leq \frac{\K -1}{\K}\left[(\K - 1)^{q-1} + 1\right] < 1 + (\K)^{q-1},
 \]
 for $1 \leq q \leq 2$ \cite[Lemma~1]{jiao2017dependence}.

For nonadaptive sampling and stopping (hence, $\Tau=T$), \cite{jiao2017dependence} showed how to bound the bias of adaptively chosen random variables with finite moments by using $I_{q}(\kappa ; \D_T)$. More precisely, suppose each $\sqrt{n}\left(\hat{\mu}_k - \mu_k\right)$ has zero mean and its $p$-norm is given by $\left(\mathbb{E}\left|\sqrt{n}\left(\hat{\mu}_k - \mu_k\right)\right|^p\right)^{1/p}=\sigma_k^{(p)}$. Also, for any $r \geq 1$, let $\|\sigma_\kappa^{(p)}\|_r$ be the $r$-norm of $\sigma_\kappa^{(p)}$, naturally defined by 
\begin{equation}
\|\sigma_\kappa^{(p)}\|_r = \left(\sum_{k=1}^{\K} \mathbb{P}(\kappa = k)\left(\sigma_k^{(p)}\right)^r \right)^{1/r}.
\end{equation}
Then, for any $p, q> 1$ with $1/p + 1/q =1$, the (scaled) bias of $\hat{\mu}_{\kappa}$ can be bounded as
\begin{equation}\label{eq::jiao_bound_on_bias}
\left|\mathbb{E}\sqrt{n}\left(\hat{\mu}_\kappa - \mu_\kappa\right)\right| \leq \|\sigma_\kappa^{(p)}\|_p I_q^{1/q} (\kappa, \D_T).
\end{equation}
The above result~\cite{jiao2017dependence} can be extended to a bound on the $\ell_2$ risk of an adaptively chosen sample mean (under nonadaptive sampling and stopping) as follows.

\begin{proposition} \label{prop::bounds_l2_finite_moment}
  Consider a nonadaptive sampling algorithm and stopping rule, and assume that each arm has a finite $2p$-norm for a given $p>1$. Then, the normalized $\ell_2$ risk of the sample mean can be bounded as
  \begin{equation} \label{eq::bounds_l2_finite_moment}
      \mathbb{E}\left[N_\kappa(T)\left(\hat{\mu}_\kappa(T) - \mu_\kappa\right)^2\right] \leq \left\|\sigma_\kappa\right\|_2^2 + C_p\left\|\sigma_\kappa^{(2p)}\right\|_{2p}^2 I_q^{1/q} (\kappa, \D_T),
  \end{equation}
  where $C_p$ is a constant depending only on $p$ and $q > 1$ is such that $1/p + 1/q =1$.
\end{proposition}
The proof of Proposition~\ref{prop::bounds_l2_finite_moment} can be found in Appendix~\ref{Appen::subSec::bounds_l2_finite_moment} and is based on a variational representation of the $f_q$-divergence along with the Marcinkiewicz–Zygmund inequality~\cite{marcinkiewicz1937fonctions}.  

The upper bound in \cref{prop::bounds_l2_finite_moment} cannot be improved in general. For example, if $\mathbb{P}\left(\kappa = k\right) = 1$ then the mutual dependence term $I_q^{1/q} (\kappa, \D_{\Tau})$ is equal to $0$ and we recover the exact $\ell_2$ risk $\sigma_k^2$. Similarly, if the selected arm $K$ is chosen in a random but nonadaptive manner, then $I_q^{1/q} (K, \D_{\Tau}) = 0$ and  the bound in \eqref{eq::bounds_l2_finite_moment} reduces to $\|\sigma_K\|_2^2$. When $\kappa$  is adaptively chosen and $I_q^{1/q} (\kappa, \D_{\Tau}) > 0$, it is unclear whether the bound in \cref{prop::bounds_l2_finite_moment} is tight, but its tightness or lack thereof is quite analogous to the bounds in earlier work, where the sampling and stopping are nonadaptive~\eqref{eq::jiao_bound_on_bias}.



 \subsection{Normalized \texorpdfstring{$\ell_2$}{l2} risk and unnormalized \texorpdfstring{$\ell_1$}{l1} risk under fully adaptive settings} \label{subSec::moment_full_adap}
  The techniques used in the previous section deliver risk bounds only under nonadaptive sampling and stopping rules but they do not generalize readily to fully adaptive settings.
 In particular, the bias bound in \cite{jiao2017dependence} and the risk bound in Proposition~\ref{prop::bounds_l2_finite_moment} are not directly applicable because each $\hat{\mu}_k(\tau) -\mu_k$ is no longer centered, due to the bias caused by adaptive sampling, stopping and rewinding.  Furthermore, the bound for the bias given in equation \eqref{eq::jiao_bound_on_bias}  no longer holds under the fully adaptive setting because the bias can be non-zero even if $\kappa$ is independent of $\mathcal{D}$. 

Below we show that the normalized $\ell_2$ risk bound for nonadaptive sampling and stopping strategies given in Proposition~\ref{prop::bounds_l2_finite_moment} generalizes to the fully adaptive setting, assuming the existence of higher moments and with a slightly stronger risk normalization factor of $N_k(\tau)/\log N_k(\tau)$, which can be regarded as a (small) price for adaptivity.

For any $t$ such that $N_k(t) > 1$ for all $k \in [k^*]$, we define
\begin{equation}
\label{eq:normalized-samplesize-log}
    \tilde{N}_k(t) := \frac{N_k(t)}{\log N_k(t)}, \quad  k \in [\K].
\end{equation}
We now present the main result of this section.
\begin{theorem} \label{thm::bounds_in_fully_adap_finite_moment}
Suppose each arm has a finite $2(p+\epsilon)$-norm for some $p \geq 1$ and $\epsilon >0$.  Consider any adaptive sampling algorithm such that $\min_{k \in [\K]} N_k(t_0) \geq 3$  almost surely for some deterministic time $t_0$, and adaptive stopping rule $\Tau \geq t_0$. Then, for any adaptive rewound time $t_0 \leq \tau \leq \Tau$ and adaptively chosen arm $\kappa$, it holds that
\begin{align}\label{eq::adap_risk_finite_moment}
	\mathbb{E} \left[\tilde{N}_\kappa (\tau) \left(\hat{\mu}_{\kappa}(\tau)- \mu_{\kappa} \right)^2 \right]
	&\leq C_{1,\epsilon}\|\sigma_\kappa\|_2^2 + C_{p,\epsilon}\|\sigma_\kappa\|_{2p}^2  I_q^{1/q}\left(\kappa, \D_{\Tau}\right),
	\end{align}
where $q > 1$ satisfies $1/p + 1/q =1$ and $C_{1, \epsilon}$ is a positive constant depending only on $\epsilon$, and $C_{p,\epsilon}$ is a positive constant depending only on $p,\epsilon$.
\end{theorem}


Compared to the nonadaptive risk bound in Proposition~\ref{prop::bounds_l2_finite_moment}, the bound~\eqref{eq::adap_risk_finite_moment} under the fully adaptive setting only suffers a multiplicative logarithmic normalization term $\log N_\kappa(\tau)$ under slightly stronger moment condition.  It is also important to note that the bound~\eqref{eq::adap_risk_finite_moment} depends on the second moment terms $\{\sigma_k\}_{k=1}^{\K}$ only, and not on any higher moment.

The proof of Theorem~\ref{thm::bounds_in_fully_adap_finite_moment}, given in Section~\ref{subSec::Proof_of_second_them_finite_moment}, combines  the  novel deviation inequality for the normalized $\ell_2$ loss of Lemma \ref{lemma::adaptive_deviation_ineq_finite_moment} below, which holds for a fixed arm $k$, with the variational representation of the $f_q$-divergence which handles adaptive choosing.

\begin{lemma} \label{lemma::adaptive_deviation_ineq_finite_moment}
	Consider an adaptive sampling algorithm and stopping rule. For a fixed arm $k \in [\K]$ with a finite $2p$-norm, where $p>1$, and any random time $\tau$ such that $N_k(\tau) \geq 3$ almost surely, it holds that, for any $\delta \geq 0$,
	\begin{equation} \label{eq::adaptive_deviation_ineq_finite_moment}
	\mathbb{P}\left(\tilde{N}_k (\tau)\left(\frac{\hat{\mu}_k(\tau)-\mu_k}{\sigma_k} \right)^2 \geq \delta \right)  
	\leq \frac{C_p}{\delta^p},
	\end{equation}
	where $C_p$ is a constant depending only on $p$. 
\end{lemma}

 We believe this is the first polynomially decaying tail bound on the $\ell_2$ risk of the sample mean that holds at any arbitrary random time and only assuming arms with finite first $2p$ moments. This inequality is thus possibly of independent interest;  its proof is based on the $\ell_p$-version of the Dubins-Savage inequality given by \cite{khan2009p}; see Appendix~\ref{subSec::proof_of_adaptive_ineq_finite_moment}.


It is interesting to reflect on whether the $\log N_k(\tau)$ in~\eqref{eq::adaptive_deviation_ineq_finite_moment} can be replaced by a term like $\log\log N_k(\tau)$. At first glance, the answer seems subtle: if $\Tau_b$ is the stopping time defined in \cref{eg::lil_stopping}, then
\begin{equation}
    	\mathbb{P}\left(\frac{N_k(\Tau_b)}{\log\log N_k(\Tau_b)}\left(\frac{\hat{\mu}_k(\Tau_b)-\mu_k}{\sigma_k} \right)^2 \geq 1 \right)  = 1.
\end{equation}
More directly, the law of iterated logarithm implies 
\begin{equation}
    	\mathbb{P}\left(\exists t \geq 1: \frac{N_k(t)}{\log\log N_k(t)}\left(\frac{\hat{\mu}_k(t)-\mu_k}{\sigma_k} \right)^2 \geq 2 \right)  = 1,
\end{equation}
provided that $N_k(t) \to \infty$ as $t \to \infty$. Could constants larger than 2 be employed? The answer is yes, but the variance must be modified using a rather sophisticated and relatively uncommon self-normalization bound. Combining the results of \citep[Lemma 3(f)]{howard2020time}  with \citep[Theorem 1]{howard2021time}, it is possible to infer that if $N_k(\tau) \geq b \geq 3$ then, for any $\delta > 1$,
\begin{equation}
   	\mathbb{P}\left(\frac{N_k(\tau)}{\log\log N_k(\tau)}\left(\frac{\hat{\mu}_k(\tau)-\mu_k}{\tilde{\sigma}_k(\tau)} \right)^2\geq C_b \delta\right)
	\leq 2\exp\left\{-\delta\right\},
\end{equation}
where $C_b$ is a constant larger than $2e$ that depends only on $b$ and $\tilde{\sigma}_k^2(\tau)$ is defined by
\begin{equation}
    \tilde{\sigma}_k^2(\tau):= \frac{2}{3}\sigma_k^2 + \frac{1}{3N_k(\tau)}\sum_{s=1}^\tau \mathbbm{1}(A_s = k)(Y_s - \mu_k)^2 := \frac{2}{3}\sigma_k^2  + \frac{1}{3}\hat{\sigma}_k^2(\tau).
\end{equation}
Note that, if $N_k(\tau) \to \infty$ then $ \tilde{\sigma}_k^2(\tau) \to \sigma_k^2$. This bound is similar to one in \cref{lemma::adaptive_deviation_ineq}. Since this leads us further afield than the purposes of this paper, we do not further discuss such bounds.

Now, for any $r >0$, define the $r$-th order logarithmically discounted sample size of an adaptively chosen arm as 
\begin{equation}
\tilde{n}^{\eff,r}_\kappa := \left[\mathbb{E}\left[1/\tilde{N}_\kappa^r (\tau)\right]\right]^{-1/r}, \quad r >0,
\end{equation}
where the expectation is over the randomness in all four sources of adaptivity. This quantity is nonrandom, and the subscript $\kappa$ merely differentiates it from the effective sample size of a fixed arm, and is not to be interpreted as residual randomness. It is easy to check that $\tilde{n}^{\eff,r}_\kappa$ is decreasing with respect to $r$, by Jensen's inequality.
The following corollary provides bounds for the unnormalized $\ell_{2r}$ risk of the sample mean for all $r \in (0,1)$ based on $\tilde{n}^{\eff, r}_\kappa$. The proof of the corollary can be found in Appendix~\ref{subSec::proof_of_bound_in_finite_moment}.
\begin{corollary}\label{cor::bounds_in_finite_moment}
	Suppose each arm has a finite $2(p+\epsilon)$-norm for some $p \geq 1$ and $\epsilon >0$. Then, for any $r \in (0,1)$, the $r$-quasi-norm of the $\ell_2$-loss is bounded as
	\begin{equation}\label{eq::1/p-norm_bound_finite_moment}
    \left[\mathbb{E}\left(\hat{\mu}_{\kappa}(\tau)- \mu_{\kappa} \right)^{2r}\right]^{1/r}
	\leq \frac{C_{1,\epsilon}\|\sigma_\kappa\|_2^2 + C_{p,\epsilon}\|\sigma_\kappa\|_{2p}^2  I_q^{1/q}\left(\kappa, \D_{\Tau}\right)}{\tilde{n}^{\eff,r/(1-r)}_\kappa} .
	\end{equation}
\end{corollary}

\noindent In particular, by choosing $r = 1/2$, the above results immediately yields a bound for the $\ell_1$ risk:
\begin{equation}\label{eq::1/2-norm_bound_finite_moment}
\mathbb{E}\left|\hat{\mu}_{\kappa}(\tau)- \mu_{\kappa} \right|
	\leq \sqrt{\frac{C_{1,\epsilon}\|\sigma_\kappa\|_2^2 + C_{p,\epsilon}\|\sigma_\kappa\|_{2p}^2  I_q^{1/q}\left(\kappa, \D_{\Tau}\right)}{\tilde{n}^{\eff,1}_\kappa}}.
\end{equation}

\noindent Note that if the choosing rule $\kappa$ is equal to $k$ (so that $I_q(\kappa, \D_{\Tau}) =0$), the $\ell_1$ risk bound \eqref{eq::1/2-norm_bound_finite_moment} matches to the nonadaptive standard $\ell_1$ risk bound, of order of $\sigma_k / \sqrt{n}$, with the sample size $n$ replaced by the logarithmically discounted effective sample size $\tilde{n}^{\eff,1}_k$. In Section~\ref{subSec::fixed_target}, we derive an alternative bound that depend on the undiscounted effective sample size $n_k^{\eff}$, for a fixed target arm and at a stopping time, by assuming stronger tail conditions. 

One may wonder whether the logarithmic discounting factor in the normalized risk is necessary to derive a finite upper bound. For arms with finite variance, the following  example shows that in general there is no finite upper bound on the normalized risk $\mathbb{E} \left[N_{k}(\Tau) \left(\hat{\mu}_{k}(\Tau)- \mu_{k} \right)^2 \right]$. 

\begin{example} \label{eg::lil_stopping}
Suppose each arm has a finite variance $\sigma_k^2$. For a fixed $k$, assume that $N_k(t) \rightarrow \infty$ almost surely as $t \rightarrow \infty$. For any $b \geq 3$, we define the stopping rule
    \begin{equation}
        \Tau_b := \inf\left\{t \geq 1 : N_k(t) \geq b ~~\text{and}~~ \frac{S_k(t) - \mu_k N_k(t)}{\sigma_k\sqrt{N_k(t) \log\log N_k(t)}} \geq 1 \right\}.
    \end{equation}
Due to the law of the iterated logarithm, we know that $\mathbb{P}\left(\Tau_b < \infty \right) = 1$. From the definition of $\Tau_b$, we immediately infer that
\begin{equation} \label{eq::lower_bound}
 \sigma_k^2\mathbb{E}\left[\log\log N_k(\Tau_b)\right] 
 ~\leq~ \mathbb{E} \left[N_{k}(\Tau_b) \left(\hat{\mu}_{k}(\Tau_b)- \mu_{k} \right)^2 \right].
\end{equation} 
Since the left hand side approaches infinity as $b \to \infty$, we see that there is no finite upper bound on the normalized risk $\mathbb{E} \left[N_{k}(\Tau) \left(\hat{\mu}_{k}(\Tau)- \mu_{k} \right)^2 \right]$ in general.
\end{example}
The above example demonstrates that some correction to the normalized risk, such as the logarithmic discounting of the sample size in \eqref{eq:normalized-samplesize-log}, is necessary to derive a finite risk bound like in Theorem~\ref{thm::bounds_in_fully_adap_finite_moment}.
It is unclear whether the logarithmic discounting we used is optimal or if a smaller factor would have sufficed. 

In the next section we show that, for arms with exponentially decaying tails, we can deploy a smaller discounting factor, measured on  a log-log scale, that leads to  upper and lower bounds matching up to a constant term; see Theorem \ref{thm::bounds_in_fully_adap}.

\section{Risk bounds for arms with exponential tails}\label{sec:risk_under_sub_psi}
In this section we will assume stronger tail-decaying conditions on the arms and derive risk bounds for the sample means under various degree of adaptivity.  While the analysis and results might be cleanest for the mean-squared error of sub-Gaussian distributions (as is commonly assumed in the bandit literature), the proof for the more general case involving Bregman risks of sub-$\psi$ distributions (to be defined below) follows the same line of argument and hence we choose to present the results in a unified way. The sub-Gaussian results can be easily inferred as a special case.

\subsection{Sub-\texorpdfstring{$\psi$}{psi} arms and Bregman divergences as loss functions}\label{sec:sub-psi}
For fixed numbers $\lambda_{\min} < 0 < \lambda_{\max}$, let $\Lambda = (\lambda_{\min}, \lambda_{\max}) \subseteq \mathbb{R}$  be an open interval that contains $0$. A function $\psi : \Lambda \rightarrow [0, \infty)$ is called CGF-like if it obeys natural properties of a cumulant generating function (CGF), specifically that it is a non-negative, twice-continuously differentiable  and strictly convex function $\psi(0) = \psi^\prime (0) = 0$. 

A probability distribution $P$ is called \emph{sub-$\psi$} if the CGF of the centered distribution exists and is equal to or upper bounded by a ``CGF-like'' function $\psi$, that is,
\begin{equation}
\log \E_{Y\sim P}[e^{\lambda (Y - \mu)}] \leq \psi(\lambda) , ~~\forall \lambda \in  \Lambda \subseteq \mathbb{R}.
\end{equation}
This assumption is quite general and applies to all distributions with a CGF, including natural exponential family distributions, sub-Gaussian and sub-exponential distributions   \cite{howard2020time, howard2021time}. Throughout this section, we assume each arm is in a sub-$\psi$ class unless otherwise specified.

Our analyses make frequent use of the function $\psi_{\mu}^* : \Lambda^* \rightarrow \mathbb{R}$, the convex conjugate of $\psi_\mu(\lambda) := \lambda \mu+ \psi(\lambda)$ defined as
\begin{equation}
\psi_\mu^*(z) := \sup_{\lambda \in \Lambda} \lambda z - \psi_\mu (\lambda),~~\forall z \in \Lambda^* := \left\{x \in \mathbb{R} : \sup_{\lambda \in \Lambda} \lambda x - \psi_\mu (\lambda) < \infty   \right\}.
\end{equation}

For arms in a sub-$\psi$ class, it turns out to be natural to define the loss function as the Bregman divergence with respect to $\psi_\mu^*$:
\begin{equation}
D_{\psi_{\mu}^*} (\hat{\mu}, \mu) = \psi_{\mu}^* (\hat{\mu}) - \psi_{\mu}^* (\mu) - \psi_{\mu}^{*\prime}(\mu) \left(\hat{\mu}  - \mu\right).
\end{equation}
For instance, if the underlying distribution is sub-Gaussian, then the Bregman divergence reduces to the scaled $\ell_2$ loss. For more examples, see Appendix~\ref{sec::example_divergence}.
More generally, the Bregman divergence is equivalent to the KL loss when the underlying distribution is a natural univariate exponential family with a density 
\[p_\theta (x) =  \exp\left\{\theta x -  B(\theta) \right\}, ~~\theta \in \Theta \subset \mathbb{R},\]
with respect to a reference measure $\gamma$, where $\Theta \subset \left\{\theta \in \mathbb{R} :  \int e^{\theta x } \gamma(\mathrm{d}x) < \infty \right\}$ is the natural parameter space and $B \colon \Theta \rightarrow \mathbb{R}$ is a strictly convex function given by $\theta \mapsto \int e^{\theta x } \gamma(\mathrm{d}x)$. We assume throughout that $\Theta$ is nonempty and open.

For a fixed $\theta \in \Theta$, define $\Lambda_\theta := \left\{\lambda \in \mathbb{R} : \lambda + \theta \in \Theta\right\}$ and, for each $\lambda \in \Lambda_\theta$, let $\psi(\lambda) = \psi(\lambda; \theta) := B(\lambda + \theta) - B(\theta) -\lambda B^\prime(\theta)$. Using the properties of the log-partition function $B$, it can be easily checked that $p_\theta$ is sub-$\psi$. Since $B$ is strictly convex, there is a one-to-one correspondence between the natural parameter space  and the mean value parameter space $\mathrm{M} = \{ \mu \in \mathbb{R} \colon \mu = B'(\theta), \theta \in \Theta\}$. For any $\mu_0,\mu_1$ in the mean parameter space, let $\theta_0,\theta_1$ be corresponding natural parameters. The KL divergence between $p_{\theta_1}$ and $p_{\theta_0}$ induces a natural loss between $\mu_1$ and $\mu_0$ which is often called the KL loss:
\[
\ell_{KL}(\mu_1, \mu_0) := D_{KL}\left(p_{\theta_1} \|p_{ \theta_0} \right).
\]
The following well-known fact, based on the properties of the CGF of an exponential family and the duality of Bregman divergence, formally captures  how the KL loss is related to the Bregman loss \cite{azoury2001relative}. For completeness, we present a proof  in Appendix~\ref{subSec::prop_KL_equiv_psi_star}.
\begin{fact} \label{fact::KL_equiv_psi_star}
	Let $\psi$ be the CGF of a centered distribution in a one-dimensional exponential family. Then, for any $\mu_1$ and $\mu_0$ in the mean parameter space, we have
	\begin{equation} \label{eq::KL_equiv_psi_star}
	\ell_{KL}(\mu_1, \mu_0) = D_{\psi_{\mu_0}^*} (\mu_1, \mu_0) =\psi_{\mu_0}^* (\mu_1) = \psi^*(\mu_1 - \mu_0).
	\end{equation}
	Further, the last two equalities hold for any CGF-like $\psi$.
\end{fact}
\noindent

Since the identity~\eqref{eq::KL_equiv_psi_star} recovers the $\ell_2$ loss for sub-Gaussian arms and the KL loss for exponential family arms, the Bregman divergence $D_{\psi_{\mu_\kappa}^*} (\hat{\mu}_\kappa(\tau), \mu_\kappa)$ is a natural loss function for the mean value parameter when the arms are sub-$\psi$. Also, by Taylor expanding $\psi_{\mu_0}^*$ around $\mu_0$, one can check that the Bregman divergence is locally equivalent to the $\ell_2$ loss since $D_{\psi_{\mu_0}^*}(\mu_1,\mu_0) =  \frac{1}{2\sigma_{\mu_0}^2}(\mu_1 - \mu_0)^2 + o(|\mu_1 - \mu_0|^2)$ for any $\mu_1$ and $\mu_0$ in the mean parameter space with $\sigma_{\mu_0}^2 := \psi_{\mu_0}''(0)$. In particular, if $\sigma_{\max}^2 := \sup_{\mu \in \mathcal{M}} \psi_{\mu}''(0) < \infty$, then the Bregman divergence is lower bounded as $D_{\psi_{\mu_0}^*}(\mu_1,\mu_0) \geq  \frac{1}{2\sigma_{\max}^2}(\mu_1 - \mu_0)^2 $. For example, the Bernoulli distribution with mean $\mu$ is a sub-$\psi$ distribution with $\psi(\lambda; \mu) = \log\left(1-\mu + \mu e^\lambda\right) - \lambda\mu$, and the corresponding Bregman divergence (which is equal to the KL divergence by \cref{fact::KL_equiv_psi_star}) is lower bounded for any $\mu_0, \mu_1 \in (0,1)$ as
\begin{align*}
 D_{\psi_{\mu_0}^*} (\mu_1, \mu_0) &= \ell_{KL}(\mu_1, \mu_0) 
 = \mu_1 \log\left(\frac{\mu_1}{\mu_0}\right) + (1-\mu_1) \log\left(\frac{1-\mu_1}{1-\mu_0}\right) \geq 2(\mu_1 - \mu_0)^2.
\end{align*}


Below, we will show that in the deterministic setting where a fixed number $n$ of independent observations are drawn from a single fixed distribution, the minimax Bregman risk for distributions belonging to an exponential family is of order $\frac{1}{n}$, and that  the sample mean is minimax rate-optimal. To get a lower bound, we need an additional regularity condition on the loss function.  For any function $d : \mathrm{M} \times \mathrm{M} \to [0,\infty)$, we say that  $d$ satisfies the \emph{local triangle inequality condition} \cite{yang1999information} if there exist positive constants $M \leq 1$ and $\epsilon_0$ such that for any $\mu_0, \mu_1, \mu_2 \in \mathrm{M}$, if $\max\left\{d(\mu_2, \mu_0), d(\mu_2, \mu_1)\right\} \leq \epsilon_0$, then $d(\mu_2, \mu_0) + d(\mu_2, \mu_1) \geq M \max\left\{ d(\mu_0, \mu_1), d(\mu_1, \mu_0)\right\}$. The local triangle inequality condition is satisfied by the square root KL divergence between Gaussian distributions with $M =1$. For general exponential family distributions, we may restrict the parameter space to make the condition satisfied. In particular, if $\inf_{\theta\in\Theta}B''(\theta) > 0$ and $\sup_{\theta\in\Theta}B''(\theta) < \infty$, the condition is satisfied with $
M = \sqrt{\frac{\inf_{\theta\in\Theta}B''(\theta)}{\sup_{\theta\in\Theta}B''(\theta)}} \in (0,1).
$

Under the local triangle inequality condition, we can prove that the minimax rate of convergence is $\frac{1}{n}$ and it can be achieved by the sample mean. This follows from standard minimax lower bounds \citep[see, e.g.][Theorem 2.2]{Tsybakov:2008} with a slight modification to accommodate for the local triangle inequality condition. The proof can be found in Appendix~\ref{Append::subSec::minimax_optimal}. Note that, for the sub-Gaussian case, the risk reduces to the normalized $\ell_2$ risk we studied in the previous section.

\begin{proposition}\label{prop::minimax_optimality}
	Let $\{X_i\}_{i=1}^n$ be an i.i.d.\ sample from a distribution in a natural exponential family  $\{p_\theta : \theta \in \Theta\}$.	For each $\theta \in \Theta$, let $\mu$ be the mean parameter and $\psi_{\mu}$ is the cumulant generating function corresponding to $\theta$. Then the risk of the sample mean, $\hat{\mu}(n) = \frac{1}{n}\sum_{i=1}^{n} X_i$, is bounded as 
	\begin{equation} \label{eq::upper_bound_sub-psi}
		\mathbb{E}_{P_\theta} \left[n D_{\psi_{\mu}^*}(\hat{\mu}(n), \mu)\right] \leq 2,~~\forall \theta \in \Theta.
	\end{equation}
	Also, if $\sqrt{D_{\psi_\mu^*}}$ satisfies the local triangle inequality condition, then, for a large enough $n$, the minimax risk  is lower bounded as 
	\begin{equation}
		\inf_{\tilde{\mu}}\sup_{\theta \in \Theta} \mathbb{E}_{P_\theta} \left[nD_{\psi_{\mu}^*}(\tilde{\mu}(n), \mu)\right] \geq \frac{M\log2}{16}.
	\end{equation}	 
\end{proposition}

 Proposition~\ref{prop::minimax_optimality} provides the inspiration for the results of the subsequent sections, where we will establish various upper bounds on both normalized and unnormalized versions the Bregman divergence risk under various degrees of adaptivity. Specifically, starting with the simple settings of a fixed target arm at a stopping time, we derive a tight upper bound  on the unnormalized Bregman risk based on the effective sample. Then we move to the fully adaptive setting and show that by inducing a small ``adaptive normalizing factor'' in a log-log scale, we can extend the bound~\eqref{eq::upper_bound_sub-psi} on the normalized risk of the sample mean to the fully adaptive setting.

\subsection{Bregman divergence risk bounds for a fixed target arm at a stopping time} \label{subSec::fixed_target}
 Recall that for each $k \in [\K]$, the \textit{effective sample size} for arm $k$ is defined as $n_k^{\eff} := \left[\mathbb{E}\left[1/N_k(\Tau)\right]\right]^{-1}$. Similarly, for any $r > 1$, the $r$-th order effective sample size is defined as 
 $n_k^{\eff,r} := \left[\mathbb{E}\left[1/N_k^r(\Tau)\right]\right]^{-1/r}$.
   Our next result exhibits a general Bregman risk bound  that depends on the effective sample size. 

\begin{theorem}\label{thm::bound_on_Bregman}
	Consider an adaptive sampling algorithm and stopping rule, and a fixed arm $k$. If there exists a time $t_0$ such that $\Tau \geq t_0$ and $N_k(t_0) \geq b > 0$ almost surely, then the risk of $\hat{\mu}_k(\Tau)$  is bounded as
	\begin{equation} \label{eq::bound_on_Bregman}
		\mathbb{E} \left[D_{\psi_{\mu_k}^*} (\hat{\mu}_k(\Tau), \mu_k)\right]  \leq \min\left\{2e\frac{1 + \log (n_k^{\eff} / b) }{n_k^{\eff}}, \inf_{r>1} \frac{ C_{r}}{n_k^{\eff, r}}\right\},
	\end{equation}
	where, for any $r>1$, 
	\begin{equation}
	    C_{r} := \inf_{q \in (1,r)}\frac{2^{q/r}}{e} \frac{r^2}{(r-q)(q-1)}.
	\end{equation}
	In particular, $C_r \rightarrow \infty$ as $r \to 1$.
\end{theorem}
Note that the bound in \eqref{eq::bound_on_Bregman} is always non-negative since $n_k^{\eff} \geq b$ by assumption. Further, if we always begin by sampling every arm once, then we may take $t_0=\K$ and $b=1$. Of course, if we can choose a larger $b$, then the bound will be stronger.
The proof of Theorem~\ref{thm::bound_on_Bregman} can be found in
Appendix~\ref{subSec::proof_of_bounds_on_Bregman} and is based on the following deviation inequality for the unnormalized Bregman divergence loss, which is proved in Appendix~\ref{subSec::proof_of_deviation_ineq}.
\begin{lemma} \label{lemma::deviation_ineq}
	Under the assumptions in Theorem~\ref{thm::bound_on_Bregman} we have that, for any $\delta \geq 0$,
	\begin{equation} \label{eq::deviation_ineq}
		\mathbb{P}\left(D_{\psi^*_{\mu_k}}(\hat{\mu}_k(\Tau) , \mu_k) \geq \delta \right) 
		\leq  2 \inf_{q \geq 1} \left[ \mathbb{E} e^{-(q-1)\delta N_k(\Tau)}\right]^{1/q} \leq 2 e^{-\delta b}.
	\end{equation}
\end{lemma}

We remark that the results in \cite{caballero1998estimation, pena2008self} imply similar deviation inequalities and moment bounds for sub-Gaussian arms. The bound in Lemma~\ref{lemma::deviation_ineq} can be viewed as a generalization to sub-$\psi$ arms.


Using Jensen's inequality and the convexity of Bregman divergences, it is almost immediate to convert the risk bound \eqref{eq::bound_on_Bregman} into a bound on  the expected $\ell_1$ loss, and on the bias $|\hat{\mu}_k(\Tau) - \mu_k|$. A minor complication arises due to the fact that the function $\psi^*$ is not invertible around $0$. Instead, we consider two invertible variants of $\psi^* $, both defined on $\Lambda^* \cap [0,\infty)$ and taking values in $[0, \infty)$:
\[
z \mapsto \psi_+^* (z) = \psi^*(z) \quad  \text{and} \quad  z \mapsto \psi_{-}^*(z) = \psi^*(-z).
\]

\begin{corollary} \label{cor::bound_on_bias_l1_risk}
	Suppose the assumptions in Theorem~\ref{thm::bound_on_Bregman} hold. For each $k \in [\K]$ and $b >0$, define 
	\begin{equation}
	U_{k,b} := \min\left\{2e\frac{1 + \log (n_k^{\eff} / b) }{n_k^{\eff}}, \inf_{r>1} \frac{ C_{r}}{n_k^{\eff, r}}\right\}.
	\end{equation}
    Then, the bias of the sample mean is bounded as
	\begin{equation} \label{eq::bound_on_bias}
		- {\psi_-^*}^{-1}\left( U_{k,b}\right) \leq	\mathbb{E}\left[\hat{\mu}_k(\Tau)\right]   - \mu_k\leq  {\psi_+^*}^{-1}\left(U_{k,b}\right).
	\end{equation}	
	Furthermore, if $\psi^*$ is symmetric around zero, then the  $\ell_1$ risk can be bounded as 
	\begin{equation} \label{eq::bound_on_l1_risk}
		\mathbb{E}\left|\hat{\mu}_k(\Tau) - \mu_k\right| \leq {\psi_+^*}^{-1}\left( U_{k,b}\right).
	\end{equation}
\end{corollary}

\noindent The proof  is a direct consequence of the Jensen's inequality; see Appendix~\ref{subSec::cor_bound_on_bias_l1_risk}  As one explicit example, if the underlying distribution is sub-Gaussian, $\psi_+^{*-1}(l) = \sigma\sqrt{2l}$ and the $\ell_1$ risk of the sample mean is bounded as
\begin{equation}
	\label{eq::SG-bias-bound}
	\mathbb{E}\left|\hat{\mu}_k(\Tau) - \mu_k\right| \leq \sigma\sqrt{2U_{k,b}}= \sigma \min\left\{\sqrt{4e\frac{1 + \log (n_k^{\eff} / b) }{n_k^{\eff}}}, \inf_{r>1} \sqrt{\frac{ 2C_r}{n_k^{\eff, r}}} \right\}.
\end{equation}
 We remark that the above bound on the bias is not improvable beyond the log factor in general by using the following stopped Brownian motion example \cite[Ch. 3]{siegmund1978estimation}.
\begin{example}\label{eg::brownian-stopping}
Let $W(t)$ be a Wiener process. If we define a stopping time as the first time $W(t)$ exceeds a line with slope $\eta$ and intercept $b>0$, that is $\Tau_B :=\inf\{t\geq 0: W(t) \geq \eta t + b \}$, then for any slope $\eta \leq \mu$, we have $\mathbb{E}\left[\frac{W(\Tau_B)}{\Tau_B} - \mu\right] = 1/b$.
\end{example}
 
 Note that a sum of Gaussians with mean $\mu$ behaves like a time-discretization of a Brownian motion with drift $\mu$; since $\mathbb{E}W(t) = t\mu$, we may interpret $W(\Tau_B)/\Tau_B$ as a stopped sample mean, and the last equation implies that its bias is $1/b$ for any  slope $\eta \leq \mu$. In particular, if we set $\eta = \mu$, it is easy to deduce that $\E[1/\Tau_B]=1/b^2$ and thus that
\[
1 / n^\eff = \mathbb{E}\left[1  / \Tau_B\right] = 1/ b^2.
\]
As a result, the bias of $W_{\Tau_B}/ \Tau_B$ as a stopped sample mean is exactly equal to $\sqrt{1/n^\eff}$ which matches \eqref{eq::SG-bias-bound} up to a log factor.

We end this section by discussing whether tight risk bounds can be obtained based on $\mathbb{E}N_k(\Tau)$. By Jensen's inequality, it can be easily checked that $n_k^\eff \leq \mathbb{E}N_k(\Tau)$. One may wonder if it is possible to obtain tighter  bounds on both the bias and the risk that scale with $1/\mathbb{E}N_k(\Tau)$ instead of $n_k^\eff$. However, it can be easily seen that this is not possible in general. For instance, in the previous stopped Brownian motion case with $\eta = \mu$, the bias is equal to $1 / b = 1/\sqrt{n^\eff} >0$. However, under the same setting, it is well-known that $\mathbb{E}N_k(\Tau)=\infty$. Therefore, the bias (namely $1/b$) can never be bounded by $1/\mathbb{E}N(\Tau)=0$. Also, a risk bound in terms of $1/\mathbb{E}N_k(\Tau)$ would imply consistency whenever $\mathbb{E}N_k(\Tau) \to \infty$, but Example~\ref{eg::counter_eg-consistency} shows that $\hat{\mu}_k$ can be inconsistent even when $\mathbb{E}N_k(\Tau) \to \infty$.

\subsection{Bregman divergence risk bounds under fully adaptive settings}

Let $I(\kappa ; \D_{\Tau})$ be the mutual information between $\kappa$ and the dataset $\D_{\Tau}$.
When the dataset $\D_{\Tau}$ is collected in a deterministic manner, \cite{russo2016controlling} showed how to bound the bias and expected $\ell_1$ and $\ell_2$ loss of adaptively chosen centered sub-Gaussian random variables by using $I(\kappa ; \D_{\Tau})$. In particular, if each $\widehat{\mu}_k - \mu_k$ has mean zero and is $(\sigma / \sqrt{n})$-sub-Gaussian, then \cite{russo2016controlling} proved that 
\begin{align} 
	\sqrt{n}\left| \mathbb{E}\widehat{\mu}_\kappa - \mu_\kappa  \right| &\leq \sigma \sqrt{2I(\kappa ; \D_{\Tau})},  \label{eq::Russo_bound_on_bias}\\
	\mathbb{E}\left[\sqrt{n}\left| \widehat{\mu}_\kappa - \mu_\kappa  \right|\right] &\leq \sigma\left(1 +c_1\sqrt{2I(\kappa ; \D_{\Tau})}\right)    \label{eq::Russo_bound_on_l1}\\
	\mathbb{E}\left[n\left( \widehat{\mu}_\kappa - \mu_\kappa  \right)^2\right] &\leq \sigma^2\left(1.25  +c_2 I(\kappa ; \D_{\Tau})\right) , \label{eq::Russo_bound_on_l2}
\end{align}
where $c_1 <36$ and $c_2 \leq 10$ are universal constants. 

For the reasons discussed in Section~\ref{sec::risk_under_finite_moment}, however, these bounds are not directly applicable to the fully adaptive setting since each $\hat{\mu}_k -\mu_k$ is no longer centered, due to the bias caused by adaptive sampling, stopping and rewinding. In particular, the bound for the bias given in equation \eqref{eq::Russo_bound_on_bias}  no longer holds under the fully adaptive setting because the bias can be non-zero even if $\kappa$ is independent of $\D_{\Tau}$. 

In this subsection we show that, by introducing an additional small ``penalty for adaptivity'', measured on the log-log scale, the bounds for deterministic and nonadaptive sampling and stopping can be basically extended to the fully adaptive setting.
Towards that end, and assuming that $N_k(t_0) > 3$ for all $k \in [\K]$, we set
\begin{equation}
    \dbtilde{N}_k(t) := \frac{N_k(t)}{\log\log N_k(t)},~~\forall k \in [\K], \forall t \geq t_0.
\end{equation}
We now present the main result of this section.

\begin{theorem} \label{thm::bounds_in_fully_adap}
    For any adaptive sampling and stopping rule and any adaptively chosen arm $\kappa$, the  risk of $\hat{\mu}_\kappa(\tau)$ at any adaptively rewound time $\tau \leq \Tau$ such that, for some deterministic time $t_0>0$ and a constant $b \geq 3$, $\tau\geq t_0$ and $N_k(t_0) \geq b$  almost surely, it holds that
	\begin{align}\label{eq::adap_risk}
	\mathbb{E} \left[\dbtilde{N}_\kappa(\tau) D_{\psi^*_{\mu_{\kappa}}}\left(\hat{\mu}_{\kappa}(\tau), \mu_{\kappa} \right) \right]
	&\leq C_b \left[I\left(\kappa;\D_{\Tau}\right) + 1.25\right],
	\end{align}
	where $C_b := 4e\left(1 + \frac{1}{\log\log b}\right)$.
\end{theorem}

Note that for the sub-Gaussian case, the inequality~\eqref{eq::adap_risk} is reduced to the following bound on the normalized $\ell_2$ risk.
\begin{equation}\label{eq::adap_risk_l2}
	\mathbb{E} \left[\dbtilde{N}_\kappa(\tau) \left(\hat{\mu}_{\kappa}(\tau)-\mu_{\kappa} \right)^2 \right]
	\leq 2 C_b \sigma^2 \left[I\left(\kappa;\D_{\Tau}\right) + 1.25\right].
\end{equation}
By comparing the above bound with the bound~\eqref{eq::Russo_bound_on_l2} of  \cite{russo2016controlling}, we can notice that our bound \eqref{eq::adap_risk} under the fully adaptive setting only suffers a multiplicative normalization term which is of order $\log\log N_{\kappa}(\tau)$. Also, for a fixed target, the following example demonstrates that, in general, the bound~\eqref{eq::adap_risk} cannot be improved upon, aside from constants. 

\begin{example}[Example~\ref{eg::lil_stopping} revisited] \label{eg::lil_stopping_revisit}
In the same setting of Example~\ref{eg::lil_stopping}, we further assume that each arm has a normal distribution variance $\sigma_k^2$. Then, from the definition of the stopping time $\Tau_b$ and the bound~\eqref{eq::adap_risk}, we obtain the following upper and lower bounds on the normalized $\ell_2$ risk for a fixed target:
\begin{equation} \label{eq::upper_lower_bound}
  \sigma_k^2 \leq \mathbb{E} \left[\dbtilde{N}_k(\Tau_b) \left(\hat{\mu}_{k}(\Tau_b)- \mu_{k} \right)^2 \right] \leq 2.5 C_{b} \sigma_k^2,
\end{equation}    
where the upper and lower bounds match up to a constant factor. 
\end{example}

The proof of Theorem~\ref{thm::bounds_in_fully_adap} in Appendix~\ref{subSec::Proof_of_second_them} relies on the following deviation inequality for the normalized Bregman divergence loss, along with the Donsker-Varadhan variational representation of the KL divergence. 

\begin{lemma} \label{lemma::adaptive_deviation_ineq}
	Consider an adaptive sampling algorithm and stopping rule. For a fixed $k \in [\K]$ and a random time $\tau$, assume $N_k(\tau) \geq b$ almost surely. Then, for any $\delta \geq 1$, 
	\begin{equation} \label{eq::adaptive_deviation_ineq}
	\mathbb{P}\left(\dbtilde{N}_k(\tau)D_{\psi^*_{\mu_k}}(\hat{\mu}_k(\tau) , \mu_k) \geq C_b \delta\right)
	\leq 2\exp\left\{-\delta\right\}.
	\end{equation}
\end{lemma}
The proof of the lemma is deferred to Appendix~\ref{subSec::proof_of_adaptive_ineq}. 
Similar inequalities have been developed in the context of always valid confidence sequences or finite-LIL bounds. Except in the sub-Gaussian case, the existing inequalities cannot be directly converted into bounds on the Bregman divergence. Recently, \cite{garivier2013informational} provided concentration inequalities for the KL loss, and \cite{kaufmann2018mixture} derived similar inequalities for the additive KL loss across several arms. However, their bounds depend on $\delta$ in a complicated way, making it difficult to develop bounds for the risk. In contrast, the bound in \eqref{eq::adaptive_deviation_ineq} is linear in $\delta$.



Next, for any $r >0$, define the $r$-th order iterated logarithmically discounted effective sample size of an adaptively chosen arm as 
\begin{equation}
\dbtilde{n}^{\eff,r}_\kappa := \left[\mathbb{E}\left[1/\dbtilde{N}_\kappa^r (\tau)\right]\right]^{-1/r},~~ \forall r >0,
\end{equation}
where the expectation is over the randomness in all four sources of adaptivity. We emphasize that the above quantity is nonrandom; the subscript $\kappa$ merely differentiates it from the effective sample size of a fixed arm and is not to be interpreted as residual randomness. One can  easily check that $\dbtilde{n}^{\eff,r}_\kappa$ is decreasing with respect to $r$ by Jensen's inequality.
The following corollary shows how to control the risk using $\dbtilde{n}^{\eff, r}_\kappa$. The proof can be found in Appendix~\ref{subSec::proof_of_bound_in_fully_adap}.
\begin{corollary}\label{cor::bounds_in_fully_adap}
	For any $r \in (0,1)$, the $r$-quasi-norm of the divergence can be bounded as
	\begin{equation}\label{eq::r-norm_bound}
	\left[\mathbb{E} D_{\psi^*_{\mu_{\kappa}}}^{r}\left(\hat{\mu}_{\kappa}(\tau), \mu_{\kappa} \right)\right]^{1/r}
	\leq \frac{C_b}{\dbtilde{n}^{\eff, r/(1-r)}_{\kappa}} \left[I\left(\kappa;\D_{\Tau}\right) + 1.25 \right] .
	\end{equation}
\end{corollary}
\noindent In the sub-Gaussian setting, by choosing $r = 1/2$, the above results results in the bound for the $\ell_1$ risk
\begin{equation}\label{eq::1/2-norm_bound_sub_G}
\mathbb{E} \left|\hat{\mu}_{\kappa}(\tau)- \mu_{\kappa} \right|
\leq \frac{\sigma}{\sqrt{\dbtilde{n}^{\eff}}}\sqrt{2C_b\left[I\left(\kappa;\D_{\Tau}\right) + 1.25\right]} ,
\end{equation}
which is also comparable with the  bound~\eqref{eq::Russo_bound_on_l1} on the $\ell_1$ risk bound given by  \cite{russo2016controlling},
\begin{equation}
\mathbb{E} \left|\hat{\mu}_\kappa(n) - \mu_\kappa  \right| \leq \frac{\sigma}{\sqrt{n}}\left(c_1 \sqrt{2I(\kappa ; \D_{\Tau})} + 1 \right).
\end{equation}

We quickly point out that \Cref{thm::bounds_in_fully_adap} and its corollary immediately yield results for the setting where we adaptively rewind to time $\tau$, but choose a fixed arm $\kappa=k$, since $I(\kappa,\D_{\Tau})=0$ in this case. Under a deterministic setting where a fixed number of observations are drawn from a fixed arm, the fully adaptive bounds on the normalized risk~\eqref{eq::adap_risk} of this subsection are strictly weaker than the ones on the unnormalized risk~\eqref{eq::bound_on_Bregman} derived in the previous subsection, because the former bounds contain the additional log-log factor. In general, however, these two bounds are not comparable. Finally, we also remark that by letting $r \to 1$, we get $\dbtilde{n}^{\eff, r/(1-r)}_{\kappa} \to b / \log\log b$ which implies the following bound on the risk:
\begin{equation}
  \mathbb{E} D_{\psi^*_{\mu_{\kappa}}}\left(\hat{\mu}_{\kappa}(\tau), \mu_{\kappa} \right)
	\leq C_b\frac{\log\log b}{b} \left[I\left(\kappa;\D_{\Tau}\right) + 1.25\right] .
\end{equation}
It is an open question whether it is possible to get a bound based on $\dbtilde{n}^\eff$ instead of $b$ in the fully adaptive setting. 

\section{Summary of the main theorems and proof techniques} \label{sec::proof_sketch}
We have analyzed the properties of the sample mean estimator under four types of adaptivity implied by arbitrary rules for sampling, stopping, choosing and rewinding. Table~\ref{tab::summary} summarizes the risk bounds we have derived under different conditions on the distribution of the arms and under different data collection / analysis procedures. 

\begin{table}[ht]
\centering
\setlength\tabcolsep{1.5mm}
\captionsetup{labelfont = bf, format =hang}
\def\arraystretch{2}
\caption{Summary of normalized and unnormalized risk bounds under different conditions. \newline Recall that $\tilde{N}_\kappa := N_{\kappa} / \log N_{\kappa}$ and $\protect\dbtilde{N}_\kappa := N_{\kappa} / \log\log N_{\kappa}$.}
\label{tab::summary}
\begin{tabular}{cclr}
\hline
Tail condition & Data collection  &  \multicolumn{1}{c}{Risk bound} & \\ \hline\hline
$\sigma_k^{(2)} < \infty$  & Nonadaptive      &   $\mathbb{E}\left[N_k(T)\left(\hat{\mu}_k(T) - \mu_k\right)^2\right] = \sigma_k^2$ & (Prop~\ref{prop::minimax_nonadaptive}) \\ \addlinespace[2mm] \hline
\addlinespace[2mm]
$\max_k \sigma_k^{(2p)} < \infty$   & Adaptive choosing &         $\begin{aligned}\mathbb{E}&\left[N_\kappa(T)\left(\hat{\mu}_\kappa(T) - \mu_\kappa\right)^2\right] \\&\leq \left\|\sigma_\kappa\right\|_2^2 + C_p\left\|\sigma_\kappa^{(2p)}\right\|_{2p}^2 I_q^{1/q} (\kappa, \D_{\Tau})\end{aligned}$ & (Prop~\ref{prop::bounds_l2_finite_moment})  \\ \addlinespace[2mm] \hline
\addlinespace[2mm]
$\max_k \sigma_k^{(2(p+\epsilon))} < \infty$  & Fully adaptive  &  $	\begin{aligned}\mathbb{E}&\left[\tilde{N}_\kappa (\tau) \left(\hat{\mu}_{\kappa}(\tau)- \mu_{\kappa} \right)^2 \right] \\
	&\leq C_{1,\epsilon}\|\sigma_\kappa\|_2^2 + C_{p,\epsilon}\|\sigma_\kappa\|_{2p}^2  I_q^{1/q}\left(\kappa, \D_{\Tau}\right)\end{aligned}$   & (Thm~\ref{thm::bounds_in_fully_adap_finite_moment})    \\\addlinespace[2mm] \hline
	\addlinespace[2mm]
$\E[e^{\lambda (Y - \mu)}] \leq e^{\psi(\lambda)} $     & \makecell{Adaptive sampling \\ and stopping}  &  
$	\begin{aligned}\mathbb{E}&\left[ D_{\psi_{\mu_k}^*} (\hat{\mu}_k(\Tau), \mu_k)\right] \\ &\leq \min\left\{2e\frac{1 + \log (n_k^{\eff} / b) }{n_k^{\eff}}, \inf_{r>1} \frac{ C_{r}}{n_k^{\eff, r}}\right\}\end{aligned}$& (Thm~\ref{thm::bound_on_Bregman})\\\addlinespace[2mm] \hline
\addlinespace[2mm]
$\E[e^{\lambda (Y - \mu)}] \leq e^{\psi(\lambda)} $        &  Fully adaptive     &   $\begin{aligned}\mathbb{E}& \left[\dbtilde{N}_\kappa(\tau) D_{\psi^*_{\mu_{\kappa}}}\left(\hat{\mu}_{\kappa}(\tau), \mu_{\kappa} \right) \right]\\
	&\leq C_b \left[I\left(\kappa;\D_{\Tau}\right) + 1.25\right] \end{aligned}$ & (Thm~\ref{thm::bounds_in_fully_adap})  \\ \addlinespace[2mm] \hline
\end{tabular}
\end{table}

The derivation of the upper bounds for the various notions of the risk of the chosen mean are based on the variational representations of the $f_q$-divergence (Theorem~\ref{thm::bounds_in_fully_adap_finite_moment}) and of the KL divergence (Theorem~\ref{thm::bounds_in_fully_adap}). These take the form of
\begin{equation}
	\frac{1}{q}D_{f_q} (P || Q)  = \sup_{f \in \mathcal{C}_p}  \mathbb{E}_P \left[f(X)\right] - \mathbb{E}_Q \left[f(X)\right] - \mathbb{E}_Q \left[\frac{\left|f(X)\right|^p}{p}\right], \label{eq::phi_q_rep}
	\end{equation}
	and
	\begin{equation}
	D_{KL} (P || Q)  = \sup_{f \in \mathcal{C}_{\exp}} \mathbb{E}_P \left[f(X)\right] - \log \mathbb{E}_Q \left[e^{f(X)}\right], \label{eq::KL_req}
\end{equation}
respectively, where $P$, $Q$ are probability measures on $\mathcal{X}$. In  the first equation~\eqref{eq::phi_q_rep}, $\mathcal{C}_p$ denotes the set of measurable functions $f \colon \mathcal{X} \mapsto \mathbb{R}$ such that $\mathbb{E}_Q \left|f(X)\right|^p < \infty$, with  $p, q > 1$ satisfying $1/p + 1/q = 1$, while in the second equation~\eqref{eq::KL_req}, $\mathcal{C}_{\exp}$ is the set of measurable functions $f : \mathcal{X} \mapsto \mathbb{R}$ such that $\mathbb{E}_Q \left[e^{f(X)}\right] < \infty$. 

Now, for each $k \in [\K]$, set $ P_k = \mathcal{L}\left(\D_{\Tau} | {\kappa} = k\right)$, $Q = \mathcal{L}\left(\D_{\Tau}\right)$ and 
	\begin{equation}
	f_k = \begin{cases}
	\lambda\tilde{N}_{k}(\tau)\left(\hat{\mu}_k(\tau)-\mu_k \right)^{2} &\text{ (Theorem~\ref{thm::bounds_in_fully_adap_finite_moment}),} \\
	\lambda\dbtilde{N}_{k}(\tau)  D_{\psi_{\mu_k}^*}(\hat{\mu}_k, \mu_k)&\text{ (Theorem~\ref{thm::bounds_in_fully_adap})}. 
	\end{cases}
	\end{equation}
By plugging these choices of $P_k, Q$ and $f_k$ in the right hand sides of \eqref{eq::phi_q_rep} and \eqref{eq::KL_req}, we obtain lower bounds on the $f_q$ divergence and the KL divergence between the conditional and unconditional laws of the data. Next, based on these lower bounds, we  derive upper bounds on the normalized risk of the chosen mean in Theorem~\ref{thm::bounds_in_fully_adap_finite_moment} and \ref{thm::bounds_in_fully_adap}. The details of our derivations can be found in Appendix~\ref{subSec::Proof_of_second_them_finite_moment} (Theorem~\ref{thm::bounds_in_fully_adap_finite_moment}) and Appendix~\ref{subSec::Proof_of_second_them} (Theorem~\ref{thm::bounds_in_fully_adap}).


This style of proof  was originally developed by \cite{russo2016controlling} and \cite{jiao2017dependence} for the fixed sample size setting. The main technical hurdle to extend it to the fully adaptive setting is to find tight upper bounds on the expectations of the $p$th power and the exponential moment of the normalized losses,  defined for each $k \in [\K]$ by
\begin{align*}
     &\mathbb{E}\left[\left|\lambda\tilde{N}_k(\tau)\left(\hat{\mu}_k(\tau)-\mu_k \right)^{2}\right|^p\right]~&\text{(Theorem~\ref{thm::bounds_in_fully_adap_finite_moment})},\\	&\mathbb{E}\left[\exp\left\{\lambda\dbtilde{N}_{k}(\tau)  D_{\psi_{\mu_k}^*}(\hat{\mu}_k, \mu_k)  \right\}\right] ~&\text{(Theorem~\ref{thm::bounds_in_fully_adap})},
\end{align*}
where $\lambda> 0$ is a parameter to be chosen appropriately. To derive upper bounds independent of the sampling algorithms and of the stopping rules, we use the deviation inequalities in Lemma~\ref{lemma::adaptive_deviation_ineq_finite_moment} and \ref{lemma::adaptive_deviation_ineq} in conjunction with the following facts:
\begin{align*}
\mathbb{E}|X|^p &\leq 1+ \int_{1}^{\infty}\mathbb{P}(|X| > \delta^{1/p}) \mathrm{d}\delta,\\
\mathbb{E}\left[e^X\right] &\leq e + \int_{1}^{\infty}\mathbb{P}(X > \delta) e^\delta \mathrm{d}\delta.
\end{align*}

In the proofs of Lemma~\ref{lemma::adaptive_deviation_ineq_finite_moment} and \ref{lemma::adaptive_deviation_ineq} we deploy  martingale inequalities to obtain high probability bounds on events where the running sum  $\{S_k(t)\}$ eventually exceeds certain linear functions of the number of draws $\{N_k(t)\}$. Specifically, in Lemma~\ref{lemma::adaptive_deviation_ineq_finite_moment} we use the $\ell_p$-version of the Dubins-Savage inequality \cite{khan2009p}, while in Lemma~\ref{lemma::adaptive_deviation_ineq} our arguments are directly inspired by the proof of the exponential line-crossing inequality of \cite{howard2020time}.

The derivation of the bound in Theorem~\ref{thm::bound_on_Bregman} is based on the deviation inequality for unnormalized loss in Lemma~\ref{lemma::deviation_ineq} and the fact that $\mathbb{E}|X| = \epsilon + \int_{\epsilon}^\infty \mathbb{P}(|X| > \delta) \mathrm{d}\delta$ for any choice of $\epsilon \geq 0$; utilizing both, we have the following 
intermediate bound:
\begin{equation} \label{eq::intermediate}
   \mathbb{E}\left[ D_{\psi_{\mu_k}^*} (\hat{\mu}_k(\Tau), \mu_k)\right] \leq \epsilon + 2  \int_{\epsilon}^{\infty}\left[ \mathbb{E} \exp\left\{-\frac{q}{p} \delta N_k(\Tau)  \right\}\right]^{1/q} \mathrm{d} \delta,
\end{equation}
where $\epsilon\geq0$ and $p, q >1$ with  $1/p + 1/q = 1$. By carefully choosing  $\epsilon, p$ and $q$ we then arrive at the final bounds in terms of effective sample sizes.  The proof the deviation inequality in Lemma~\ref{lemma::deviation_ineq} is based on the following process:
\[
\left\{\exp\left\{ \lambda\left( S_k(t)- \mu_k N_k(t) \right) - N_k(t)\psi(\lambda)  \right\}\right\}_{t \geq 0},
\]
which is supermartingale with respect to the filtration $\{\mathcal{F}_t\}_{t\geq0}$, for any fixed $\lambda \in \Lambda$.

        As a a final remark on the tightness of the risk bounds in \cref{tab::summary}, we point out that, except for the bounds presented in Section~\ref{subSec::moment_full_adap} (2nd and 3rd rows in \cref{tab::summary}), all the other risk bounds cannot be improved in general, as we have demonstrated in \cref{prop::minimax_nonadaptive}, \cref{eg::brownian-stopping} and \cref{eg::lil_stopping_revisit}. At the present time, it is unclear to us whether the log factor in the expression \eqref{eq:normalized-samplesize-log} for the discounted sample size for the  bounds in Section~\ref{subSec::moment_full_adap} can be sharpened without self-normalization (see the remark following \Cref{lemma::adaptive_deviation_ineq_finite_moment}.)




\section{Discussion and future work}\label{sec::disc}
 We build on a line of interesting work that considered one type of adaptivity at a time. For example, the important work of  \cite{russo2016controlling} and its extensions by \cite{jiao2017dependence} can be viewed as understanding the bias of the sample mean under nonadaptive sampling, nonadaptive stopping and adaptive choosing. Similarly, the work by \cite{nie2018adaptively} can be seen as providing a qualitative understanding of the sample mean  under ``optimistic'' adaptive sampling, but for a deterministic arm stopped at a deterministic time. Further, while these past works have primarily focused on the bias, our work answers natural questions involving the estimation risk and consistency. 


Several interesting questions remain for future research. The first one revolves around the choice of loss function for calculating the risk. Arguably, we picked the most natural loss function, which is the $\ell_2$ loss for heavy-tailed arms and the Bregman divergence with respect to the convex conjugate of the CGF, which reduces to the KL-loss for exponential families. However, it is likely that the bounds achieved as implications of our results are not tight for other loss functions, and different techniques may be required. A second, related, question involves proving minimax lower bounds for risk (for various loss functions) under all kinds of adaptivity. The work of \cite{sackrowitz1986evaluating} on Bayes and minimax approaches towards evaluating a selected population may be a relevant starting point. 

A final question revolves around possibly moving away from the sample mean, and in particular exploring whether there exist generic methods to either (a) alter the process of collecting the data to produce an unbiased estimator of the mean, or (b) to debias the sample mean post hoc given explicit knowledge of the exact sampling, stopping and choosing rule used. For aim (b), sample splitting was proposed by \cite{xu2013estimation}, techniques from conditional inference were suggested by \cite{nie2018adaptively}, and a ``one-step'' estimator was suggested by \cite{deshpande2017accurate}. However, all three methods
seem to account for adaptive sampling, but not adaptive stopping or choosing, but their techniques seem to provide a good starting point.  More recently, ideas from differential privacy were exploited by \cite{neel2018mitigating} for aim (a). It remains unclear what the theoretical and practical tradeoffs are between these methods, and how much they improve on the risk of the sample mean in a nonparametric and nonasymptotic sense in the fully adaptive setting.


Overall, we anticipate much progress on the above and other related questions in future years, due to the pressing concerns raised by the need to perform statistical inference on data collected via adaptive schemes that are common in the tech industry.

		\bibliographystyle{plainnat}
	\bibliography{risk_consistency}

\begin{thebibliography}{48}
\providecommand{\natexlab}[1]{#1}
\providecommand{\url}[1]{\texttt{#1}}
\expandafter\ifx\csname urlstyle\endcsname\relax
  \providecommand{\doi}[1]{doi: #1}\else
  \providecommand{\doi}{doi: \begingroup \urlstyle{rm}\Url}\fi

\bibitem[Anscombe(1952)]{anscombe1952large}
Francis~J Anscombe.
\newblock Large-sample theory of sequential estimation.
\newblock In \emph{Mathematical Proceedings of the Cambridge Philosophical
  Society}, volume~48, pages 600--607. Cambridge University Press, 1952.

\bibitem[Azoury and Warmuth(2001)]{azoury2001relative}
Katy~S Azoury and Manfred~K Warmuth.
\newblock Relative loss bounds for on-line density estimation with the
  exponential family of distributions.
\newblock \emph{Machine Learning}, 43\penalty0 (3):\penalty0 211--246, 2001.

\bibitem[Balsubramani(2014)]{balsubramani_sharp_2014}
Akshay Balsubramani.
\newblock Sharp finite-time iterated-logarithm martingale concentration.
\newblock \emph{arXiv preprint arXiv:1405.2639}, 2014.

\bibitem[Balsubramani and Ramdas(2016)]{balsubramani_sequential_2016}
Akshay Balsubramani and Aaditya Ramdas.
\newblock Sequential nonparametric testing with the law of the iterated
  logarithm.
\newblock In \emph{Proceedings of the {Thirty}-{Second} {Conference} on
  {Uncertainty} in {Artificial} {Intelligence}}, pages 42--51, 2016.

\bibitem[Bowden and Trippa(2017)]{bowden2017unbiased}
Jack Bowden and Lorenzo Trippa.
\newblock Unbiased estimation for response adaptive clinical trials.
\newblock \emph{Statistical methods in medical research}, 26\penalty0
  (5):\penalty0 2376--2388, 2017.

\bibitem[Caballero et~al.(1998)Caballero, Fern{\'a}ndez, and
  Nualart]{caballero1998estimation}
Mar{\'\i}a~Emilia Caballero, Bego{\~n}a Fern{\'a}ndez, and David Nualart.
\newblock Estimation of densities and applications.
\newblock \emph{Journal of Theoretical Probability}, 11\penalty0 (3):\penalty0
  831--851, 1998.

\bibitem[Cox(1952)]{cox1952note}
DR~Cox.
\newblock A note on the sequential estimation of means.
\newblock In \emph{Mathematical Proceedings of the Cambridge Philosophical
  Society}, volume~48, pages 447--450. Cambridge University Press, 1952.

\bibitem[Darling and Robbins(1967{\natexlab{a}})]{darling_confidence_1967}
D.~A. Darling and Herbert Robbins.
\newblock Confidence {Sequences} for {Mean}, {Variance}, and {Median}.
\newblock \emph{Proceedings of the National Academy of Sciences}, 58\penalty0
  (1):\penalty0 66--68, 1967{\natexlab{a}}.

\bibitem[Darling and Robbins(1967{\natexlab{b}})]{darling_inequalities_1967}
D.~A. Darling and Herbert Robbins.
\newblock Inequalities for the {Sequence} of {Sample} {Means}.
\newblock \emph{Proceedings of the National Academy of Sciences}, 57\penalty0
  (6):\penalty0 1577--1580, 1967{\natexlab{b}}.

\bibitem[Darling and Robbins(1968)]{darling_further_1968}
D.~A. Darling and Herbert Robbins.
\newblock Some {Further} {Remarks} on {Inequalities} for {Sample} {Sums}.
\newblock \emph{Proceedings of the National Academy of Sciences}, 60\penalty0
  (4):\penalty0 1175--1182, 1968.

\bibitem[de~la Pe{\~n}a et~al.(2004)de~la Pe{\~n}a, Klass, and Lai]{de2004self}
Victor~H de~la Pe{\~n}a, Michael~J Klass, and Tze~Leung Lai.
\newblock Self-normalized processes: exponential inequalities, moment bounds
  and iterated logarithm laws.
\newblock \emph{Annals of probability}, pages 1902--1933, 2004.

\bibitem[de~la Pe{\~n}a et~al.(2008)de~la Pe{\~n}a, Lai, and
  Shao]{pena2008self}
Victor~H de~la Pe{\~n}a, Tze~Leung Lai, and Qi-Man Shao.
\newblock \emph{Self-normalized processes: Limit theory and Statistical
  Applications}.
\newblock Springer Science \& Business Media, 2008.

\bibitem[Deshpande et~al.(2018)Deshpande, Mackey, Syrgkanis, and
  Taddy]{deshpande2017accurate}
Yash Deshpande, Lester Mackey, Vasilis Syrgkanis, and Matt Taddy.
\newblock Accurate inference for adaptive linear models.
\newblock In \emph{Proceedings of the 35th International Conference on Machine
  Learning}, 2018.

\bibitem[Donsker and Varadhan(1983)]{donsker1983asymptotic}
Monroe~D Donsker and SR~Srinivasa Varadhan.
\newblock Asymptotic evaluation of certain markov process expectations for
  large time. iv.
\newblock \emph{Communications on Pure and Applied Mathematics}, 36\penalty0
  (2):\penalty0 183--212, 1983.

\bibitem[Doob(1953)]{doob:90}
J.~L. Doob.
\newblock \emph{Stochatsic Processes}.
\newblock Wiley, 1953.

\bibitem[Durrett(2019)]{durrett2019probability}
Rick Durrett.
\newblock \emph{Probability: theory and examples}, volume~49.
\newblock Cambridge university press, 2019.

\bibitem[Garivier(2013)]{garivier2013informational}
Aur{\'e}lien Garivier.
\newblock Informational confidence bounds for self-normalized averages and
  applications.
\newblock In \emph{2013 IEEE Information Theory Workshop (ITW)}, pages 1--5.
  IEEE, 2013.

\bibitem[Garivier and Capp{\'e}(2011)]{garivier2011kl}
Aur{\'e}lien Garivier and Olivier Capp{\'e}.
\newblock The {KL-UCB} algorithm for bounded stochastic bandits and beyond.
\newblock In \emph{Proceedings of the 24th Annual Conference On Learning
  Theory}, pages 359--376, 2011.

\bibitem[Garivier et~al.(2018)Garivier, Hadiji, Menard, and
  Stoltz]{garivier2018kl}
Aur{\'e}lien Garivier, H{\'e}di Hadiji, Pierre Menard, and Gilles Stoltz.
\newblock Kl-ucb-switch: optimal regret bounds for stochastic bandits from both
  a distribution-dependent and a distribution-free viewpoints.
\newblock \emph{arXiv preprint arXiv:1805.05071}, 2018.

\bibitem[Gut(2009)]{gut2009stopped}
Allan Gut.
\newblock \emph{Stopped random walks}.
\newblock Springer, 2009.

\bibitem[Howard et~al.(2020)Howard, Ramdas, McAuliffe, and
  Sekhon]{howard2020time}
Steven~R Howard, Aaditya Ramdas, Jon McAuliffe, and Jasjeet Sekhon.
\newblock Time-uniform chernoff bounds via nonnegative supermartingales.
\newblock \emph{Probability Surveys}, 17:\penalty0 257--317, 2020.

\bibitem[Howard et~al.(2021)Howard, Ramdas, McAuliffe, and
  Sekhon]{howard2021time}
Steven~R Howard, Aaditya Ramdas, Jon McAuliffe, and Jasjeet Sekhon.
\newblock Time-uniform, nonparametric, nonasymptotic confidence sequences.
\newblock \emph{The Annals of Statistics}, 49\penalty0 (2):\penalty0
  1055--1080, 2021.

\bibitem[Jamieson et~al.(2014)Jamieson, Malloy, Nowak, and
  Bubeck]{jamieson_lil_2014}
Kevin Jamieson, Matthew Malloy, Robert Nowak, and Sébastien Bubeck.
\newblock lil' {UCB}: {An} {Optimal} {Exploration} {Algorithm} for
  {Multi}-{Armed} {Bandits}.
\newblock In \emph{Proceedings of {The} 27th {Conference} on {Learning}
  {Theory}}, volume~35 of \emph{Proceedings of {Machine} {Learning}
  {Research}}, pages 423--439, 2014.

\bibitem[Jiao et~al.(2017)Jiao, Han, and Weissman]{jiao2017dependence}
Jiantao Jiao, Yanjun Han, and Tsachy Weissman.
\newblock Dependence measures bounding the exploration bias for general
  measurements.
\newblock In \emph{IEEE International Symposium on Information Theory (ISIT)},
  pages 1475--1479, 2017.

\bibitem[Kaufmann and Koolen(2018)]{kaufmann2018mixture}
Emilie Kaufmann and Wouter Koolen.
\newblock Mixture martingales revisited with applications to sequential tests
  and confidence intervals.
\newblock \emph{arXiv preprint arXiv:1811.11419}, 2018.

\bibitem[Kaufmann et~al.(2016)Kaufmann, Capp{\'e}, and
  Garivier]{kaufmann_complexity_2014}
Emilie Kaufmann, Olivier Capp{\'e}, and Aur{\'e}lien Garivier.
\newblock On the complexity of best-arm identification in multi-armed bandit
  models.
\newblock \emph{The Journal of Machine Learning Research}, 17\penalty0
  (1):\penalty0 1--42, 2016.

\bibitem[Khan(2009)]{khan2009p}
Rasul~A Khan.
\newblock {$L_p$}-version of the {Dubins--Savage} inequality and some
  exponential inequalities.
\newblock \emph{Journal of Theoretical Probability}, 22\penalty0 (2):\penalty0
  348, 2009.

\bibitem[Lai(1976)]{lai_confidence_1976}
Tze~Leung Lai.
\newblock On {Confidence} {Sequences}.
\newblock \emph{The Annals of Statistics}, 4\penalty0 (2):\penalty0 265--280,
  1976.

\bibitem[Lattimore and Szepesv{\'a}ri(2019)]{lattimore2018bandit}
Tor Lattimore and Csaba Szepesv{\'a}ri.
\newblock Bandit algorithms.
\newblock \emph{Cambridge University Press}, 2019.

\bibitem[Li et~al.(2015)Li, Munos, and Szepesvari]{li2015toward}
Lihong Li, Remi Munos, and Csaba Szepesvari.
\newblock Toward minimax off-policy value estimation.
\newblock In \emph{Artificial Intelligence and Statistics}, pages 608--616,
  2015.

\bibitem[Marcinkiewicz and Zygmund(1937)]{marcinkiewicz1937fonctions}
J{\'o}zef Marcinkiewicz and Antoni Zygmund.
\newblock Sur les fonctions ind{\'e}pendantes.
\newblock \emph{Fundamenta Mathematicae}, 29\penalty0 (1):\penalty0 60--90,
  1937.

\bibitem[Neel and Roth(2018)]{neel2018mitigating}
Seth Neel and Aaron Roth.
\newblock Mitigating bias in adaptive data gathering via differential privacy.
\newblock In \emph{International Conference on Machine Learning}, pages
  3717--3726, 2018.

\bibitem[Nie et~al.(2018)Nie, Tian, Taylor, and Zou]{nie2018adaptively}
Xinkun Nie, Xiaoying Tian, Jonathan Taylor, and James Zou.
\newblock Why adaptively collected data have negative bias and how to correct
  for it.
\newblock In \emph{International Conference on Artificial Intelligence and
  Statistics}, pages 1261--1269, 2018.

\bibitem[Richter(1965)]{richter1965limit}
Wolfgang Richter.
\newblock Limit theorems for sequences of random variables with sequences of
  random indeces.
\newblock \emph{Theory of Probability \& Its Applications}, 10\penalty0
  (1):\penalty0 74--84, 1965.

\bibitem[Robbins(1952)]{robbins1952some}
Herbert Robbins.
\newblock Some aspects of the sequential design of experiments.
\newblock \emph{Bulletin of the American Mathematical Society}, 58\penalty0
  (5):\penalty0 527--535, 1952.

\bibitem[Russo and Zou(2016)]{russo2016controlling}
Daniel Russo and James Zou.
\newblock Controlling bias in adaptive data analysis using information theory.
\newblock In \emph{Artificial Intelligence and Statistics}, pages 1232--1240,
  2016.

\bibitem[Sackrowitz and Samuel-Cahn(1986)]{sackrowitz1986evaluating}
Harold Sackrowitz and Ester Samuel-Cahn.
\newblock Evaluating the chosen population: a {B}ayes and minimax approach.
\newblock \emph{Lecture Notes-Monograph Series}, pages 386--399, 1986.

\bibitem[Shin et~al.(2019)Shin, Ramdas, and Rinaldo]{shin2019bias}
Jaehyeok Shin, Aaditya Ramdas, and Alessandro Rinaldo.
\newblock Are sample means in multi-armed bandits positively or negatively
  biased?
\newblock In \emph{Advances in Neural Information Processing Systems}, 2019.

\bibitem[Siegmund(1978)]{siegmund1978estimation}
David Siegmund.
\newblock Estimation following sequential tests.
\newblock \emph{Biometrika}, 65\penalty0 (2):\penalty0 341--349, 1978.

\bibitem[Starr(1966)]{starr1966asymptotic}
Norman Starr.
\newblock On the asymptotic efficiency of a sequential procedure for estimating
  the mean.
\newblock \emph{The Annals of Mathematical Statistics}, 37\penalty0
  (5):\penalty0 1173--1185, 1966.

\bibitem[Starr and Woodroofe(1972)]{starr1972further}
Norman Starr and Michael Woodroofe.
\newblock Further remarks on sequential estimation: the exponential case.
\newblock \emph{The Annals of Mathematical Statistics}, pages 1147--1154, 1972.

\bibitem[Tsybakov(2008)]{Tsybakov:2008}
Alexandre~B. Tsybakov.
\newblock \emph{Introduction to Nonparametric Estimation}.
\newblock Springer Publishing Company, Incorporated, 1st edition, 2008.

\bibitem[Villar et~al.(2015)Villar, Bowden, and Wason]{villar2015multi}
Sof{\'\i}a~S Villar, Jack Bowden, and James Wason.
\newblock Multi-armed bandit models for the optimal design of clinical trials:
  benefits and challenges.
\newblock \emph{Statistical science: a review journal of the Institute of
  Mathematical Statistics}, 30\penalty0 (2):\penalty0 199, 2015.

\bibitem[Ville(1939)]{villeetude}
J~Ville.
\newblock {\'E}tude critique de la notion de collectif.
\newblock \emph{Gauthier-Villars, Paris}, 1939.

\bibitem[Wald and Wolfowitz(1948)]{wald1948optimum}
Abraham Wald and Jacob Wolfowitz.
\newblock Optimum character of the sequential probability ratio test.
\newblock \emph{The Annals of Mathematical Statistics}, pages 326--339, 1948.

\bibitem[Williams(1991)]{williams1991probability}
David Williams.
\newblock \emph{Probability with martingales}.
\newblock Cambridge university press, 1991.

\bibitem[Xu et~al.(2013)Xu, Qin, and Liu]{xu2013estimation}
Min Xu, Tao Qin, and Tie-Yan Liu.
\newblock Estimation bias in multi-armed bandit algorithms for search
  advertising.
\newblock In \emph{Advances in Neural Information Processing Systems}, pages
  2400--2408, 2013.

\bibitem[Yang and Barron(1999)]{yang1999information}
Yuhong Yang and Andrew Barron.
\newblock Information-theoretic determination of minimax rates of convergence.
\newblock \emph{Annals of Statistics}, pages 1564--1599, 1999.

\end{thebibliography}

\begin{appendices}
	\section{Examples of the Bregman divergences as a loss function} \label{sec::example_divergence}
In this section, we present examples of Bregman divergences under commonly used assumptions on the underlying distribution.

Using the same notation as in Section~\ref{sec:sub-psi},  the convex conjugate of the function $\lambda \in \Lambda \mapsto \psi_\mu(\lambda) := \lambda \mu+ \psi(\lambda)$ is the function $\psi^*_{\mu}$ on $\Lambda^* := \left\{x \in \mathbb{R} : \sup_{\lambda \in \Lambda} \lambda x - \psi_\mu (\lambda) < \infty   \right\}$ given by
\begin{equation}
	\psi_\mu^*(z) := \sup_{\lambda \in \Lambda} \lambda z - \psi_\mu (\lambda), \quad z \in \Lambda^*.
\end{equation}
The Bregman divergence with respect to $\psi_\mu^*$ is then defined as
\begin{equation}
	D_{\psi_{\mu}^*} (\hat{\mu}, \mu) = \psi_{\mu}^* (\hat{\mu}) - \psi_{\mu}^* (\mu) - \psi_{\mu}^{*\prime}(\mu) \left(\hat{\mu}  - \mu\right), \quad \hat{\mu}, \mu \in \Lambda^*.
\end{equation}

Below we provide some examples demonstrating that $D_{\psi_{\mu}^*} (\hat{\mu}, \mu)$ is a natural loss for the mean estimation problem when the underlying distribution is sub-$\psi$.

\begin{example} \label{eg::Bregman_loss_sub_G}
	If the data are generated from a sub-Gaussian distribution with parameter $\sigma$, then $\psi_\mu(\lambda)$ is defined for all $\lambda \in \mathbb{R}$ as $\psi_{\mu}(\lambda) := \mu\lambda + \frac{\sigma^2}{2}\lambda^2$, the Bregman divergence is defined over $\mathbb{R}$ and is equal to the scaled $\ell_2$ loss: 
	\begin{equation}
		D_{\psi_{\mu}^*} (\hat{\mu}, \mu) := \frac{\left(\hat{\mu} - \mu \right)^2}{2\sigma^2}.
	\end{equation}
\end{example}
\begin{example} \label{eg::Bregman_loss_sub_E}
	If the data are generated from sub-exponential distributions with parameter $(\nu, \alpha)$, then $\psi_{\mu}(\lambda)$ is defined for $\lambda \in ( - 1/\alpha, 1/\alpha)$ as
	$\psi_{\mu}(\lambda) =
	\mu \lambda +\frac{\nu^2}{2}\lambda^2$,
	and  the Bregman divergence is defined over $\mathbb{R}$ and is given as:
	\begin{equation}
		D_{\psi_{\mu}^*} (\hat{\mu}, \mu)  = \begin{cases}
			\frac{1 }{2\nu^2}\left(\hat{\mu} - \mu \right)^2, &\text{ if~~} 
			\left|\hat{\mu} - \mu \right| \leq \frac{\nu^2}{\alpha}, \\
			\frac{1 }{\alpha}\left|\hat{\mu} - \mu \right| - \frac{\nu^2}{2\alpha^2}, &\text{ if~~} 
			\left|\hat{\mu} - \mu \right| >\frac{\nu^2}{\alpha}.
		\end{cases}
	\end{equation}
\end{example}
\begin{example} \label{eg::Bregman_loss_Bernstein}
	If the data-generating distribution $P$ satisfies the Bernstein condition
	\[
	\left|\mathbb{E}_{X\sim P}\left(X-\mu\right)^k \right| \leq \frac{1}{2}k!\sigma^2b^{k-2},~~\text{ for}~~k = 3,4, \ldots,
	\]
	for some $b>0$, where $\sigma^2 = \mathbb{E}_{X\sim P}\left(X-\mu\right)^2$, then, it can be shown that $P$ is sub-$\psi$,  where $\psi_{\mu}(\lambda)$ is defined for $\lambda \in ( - 1/b, 1/b)$ as
	$
	\psi_{\mu}(\lambda) = \mu \lambda + \frac{\lambda^2\sigma^2}{2(1-b|\lambda|)}.
	$
	In this case, the Bregman divergence is defined on $\mathbb{R}$ and can be lower bounded by
	\begin{equation}
		D_{\psi_{\mu}^*} (\hat{\mu}, \mu)  \geq \frac{1}{2}\frac{\left(\hat{\mu} - \mu \right)^2}{\sigma^2 + b\left|\hat{\mu} -\mu\right|} .
	\end{equation}
\end{example}
\begin{example} \label{eg::Bregman_loss_Ber}
	If the data are generated from a Bernoulli distribution, then recalling that the uncentered CGF is given by $\psi_{\mu}(\lambda) = \log(1-\mu + \mu e^\lambda)$, for $\mu \in (0,1)$, the Bregman divergence is defined on $(0,1)$ and is given by 
	\begin{equation}
		D_{\psi_{\mu}^*} (\hat{\mu}, \mu) = \hat{\mu} \log \frac{\hat{\mu}}{\mu} + \left(1- \hat{\mu}\right) \log \frac{1-\hat{\mu}}{1-\mu}.
	\end{equation}
\end{example}

\section{Proof of Theorem~\ref{thm::bounds_in_fully_adap_finite_moment} and related statements}
\label{sec::proof_of_main_thm}

Recall that we assume that there exists a time $t_0$ such that, almost surely, $\Tau \geq \tau \geq t_0$ and $N_k(t_0) \geq 3$  for all $ k \in [{\K}]$.
For ease of readability, we drop the subscript $k$ throughout this section. We begin with the proof of the adaptive deviation inequality~\eqref{eq::adaptive_deviation_ineq_finite_moment} of Lemma~\ref{lemma::adaptive_deviation_ineq_finite_moment}, which is a fundamental component of the proof of Theorem~\ref{thm::bounds_in_fully_adap_finite_moment} and related statements.

\subsection{Proof of Lemma~\ref{lemma::adaptive_deviation_ineq_finite_moment}} \label{subSec::proof_of_adaptive_ineq_finite_moment}
The proof strategy involves splitting the deviation event into simpler sub-events and then find exponential bounds for the probability of each sub-event. In detail, for each $t \geq 0$ and $j \geq 2$, define the events
\begin{align*}
F_t &:= \left\{  N(t) > e, \frac{N(t)}{4\sigma^2}\left(\hat{\mu}_t - \mu\right)^2> e \delta \log N(t)  \right\}, \\
G_t &:= \left\{ \hat{\mu}(t) \geq \mu \right\}, \text{ and}\\
H_t^j & := \left\{ e^{j-1} \leq N(t) < e^j \right\}.
\end{align*}
We remark that the use of constant $e$ above is purely for mathematical convenience; any other constant would have also sufficed.
To bound the probability of the aforementioned events, we prove the following lemma.
\begin{lemma} \label{lemma::1_l2}
	For any fixed $\delta >0$, there exists a deterministic $\lambda_j \geq 0$ such that
	\begin{equation}
	\left\{ F_t \cap G_t \cap H_t^j \right\} \subset  \left\{\lambda_j \left[S(t) -\mu N(t)\right]- \lambda_j^2\sigma^2 N(t) \geq \delta j\right\},
	\end{equation}
	and  a deterministic $\lambda_j' < 0$ such that
	\begin{equation}
	\left\{ F_t \cap G_t^c \cap H_t^j \right\} \subset  \left\{\lambda_j^\prime  \left[S(t) -\mu N(t)\right]- {\lambda_j^\prime}^2\sigma^2 N(t) \geq \delta j\right\}.
	\end{equation}
\end{lemma}

\begin{proof}[Proof of Lemma~\ref{lemma::1_l2}]
 The proof borrows arguments from the proof of Theorem 11 in \cite{garivier2011kl}.
	On the event $F_t \cap G_t \cap H_t^j$, since $e \leq e^{j-1} \leq N(t) < e^j$ and $v \mapsto\frac{\log v }{v}$ is non-increasing on $[e, \infty)$, we have that 
	\begin{equation} \label{eq::control_psi_star_l2}
    \frac{1}{4\sigma^2} \left(\hat{\mu}(t) -\mu\right)^2 > \frac{e \delta \log N(t)}{N(t)} 
	\geq \frac{ \delta j}{e^{j-1}}  > 0.
	\end{equation}
    Now, pick a deterministic real number $z_j \geq 0$ such that 
	\[
	\frac{1}{4\sigma^2}z_j^2 =  \frac{\delta j}{e^{j-1}} .
	\]
	Since  $\hat{\mu}(t)  -\mu \geq 0$ and $x \mapsto x^2$ is an increasing function on $[0, \infty)$, our choice of $z_j$ along with the inequalities in \eqref{eq::control_psi_star_l2}  implies that $\hat{\mu}(t)  -\mu \geq z_j$, on the event $F_t \cap G_t \cap H_t^j$. Define $\lambda_j := \frac{z_j}{2\sigma^2}$, then, on the event  $F_t \cap G_t \cap H_t^j$, we have that
	\begin{align*}
	\lambda_j \left[\hat{\mu}(t) -\mu\right] - \lambda_j^2\sigma^2  \geq \lambda_j z_j - \lambda_j^2\sigma^2 = \frac{1}{4\sigma^2}z_j^2 = \frac{\delta j}{e^{j-1}} \geq \frac{\delta j}{N(t)} .
	\end{align*} 	
	Re-arranging, we get that
	\[
	\lambda_j \left[S(t) -\mu N(t)\right]- \lambda_j^2\sigma^2 N(t) \geq \delta j.
	\]
	which proves the first statement in the lemma. For the second statement, set $z_j^\prime := -z_j$. Note that, since  $\hat{\mu}(t)  -\mu < 0$ and $x \mapsto x^2$ is an decreasing function on $(-\infty, 0]$, the choice of $z_j'$ and the inequalities in \eqref{eq::control_psi_star_l2} yield that $\hat{\mu}(t)  -\mu < z_j$, on the event $F_t \cap G_t^c \cap H_t^j$. 	Let $\lambda_j^\prime := \frac{z_j^\prime}{2\sigma^2}$. Then, by the same argument used for $\hat{\mu}(t) -\mu \geq 0$ case,  the second statement holds which completes the proof.
\end{proof}

We now return to the proof of Lemma~\ref{lemma::adaptive_deviation_ineq_finite_moment}.
The adaptive deviation inequality \eqref{eq::adaptive_deviation_ineq_finite_moment} can be re-written as the following inequality, 
	\begin{equation}
	\begin{aligned}
	\mathbb{P}\left(\frac{N(\tau)}{4\sigma^2}\left(\hat{\mu}(\tau) -\mu\right)^2 \geq e \delta \log_e N(\tau)\right) 
	\leq \frac{C_p}{\delta^p},
	\end{aligned}
	\end{equation}	
	where $C_p$ is a constant depending only on $p$. To prove the above inequality, it suffices to show that it holds uniformly over time (e.g., see Lemma 3 in \cite{howard2021time}). Therefore, below, we prove the following uniform concentration inequality: 
	\begin{equation}\label{eq::adap_deviation_simple_form_l2}
	\begin{aligned}
	\mathbb{P}\left(\exists t \in \mathbb{N} : N(t) \geq e, \frac{N(t)}{4\sigma^2}\left(\hat{\mu}(t) -\mu\right)^2 \geq e\delta \log N(t)\right) 
	\leq \frac{C_p}{\delta^p}.
	\end{aligned}
	\end{equation}				
	The event on the left-hand side of \eqref{eq::adap_deviation_simple_form_l2} is equal to $\bigcup_{t = 1}^\infty F_t$, and its probability can be bounded follows:
	\begin{align*}
	\mathbb{P}\left( \bigcup_{t = 1}^\infty F_t \right) &= \mathbb{P} \left( \bigcup_{t = 1}^\infty \bigcup_{j =2}^\infty \left[ H_t^{j} \cap F_t \cap \left( G_t \cup G_t^c \right) \right]\right) \\
	&= \mathbb{P} \left( \left[\bigcup_{t = 1}^\infty \bigcup_{j =2}^\infty \left( H_t^{j} \cap F_t \cap  G_t \right) \right] \cup \left[ \bigcup_{t = 1}^\infty \bigcup_{j =2}^\infty \left( H_t^{j} \cap F_t \cap  G_t^c \right)\right]\right) \\
	& \leq \mathbb{P} \left( \bigcup_{t = 1}^\infty \bigcup_{j =2}^\infty \left( H_t^{j} \cap F_t \cap  G_t \right) \right) + \mathbb{P} \left( \bigcup_{t = 1}^\infty \bigcup_{j =2}^\infty \left( H_t^{j} \cap F_t \cap  G_t^c \right) \right).
	\end{align*}
     By Lemma~\ref{lemma::1_l2}, we have that
	\begin{align*}
	&\mathbb{P} \left( \bigcup_{t = 1}^\infty \bigcup_{j =2}^\infty \left( H_t^{j} \cap F_t \cap  G_t \right) \right)  \\
	&\leq \mathbb{P}\left(\exists t \geq 0, \exists j \geq 2 :  \lambda_j \left[S(t)-\mu N(t)\right] - \lambda_j^2\sigma^2 N(t) > \delta j  \right) ~~\text{(by Lemma~\ref{lemma::1_l2}.)}\\
	& \leq  \sum_{j \geq 2} \mathbb{P}\left(\exists t \geq 0,\lambda_j \left[S(t)-\mu N(t)\right] - \lambda_j^2\sigma^2 N(t) > \delta j  \right) ,
	\end{align*}
	where the last inequality stems from the union bound. To get a bound for each probability term, we use the following inequality from \cite{khan2009p} which is a generalization of the Dubins-Savage inequality \cite{darling_inequalities_1967}.
	
	\begin{proposition}[$\ell_p$-version of the Dubins-Savage inequality \cite{khan2009p}]
	    Let $\left\{M(t)\right\}$ be a martingale with respect to a filtration $\{\F_t\}_{t \geq 0}$ such that $M(0) = 0$ and $\mathbb{E}\left[X_t^{2p}\mid \F_{t-1}\right] < \infty$ with  $X_t := M(t) - M(t-1)$ for all $t \geq 1$. Let $\nu_t$ be the conditional variance given by $\nu_t = \mathbb{E}\left[ X_t^2 | \mathcal{F}_{t-1}\right]$. Then, there exist a constant $C_p'$ depending only on $p$ such that for any $a \geq 0$, $b >0$, the following inequality holds.
	    \begin{equation} \label{eq::DS_lp_ineq}
	         \mathbb{P}\left(\exists t \geq 0 : M(t) \geq a + b \sum_{s=1}^t \nu_s  \right) \leq \frac{1}{\left(1 + ab / C_p'\right)^p}.
	    \end{equation}
	\end{proposition}
	\noindent
	Applying inequality~\eqref{eq::DS_lp_ineq} with $M(t) =\lambda_j \left[S(t)-\mu N(t)\right]$, $\sum_{s=1}^t \nu_s  = \lambda_j^2\sigma^2 N(t)$, $a = \delta j$ and $b = 1$, we have the following bound,
	\begin{align*}
	   &\sum_{j \geq 2} \mathbb{P}\left(\exists t \geq 0,\lambda_j \left[S(t)-\mu N(t)\right] - \lambda_j^2\sigma^2 N(t) > \delta j  \right) \\
	   &\leq \sum_{j \geq 2} \frac{1}{\left(1 + \delta j / C_p'\right)^p} \\
	   & \leq  \frac{C_p^{'p}}{\delta^p} \sum_{j \geq 2} \frac{1}{j^p} \quad := \quad  \frac{C''_{p}}{\delta^p}.
	\end{align*}
	Similarly, it can be shown that
	\[
	\mathbb{P} \left( \bigcup_{t = 1}^\infty \bigcup_{j =2}^\infty \left( H_t^{j} \cap F_t \cap  G_t^c \right) \right) 
	\leq \frac{C''_{p}}{\delta^p}.
	\]
	By combining  two bounds, we get that
	\begin{align*}
	\mathbb{P}\left(\exists t \geq 0 : N(t) > e, \frac{N(t)}{4\sigma^2}\left(\hat{\mu}_t - \mu\right)^2 > e\delta \log N(t)  \right)
	\leq \frac{2C''_{p}}{\delta^p},
	\end{align*}
	which implies the desired bound on the adaptive deviation probability in \eqref{eq::adaptive_deviation_ineq_finite_moment} with $C_p :=2(4e)^p C_p''$. This completes the proof of Lemma~\ref{lemma::adaptive_deviation_ineq_finite_moment}.

\subsection{Proof of Theorem~\ref{thm::bounds_in_fully_adap_finite_moment}} \label{subSec::Proof_of_second_them_finite_moment}
The proof of Theorem~\ref{thm::bounds_in_fully_adap_finite_moment} borrows arguments from \cite{jiao2017dependence}, like the following lower bound of $D_{f_q}$. 
\begin{lemma} \label{lemma::phi_diver_lower_bound}
	Let $P$, $Q$ be probability measures on $\mathcal{X}$ and let  $f : \mathcal{X} \mapsto \mathbb{R}$ be a function satisfying $\mathbb{E}_Q \left[f^p(X)\right] < \infty$ for some $p \geq 1$. Then, for $q$ such that $1/p + 1/q = 1$, we have
	\begin{equation}
	\frac{1}{q}D_{f_q} (P || Q)  \geq  \mathbb{E}_P \left[f(X)\right] - \mathbb{E}_Q \left[f(X)\right] - \mathbb{E}_Q \left[\frac{\left|f(X)\right|^p}{p}\right].
	\end{equation}
\end{lemma}

To apply the above inequality, we need the following bound on the expectation of the $2p$-norm of the stopped adaptive process, which is based on the adaptive deviation inequality in Lemma~\ref{lemma::adaptive_deviation_ineq_finite_moment}.
\begin{claim} \label{claim::martingale_adap_ineq_finite_moment}
	Under the assumptions of Theorem~\ref{thm::bounds_in_fully_adap_finite_moment},	for each $k \in[{\K}]$ and for any $\alpha \leq p$ we have that
	\begin{equation} \label{eq::martingale_adap_ineq_finite_moment}
    \left\|\frac{N_k(\tau)}{\log N_k(\tau)}\left(\frac{\hat{\mu}_k(\tau)-\mu_k}{\sigma_k} \right)^{2}\right\|_\alpha   \leq C_{\alpha,\epsilon},
	\end{equation}
	where $C_{\alpha,\epsilon}$ is a constant depending only on $\alpha$ and $\epsilon$.
\end{claim}
\begin{proof}[Proof of Claim~\ref{claim::martingale_adap_ineq_finite_moment}]
    Since $\alpha \leq p$, arm $k$ also has a finite $2(\alpha + \epsilon)$-norm. Therefore, applying Lemma~\ref{lemma::adaptive_deviation_ineq_finite_moment} with $p = \alpha$, we get 
	\begin{equation} \label{eq::Fubini_lp} 
	\begin{aligned}
	\mathbb{E}\left[\frac{N_k(\tau)}{\log N_k(\tau)}\left(\frac{\hat{\mu}_k(\tau)-\mu_k}{\sigma_k} \right)^{2}\right]^\alpha
	&\leq 1 +  \int_{1}^{\infty}\mathbb{P}\left(\frac{N_k(\tau)}{\log N_k(\tau)}\left(\frac{\hat{\mu}_k(\tau)-\mu_k}{\sigma_k} \right)^{2} >  \delta^{1/\alpha}\right)\mathrm{d}\delta \\
	&\leq 1+ C_{\alpha+\epsilon}\int_{1}^{\infty}\frac{1}{\delta^{1+\epsilon/\alpha}}\mathrm{d}\delta 
	=1 +  \frac{C_{\alpha+\epsilon}}{\epsilon /\alpha }.
	\end{aligned}
	\end{equation}
    The claim  readily follows by letting $C_{\alpha, \epsilon}:= \left(1 +  \frac{C_{\alpha+\epsilon}}{\epsilon /\alpha}\right)^{1/\alpha}$.
\end{proof}

\noindent We now have all the components in place to complete the proof of Theorem~\ref{thm::bounds_in_fully_adap_finite_moment}.
	For $k \in [{\K}]$, set $   P_k = \mathcal{L}\left(\D_{\Tau} | {\kappa} = k\right)$, $Q = \mathcal{L}\left(\D_{\Tau}\right)$ and 
	\[
	f_k = \lambda\frac{N_{k}(\tau)}{\log N_{k}(\tau)} \left(\hat{\mu}_k(\tau)-\mu_k \right)^{2}.
	\]
	for a $\lambda >0$. 	Then, from Lemma~\ref{lemma::phi_diver_lower_bound}, we can lower bound $I_q\left(\kappa, \D_{\Tau}\right)$ as follows:
	\begin{align*}
	\frac{1}{q}I_q\left(\kappa, \D_{\Tau}\right) &= \sum_{k =  1}^{\K} \mathbb{P}(\kappa = k)  \left\{ \frac{1}{q}D_{f_q}\left(\mathcal{L}\left(\D_{\Tau} | {\kappa} = k\right) ||\mathcal{L}\left(\D_{\Tau}\right) \right)\right\} \\
	& \geq  \sum_{k =  1}^{\K} \mathbb{P}(\kappa = k) \left\{ \mathbb{E}_{P_k} \left[f_k\right] - \mathbb{E}_{Q} \left[f_k\right] - \mathbb{E}_Q\left[\frac{\left|f_k\right|^p}{p}\right]\right\}   \\	
	& = \sum_{k = 1}^{\K} \mathbb{P}(\kappa = k)  \left\{\lambda\mathbb{E}\left[\frac{N_{k}(\tau)}{\log N_{k}(\tau)} \left(\hat{\mu}_k(\tau)-\mu_k \right)^{2} \mid \kappa = k\right] \right. \\
	&~~~~~~~~~~ \left.-\lambda\mathbb{E}\left[\frac{N_{k}(\tau)}{\log N_{k}(\tau)} \left(\hat{\mu}_k(\tau)-\mu_k \right)^{2} \right] -  \frac{\lambda^p}{p}\mathbb{E}\left[\frac{N_{k}(\tau)}{\log N_{k}(\tau)} \left(\hat{\mu}_k(\tau)-\mu_k \right)^{2} \right]^p\right\}\\
	& \geq \sum_{k = 1}^{\K} \mathbb{P}(\kappa = k)  \left\{\lambda\mathbb{E}\left[\frac{N_{k}(\tau)}{\log N_{k}(\tau)} \left(\hat{\mu}_k(\tau)-\mu_k \right)^{2} \mid \kappa = k\right] \right\}\\	&~~~~~~~~~~-\mathbb{P}(\kappa = k)\left\{  \lambda C_{1,\epsilon}\sigma_k^2  + (\lambda C_{p,\epsilon}\sigma_k^2)^p / p \right\}\\
	& =  \lambda\mathbb{E}\left[\frac{N_{\kappa}(\tau)}{\log N_{\kappa}(\tau)} \left(\hat{\mu}_\kappa(\tau)-\mu_\kappa \right)^{2}\right] -  \left(\lambda C_{1,\epsilon}\|\sigma_\kappa\|_2^2  + \frac{\lambda^p C_{p,\epsilon}^p}{p}\|\sigma_\kappa\|_{2p}^{2p}\right).
	\end{align*}
	Since this inequality holds for any $\lambda >0$, we get
	\begin{align*}
	  \mathbb{E}\left[\frac{N_{\kappa}(\tau)}{\log N_{\kappa}(\tau)} \left(\hat{\mu}_\kappa(\tau)-\mu_\kappa \right)^{2}\right] &= C_{1,\epsilon} \|\sigma_\kappa\|_2^2 + \inf_{\lambda >0} \frac{1}{\lambda}\left\{ \frac{I_q\left(\kappa, \D_{\Tau}\right)}{q}  +  \frac{\lambda^p C_{p,\epsilon}^p}{p}\|\sigma_\kappa\|_{2p}^{2p}\right\} \\
	  & =  C_{1,\epsilon}\|\sigma_\kappa\|_2^2 + C_{p,\epsilon}\|\sigma_\kappa\|_{2p}^2  I_q^{1/q}\left(\kappa, \D_{\Tau}\right),
	\end{align*}
thus completing the proof of the theorem.

We  conclude this section with a short proof of Corollary~\ref{cor::bounds_in_finite_moment}.

\subsection{Proof of Corollary~\ref{cor::bounds_in_finite_moment}} \label{subSec::proof_of_bound_in_finite_moment}

	For any $p, q >1$ with $\frac{1}{p} + \frac{1}{q} = 1$,  H\"older's inequality along with the bound on the adaptive risk in \eqref{eq::adap_risk_finite_moment}  implies that
	\begin{align*}
	\left[\mathbb{E} \left(\hat{\mu}_\kappa-\mu_\kappa\right)^{2/p} \right]^p 
	&= \left[\mathbb{E} \left(\frac{N_{\kappa}}{\log N_{\kappa}(\tau)}\right)^{-1/p} \left(\frac{N_{\kappa}}{\log N_{\kappa}(\tau)}\right)^{1/p}  \left(\hat{\mu}_\kappa-\mu_\kappa\right)^{2/p} \right]^p\\
	&\leq \left[\mathbb{E} \left(\frac{N_{\kappa}}{\log N_{\kappa}(\tau)}\right)^{-q/p} \right]^{p/q} \mathbb{E} \left[\frac{N_{\kappa}}{\log N_{\kappa}(\tau)} \left(\hat{\mu}_\kappa-\mu_\kappa\right)^{2} \right] \\  
	& \leq  \left[\mathbb{E} \left(\frac{N_{\kappa}}{\log N_{\kappa}(\tau)}\right)^{-q/p} \right]^{p/q}  \left[ C_{1,\epsilon}\|\sigma_\kappa\|_2^2 + C_{p,\epsilon}\|\sigma_\kappa\|_{2p}^2 I_q^{1/q}\left(\kappa, \D_{\Tau}\right) \right]\\
	& = \frac{1}{\tilde{n}^{\eff,q/p}}  \left[ C_{1,\epsilon}\|\sigma_\kappa\|_2^2 + C_{p,\epsilon}\|\sigma_\kappa\|_{2p}^2 I_q^{1/q}\left(\kappa, \D_{\Tau}\right) \right].
	\end{align*}
	By setting $r := 1/p$, we infer inequality \eqref{eq::1/p-norm_bound_finite_moment},  completing the proof.  

\section{Proofs of Theorem~\ref{thm::bound_on_Bregman} and related statements} \label{sec::proof_of_first_theorem}
The proof of Theorem~\ref{thm::bound_on_Bregman} is based on the deviation inequality given in Lemma~\ref{lemma::deviation_ineq}, which we prove first. 
\subsection{Proof of Lemma~\ref{lemma::deviation_ineq}} \label{subSec::proof_of_deviation_ineq}
The proof of the deviation inequality in Lemma~\ref{lemma::deviation_ineq} is based on the following bound on the expectation of the exponential of the stopped process. Similar versions of this bound has been exist in the literature: see, e.g., see \cite{garivier2011kl, howard2020time}. For the completeness, we provide the proof of the bound. 
\begin{claim}\label{claim::martingale_ineq}
	Under the assumptions of Theorem~\ref{thm::bound_on_Bregman},  for any $\lambda \in \Lambda$, it holds that
	\begin{equation}
		\mathbb{E}	\left[\exp\left\{\lambda \left( S_k(\Tau)- \mu_k N_k(\Tau) \right)- N_k(\Tau)\psi(\lambda)  \right\} \right] \leq 1.
	\end{equation}
\end{claim}
\begin{proof}[Proof of Claim~\ref{claim::martingale_ineq}]
	Set $L_t^k(\lambda) := \exp\left\{ \lambda\left( S_k(t)- \mu_k N_k(t) \right) - N_k(t)\psi(\lambda)  \right\}$. First note that, for any $t \geq 0$,  
	\begin{align*}
		&\mathbb{E}\left[\exp\left\{ \lambda \left[\left( S_k(t+1)- \mu_k N_k(t+1) \right) - \left( S_k(t)- \mu_k N_k(t) \right)\right] \right\} \mid \mathcal{F}_t   \right] \\
		&=	\mathbb{E}\left[\exp\left\{ \lambda  \mathbbm{1}(A_{t+1} = k) \left[Y_{t+1} - \mu_k\right] \right\} \mid \mathcal{F}_t   \right] \\
		&=	 \mathbb{E}\left[ \mathbbm{1}(A_{t+1} = k)\exp\left\{ \lambda  \left(Y_{t+1} - \mu_k\right) \right\}  +  \mathbbm{1}(A_{t+1} \neq k) \mid \mathcal{F}_t   \right] \\
		&=	\mathbbm{1}(A_{t+1} = k) \mathbb{E}\left[ \exp\left\{ \lambda  \left(Y_{t+1} - \mu_k\right) \right\} \mid \mathcal{F}_t   \right] +  \mathbbm{1}(A_{t+1} \neq k) 
		~~\text{(since $\mathbbm{1}\left(A_{t+1}= k \right) \in \mathcal{F}_t$.)}\\
		& \leq 	\mathbbm{1}(A_{t+1} = k) \exp\left\{\psi(\lambda) \right\} +  \mathbbm{1}(A_{t+1} \neq k)  
		~~\text{(since $k$-the distribution is sub-$\psi$.)}\\
		& = \exp\left\{	\mathbbm{1}(A_{t+1} = k)  \psi(\lambda) \right\} \\
		& = \exp\left\{	\left[N_k(t+1) - N_k(t) \right]\psi(\lambda) \right\}.
	\end{align*}
	Thus, we obtain that
	\begin{align*}
		&\mathbb{E}\left[L_{t+1}^k(\lambda) \mid \mathcal{F}_{t}\right]  \\
		&= \mathbb{E}\left[\exp\left\{ \lambda\left( S_k(t+1)- \mu_k N_k(t+1) \right) - N_k(t+1)\psi(\lambda)  \right\} \mid \mathcal{F}_t\right] \\
		&=  \mathbb{E}\left[\exp\left\{ \lambda \left[\left( S_k(t+1)- \mu_k N_k(t+1) \right) - \left( S_k(t)- \mu_k N_k(t) \right)\right] \right\} \mid \mathcal{F}_t   \right]  \\
		&~~~~\cdot \exp\left\{ \lambda\left( S_k(t)- \mu_k N_k(t) \right) -  N_k(t+1)\psi(\lambda) \right\}
		~~\text{(since $S_k(t), N_k(t), N_k(t+1) \in \mathcal{F}_t$.)}\\
		&\leq  \exp\left\{	\left[N_k(t+1) - N_k(t) \right]\psi(\lambda) \right\} \exp\left\{ \lambda\left( S_k(t)- \mu_k N_k(t) \right) -  N_k(t+1)\psi(\lambda) \right\}\\ 
		&\leq \exp\left\{ \lambda\left( S_k(t)- \mu_k N_k(t) \right) - N_k(t)\psi(\lambda)  \right\} \\
		& = L_t^k(\lambda).
	\end{align*}
	In particular, 
	\begin{align*}
		\mathbb{E}\left[L_{1}^k(\lambda) \mid \mathcal{F}_{0}\right] =1 := L_0^k,~~\forall \lambda \in \Lambda.
	\end{align*}
	Therefore $\{L_t^k(\lambda)\}_{t \geq 0}$ is a non-negative super-martingale, and the result follows from  the optional stopping theorem. 
\end{proof}

Returning to the proof of Lemma~\ref{lemma::deviation_ineq}, we first consider the case $\mathbb{P}(\Tau \leq M) = 1$ for some constant $M > 0$. Since $N_k(\Tau) \leq \Tau$, we must also have that $\mathbb{P}(N_k(\Tau) \leq M) = 1$. Next, for any $\epsilon \geq 0$ and $\lambda \in [0,\lambda_{\max} / p ) \subset \Lambda $, we have
\begin{align*}
	\mathbb{P} \left(  S_k(\Tau) / N_k(\Tau)  -\mu_k \geq \epsilon \right) 
	& = \mathbb{P}\left(  S_k(\Tau) \geq N_k(\Tau) (\epsilon + \mu_k) \right) \\
	&=  \mathbb{P}\left( \exp\left\{\lambda S_k(\Tau) - \lambda (\epsilon + \mu_k) N_k(\Tau)   \right\} \geq 1 \right)\\
	& \leq\mathbb{E}\left[\exp\left\{\lambda S_k(\Tau) - \lambda (\epsilon+ \mu_k) N_k(\Tau) \right\} \right],
\end{align*}
where in the final step we have used Markov's inequality. By
using  H\"older's inequality with any conjugate pairs $p, q > 1$ with $1/p + 1/q = 1$, the last term can be  bounded as follows:
\begin{align*}
	&\mathbb{E}\left[\exp\left\{\lambda \left( S_k(\Tau)- \mu_k N_k(\Tau) \right)- \lambda \epsilon N_k(\Tau) \right\} \right] \\
	&= \mathbb{E}\left[\exp\left\{\lambda \left( S_k(\Tau)- \mu_k N_k(\Tau) \right) - \frac{N_k(\Tau)}{p} \psi(p\lambda)  \right\} \exp\left\{N_k(\Tau) \left(\frac{1}{p}\psi(p\lambda) - \lambda \epsilon   \right) \right\}\right] \\\
	&\leq\left[\mathbb{E}\exp\left\{p\lambda \left( S_k(\Tau)- \mu_k N_k(\Tau) \right)- N_k(\Tau)\psi(p\lambda)  \right\} \right]^{1/p}\left[\mathbb{E} \exp\left\{qN_k(\Tau) \left(\frac{1}{p}\psi(p\lambda)- \lambda \epsilon\right) \right\}\right]^{1/q}\\
	&\leq  \left[\mathbb{E} \exp\left\{qN_k(\Tau) \left(\frac{1}{p}\psi(p\lambda) - \lambda \epsilon  \right) \right\}\right]^{1/q} .
\end{align*}
where  the last inequality follows from Claim~\ref{claim::martingale_ineq}. Thus we have established the following intermediate bound on the deviation probability:
\begin{equation}
	\mathbb{P} \left(  S_k(\Tau) / N_k(\Tau)  -\mu_k \geq \epsilon \right)  \leq \left[\mathbb{E} \exp\left\{qN_k(\Tau) \left(\frac{1}{p}\psi(p\lambda) - \lambda \epsilon  \right) \right\}\right]^{1/q}.
\end{equation}
Since $\epsilon \geq 0$, the convex conjugate of $\psi$ at $\epsilon$ can be written as
\[
\psi^*(\epsilon) = \sup_{\lambda\in \Lambda} \left\{\lambda \epsilon-  \psi(\lambda)\right\} = \sup_{\lambda\in [0, \lambda_{\max})} \left\{\lambda \epsilon-  \psi(\lambda)\right\}.
\]
Thus,
\begin{align*}
	\sup_{\lambda \in [0, \lambda_{\max}/p)} \left\{\lambda \epsilon - \frac{1}{p} \psi(p\lambda)\right\} &=\frac{1}{p}\sup_{\lambda\in [0, \lambda_{\max})} \left\{\lambda \epsilon-  \psi(\lambda)\right\} \\
	&=\frac{1}{p}\sup_{\lambda\in \Lambda} \left\{\lambda \epsilon-  \psi(\lambda)\right\} \\
	& = \frac{1}{p} \psi^*(\epsilon).
\end{align*}
Using this identity, the deviation probability can be further bounded as
\begin{align*}
	&\mathbb{P} \left(  S_k(\Tau) / N_k(\Tau)  -\mu_k \geq \epsilon \right) \\
	&\leq \inf_{\lambda \in [0, \lambda_{\max}/p)} \left[\mathbb{E} \exp\left\{qN_k(\Tau) \left(\frac{1}{p}\psi(p\lambda) - \lambda \epsilon  \right) \right\}\right]^{1/q} \\
	&= \left[\mathbb{E} \exp\left\{-qN_k(\Tau) \sup_{\lambda \in [0, \lambda_{\max}/p)} \left(\lambda \epsilon -\frac{1}{p}\psi(p\lambda)  \right) \right\}\right]^{1/q}\\
	&= \left[ \mathbb{E} \exp\left\{-\frac{q}{p} \psi^*(\epsilon)N_k(\Tau)\right\}\right]^{1/q} .
\end{align*}
Using the same argument, it also follows that 
\[
\mathbb{P} \left(  S_k(\Tau) / N_k(\Tau)  -\mu_k \leq -\epsilon \right) \leq \left[ \mathbb{E} \exp\left\{-\frac{q}{p} \psi^*(-\epsilon)N_k(\Tau)\right\}\right]^{1/q}. 
\]	
Since $\psi^*$ is a non-negative convex function with $\psi^*(0) = 0$, for any $\delta \geq0$, there exist $\epsilon_{1}, \epsilon_{2} \geq 0$ with $\psi^*(\epsilon_{1}) = \psi^*(- \epsilon_{2})  = \delta$ such that 
\[
\left\{z \in \mathbb{R} : \psi^*(z) \geq \delta  \right\} = \left\{z \in \mathbb{R} :  z  \geq \mu_k + \epsilon_{1}, z \leq \mu_k-\epsilon_{2}  \right\}.
\]
Therefore, for any $\delta \geq 0$ and  $p, q >1$ with $1/p + 1/q = 1$, we conclude that
\begin{align*}
	\mathbb{P}\left(D_{\psi^*_{\mu_k}}(\hat{\mu}_k(\Tau) , \mu_k) \geq \delta \right) 
	&= \mathbb{P}\left(\psi^*_{\mu_k}(S_k(\Tau) / N_k(\Tau)) \geq \delta\right) ~~\text{(By the equality~\eqref{eq::KL_equiv_psi_star} in Fact~\ref{fact::KL_equiv_psi_star}.)} \\
	& \leq   \mathbb{P}\left(S_k(\Tau) / N_k(\Tau))   - \mu_k\geq  \epsilon_{1} \right)  + \mathbb{P}\left(S_k(\Tau) / N_k(\Tau))  - \mu_k\leq - \epsilon_2 \right)  \\
	&\leq 2 \left[ \mathbb{E} \exp\left\{-\frac{q}{p} \delta N_k(\Tau)\right\}\right]^{1/q}.
\end{align*}
For general $\Tau$, let $\Tau_M := \min\left\{\Tau, M  \right\}$ for all $M >0$. Since $\Tau_M$ is a stopping time with $\mathbb{P}(\Tau_M \leq M) = 1$, we have 
\begin{equation}\label{eq::deviation_finite_tau}
	\mathbb{P}\left(D_{\psi^*_{\mu_k}}(\hat{\mu}_k(\Tau_M) , \mu_k) \geq \delta \right)  \leq 2 \left[ \mathbb{E} \exp\left\{-\frac{q}{p} \delta N_k(\Tau_M)\right\}\right]^{1/q}
\end{equation}
for any $\delta \geq 0$ and  $p, q >1$ with $1/p + 1/q = 1$. Then, we have 
\begin{align*}
	\mathbb{P}\left(D_{\psi^*_{\mu_k}}(\hat{\mu}_k(\Tau) , \mu_k) \geq \delta \right)  
	& \leq \liminf_{M \rightarrow \infty} \mathbb{P}\left(D_{\psi^*_{\mu_k}}(\hat{\mu}_k(\Tau_M) , \mu_k) \geq \delta \right)  \\
	& \leq \liminf_{M \rightarrow \infty} 	2 \left[ \mathbb{E} \exp\left\{-\frac{q}{p} \delta N_k(\Tau_M)\right\}\right]^{1/q} \\
	& \leq 2 \left[ \mathbb{E} \exp\left\{-\frac{q}{p} \delta N_k(\Tau)\right\}\right]^{1/q},			
\end{align*}
where the first inequality comes from the Fatous's lemma and the continuity of the Bregman divergence, the second one from the inequality \eqref{eq::deviation_finite_tau} and the last one  from the monotone convergence theorem, along with the facts that
\[
0 \leq \exp\left\{-\frac{q}{p} \delta N_k(\Tau_M)\right\} \leq 1,~~\forall M >0,
\]
and that
$\exp\left\{-\frac{q}{p} \delta N_k(\Tau_M)\right\} $ is decreasing in $M$ and converges almost surely to $\exp\left\{-\frac{q}{p} \delta N_k(\Tau)\right\}$ as $M \rightarrow \infty$.

Finally, from the identity $q/p = q-1$, we have that
\begin{equation}
	\mathbb{P}\left(D_{\psi^*_{\mu_k}}(\hat{\mu}_k(\Tau) , \mu_k) \right)\leq 2 \inf_{q > 1} \left[ \mathbb{E} \exp\left\{-(q-1) \delta N_k(\Tau)\right\}\right]^{1/q}.
\end{equation}
Since choosing $q =1$  gives a valid, albeit trivial, bound, we can take the infimum over $q \geq 1$, which proves the first inequality in \eqref{eq::deviation_ineq}.  The second inequality follows from the assumption that $N_k(\Tau) \geq b$ and the inequality
\begin{align*}
	2 \inf_{q \geq 1} \left[ \mathbb{E} \exp\left\{-(q-1) \delta N_k(\Tau)\right\}\right]^{1/q} 
	& \leq 2 \inf_{q \geq 1} \left[  \exp\left\{-(q-1) \delta b\right\}\right]^{1/q} \\
	& = 2\exp\left\{-\delta b \right\}.
\end{align*}
This completes the proof of Lemma~\ref{lemma::deviation_ineq}.
\subsection{Proof of Theorem~\ref{thm::bound_on_Bregman}} 
\label{subSec::proof_of_bounds_on_Bregman}
In Lemma~\ref{lemma::deviation_ineq}, we have established  the the deviation inequality
\begin{equation} \label{eq::deviation_p_q}
	\mathbb{P}\left(D_{\psi^*_{\mu_k}}(\hat{\mu}_k(\Tau) , \mu_k) \geq \delta \right)  \leq 2 \left[ \mathbb{E} \exp\left\{-\frac{q}{p} \delta N_k(\Tau)\right\}\right]^{1/q},			
\end{equation}
for any $p, q > 1$ with $\frac{1}{p} + \frac{1}{q} = 1$. We first prove Theorem~\ref{thm::bound_on_Bregman} by consider the case of $\mathbb{P}(\Tau \leq M) = 1$ for a  $M > 0$. Since $N_k(\Tau) \leq \Tau$, we then have that $\mathbb{P}(N_k(\Tau) \leq M) = 1$. 	By using the above deviation inequality and the well-known identity  $\mathbb{E}|X| = \int_{0}^\infty \mathbb{P}(|X| > \delta) \mathrm{d}\delta$ for any integrable random variable $X$, we have
\begin{align*}
	&\mathbb{E} D_{\psi^*_{\mu_k}}\left(\hat{\mu}_k(\Tau), \mu_k\right)\\
	&=\mathbb{E} D_{\psi^*_{\mu_k}}\left(S_k(\Tau) / N_k(\Tau), \mu_k\right)  \\
	& = \int_{0}^{\infty} \mathbb{P}\left(D_{\psi^*_{\mu_k}}\left(S_k(\Tau) / N_k(\Tau), \mu_k\right) > \delta \right) \mathrm{d} \delta  ~~\text{(since the divergence is non-negative).}\\
	& \leq 2  \int_{0}^{\infty}\left[ \mathbb{E} \exp\left\{-\frac{q}{p} \delta N_k(\Tau)  \right\}\right]^{1/q} \mathrm{d} \delta ~~\text{(by the deviation inequality  \eqref{eq::deviation_p_q}).}\\
	& = 2 \frac{ep}{b}  \int_{0}^{\infty}\left[ \mathbb{E} \exp\left\{-\frac{q}{p} \left(N_k(\Tau) -b/e \right) \delta \right\}\right]^{1/q} \frac{b}{ep}  \exp\left\{-\frac{b}{ep}\delta\right\} \mathrm{d} \delta \\
	& := 2 \frac{ep}{b} \int_{0}^{\infty} \left[f(\delta)\right]^{1/q} p(\delta) \mathrm{d}\delta,
\end{align*}
where we have set $f(\delta) = \mathbb{E} \exp\left\{-\frac{q}{p} \left(N_k(\Tau) -b/e \right) \delta \right\}$ and $p(\delta) =  \frac{ep}{b}  \exp\left\{- \frac{ep}{b}\delta\right\}$.
Note that $p$ is the Lebesgue density of a probability measure on $[0, \infty)$. Since $\delta \mapsto \delta^{1/q}$ is a concave function on $[0, \infty)$, using Jensen's inequality we have that
\begin{align*}
	\int_{0}^{\infty} \left[f(\delta)\right]^{1/q} p(\delta) \mathrm{d}\delta \leq \left[ \int_{0}^{\infty} f(\delta)p(\delta) \mathrm{d}\delta \right]^{1/q}.
\end{align*}
Therefore, $\mathbb{E} D_{\psi^*_{\mu_k}}\left(\hat{\mu}_k(\Tau), \mu_k\right)$ can be further bounded by
\begin{align*}
	& 2 \frac{ep}{b} \left[\int_{0}^{\infty} f(\delta) p(\delta) \mathrm{d}\delta \right]^{1/q}\\
	& = 2 \frac{ep}{b}  \left[\int_{0}^{\infty} \mathbb{E} \exp\left\{-\frac{q}{p} \left(N_k(\Tau) -b/e \right) \delta \right\} \frac{b}{ep}  \exp\left\{-\frac{b}{ep}\delta\right\} \mathrm{d} \delta \right]^{1/q}\\
	& = 2 \left(\frac{ep}{b}\right)^{1/p}  \left[ \mathbb{E}\int_{0}^{\infty} \exp\left\{-\frac{q}{p} \left(N_k(\Tau) -\frac{b}{ep} \right) \delta \right\}   \mathrm{d} \delta \right]^{1/q} \\
	& = 2 \left(\frac{ep}{b}\right)^{1/p}  \left[ \mathbb{E} \frac{1}{\frac{q}{p} \left(N_k(\Tau) -\frac{b}{ep} \right) }\right]^{1/q}~~\left(\text{since $N_k(\Tau) \geq b > \frac{b}{ep}$}\right) \\
	& = 2 \left(\frac{e}{b}\right)^{1/p} p  \left[ \mathbb{E} \frac{1}{q \left(N_k(\Tau) -\frac{b}{ep} \right) }\right]^{1/q} \\
	&\leq2\left(\frac{e}{b}\right)^{1/p}  p  \left[ \mathbb{E} \frac{1}{N_k(\Tau)   }\right]^{1/q},
\end{align*}
where in the last inequality we have used the bound
\begin{align*}
	N_k(\Tau) -\frac{b}{ep}  = \frac{1}{p} \left(N_k(\Tau) -b/e\right) + \frac{N_k(\Tau)}{q} > \frac{N_k(\Tau)}{q}. 
\end{align*}
Thus, for any $p, q > 1$ with $1/p + 1/q = 1$, we have shown that
\begin{equation}
	\mathbb{E} D_{\psi^*_{\mu_k}}\left(\hat{\mu}_k(\Tau) , \mu_k\right)   \leq 2\left(\frac{e}{b}\right)^{1/p}  p  \left[ \mathbb{E} \frac{1}{N_k(\Tau)   }\right]^{1/q}.	
\end{equation}
Since the above bound holds for any $p, q > 1$ with $1/p + 1/q = 1$, by taking infimum over all $p>1$, we then have that
\begin{align*}
	\mathbb{E} D_{\psi^*_{\mu_k}}\left(\hat{\mu}_k(\Tau) , \mu_k\right)   &\leq 2 \inf_{p > 1}\left(\frac{e}{b}\right)^{1/p}  p  \left[ \mathbb{E} \frac{1}{N_k(\Tau)   }\right]^{1-1/p} 	 \\
	& =  \frac{2}{ n^{\eff}} \inf_{p > 1}  p  \left(\frac{en^{\eff}}{b}\right)^{1/p} \\
	& =  \frac{2}{ n^{\eff}}   \exp\left\{ \inf_{p > 1} \left[\log p + \frac{1}{p}\log(en^\eff / b)\right]\right\} \\
	& =  \frac{2}{ n^{\eff}}   \exp\left\{ \log \log(en^\eff / b) + 1\right\} \\
	& =  2e\frac{1 + \log (n_k^{\eff} / b) }{n_k^{\eff}},
\end{align*}
where the second equality is justified by the continuity of the exponential and logarithmic functions. The third equality follows from the fact that if $a \geq e$,  
\[
\log p + \frac{1}{p}\log a \geq \log \log a + 1,~~\forall p \geq 1,
\]
with equality if and only if $p = \log a$. Since, by assumption $N_k(\Tau) \geq b$, we have that $n_k^\eff \geq b$ and therefore we can set $p = \log(en^\eff / b) \geq 1$. Thus, the first part of the claimed bound on the risk in \eqref{eq::bound_on_Bregman} is proven.

To prove the second part of the upper bound, we use the deviation inequality in a slightly different way which is motivated from the proof of Theorem 12.1. in \cite{pena2008self}. Specifically, for any $\epsilon > 0$ and $r > 1$, we have
\begin{align*}
	&\mathbb{E} D_{\psi^*_{\mu_k}}\left(\hat{\mu}_k(\Tau), \mu_k\right)\\
	&=\mathbb{E} D_{\psi^*_{\mu_k}}\left(S_k(\Tau) / N_k(\Tau), \mu_k\right)  \\
	& = \int_{0}^{\infty} \mathbb{P}\left(D_{\psi^*_{\mu_k}}\left(S_k(\Tau) / N_k(\Tau), \mu_k\right) > \delta \right) \mathrm{d} \delta\\
	& \leq \epsilon + \int_{\epsilon}^{\infty} \mathbb{P}\left(D_{\psi^*_{\mu_k}}\left(S_k(\Tau) / N_k(\Tau), \mu_k\right) > \delta \right) \mathrm{d} \delta\\
	& \leq \epsilon + 2  \int_{\epsilon}^{\infty}\left[ \mathbb{E} \exp\left\{-\frac{q}{p} \delta N_k(\Tau)  \right\}\right]^{1/q} \mathrm{d} \delta ~~\text{(by the deviation inequality  \eqref{eq::deviation_p_q}).}\\
	& = \epsilon + 2  \int_{\epsilon}^{\infty} \delta^{-r}\left[ \mathbb{E}\delta^{qr} \exp\left\{-\frac{q}{p} \delta N_k(\Tau)  \right\}\right]^{1/q} \mathrm{d} \delta \\
	& \leq \epsilon + 2  \int_{\epsilon}^{\infty} \delta^{-r}\left[ \mathbb{E} \sup_{\tau > 0} \tau^{qr} \exp\left\{-\frac{q}{p} \tau N_k(\Tau)  \right\}\right]^{1/q} \mathrm{d} \delta .
\end{align*}
It can be easily checked that the supremum is achieved at $\tau = \frac{pr}{N_k(\Tau)}$. Therefore we have that
\[
\sup_{\tau > 0} \tau^{qr} \exp\left\{-\frac{q}{p} \tau N_k(\Tau)  \right\} = \left(\frac{pr}{N_k(\Tau)}\right)^{qr} e^{-qr},
\]
which implies that
\begin{align*}
	\mathbb{E} D_{\psi^*_{\mu_k}}\left(\hat{\mu}_k(\Tau), \mu_k\right)
	& \leq \epsilon + 2  \int_{\epsilon}^{\infty} \delta^{-r}\left[ \mathbb{E} \sup_{\delta > 0} \delta^{qr} \exp\left\{-\frac{q}{p} \delta N_k(\Tau)  \right\}\right]^{1/q} \mathrm{d} \delta\\
	& = \epsilon + 2  \int_{\epsilon}^{\infty} \delta^{-r}\left[ \mathbb{E} \left(\frac{pr}{N}\right)^{qr} e^{-qr}\right]^{1/q} \mathrm{d} \delta \\
	& = \epsilon + 2 \left(\frac{pr}{e}\right)^{r}\left[ \mathbb{E} \left(\frac{1}{N^{qr}}\right) \right]^{1/q} \int_{\epsilon}^{\infty} \delta^{-r} \mathrm{d} \delta \\
	& = \epsilon +  \frac{2}{(r-1)\epsilon^{r-1}} \left(\frac{pr}{en^{\eff, qr}}\right)^{r}.
\end{align*}
Since the above bound holds for any $\epsilon > 0$, by taking infimum on the RHS over $\epsilon >0$, we have the following upper bound. 
\[
\mathbb{E} D_{\psi^*_{\mu_k}}\left(\hat{\mu}_k(\Tau), \mu_k\right) \leq \frac{2^{1/r} pr^2}{e(r-1)} \frac{1}{n^{\eff, qr}}.
\]
By setting $r' = qr$, we can write the above inequality as
\begin{equation}
	\mathbb{E} D_{\psi^*_{\mu_k}}\left(\hat{\mu}_k(\Tau), \mu_k\right) \leq C_{q,r'} \frac{1}{n^{\eff, r'}},
\end{equation}
where 
\[
C_{q,r'} = \frac{2^{q/r'}}{e} \frac{r'^2}{(r'-q)(q-1)}.\]
Since this upper bound holds for any choice of $r' > q > 1$,  the second part of the upper bound is proved.

For general $\Tau$, let $\Tau_M := \min\left\{\Tau, M  \right\}$ for all $M \geq t_0$.  Since $\Tau_M$ is a stopping time with $T_M \geq t_0$ and $\mathbb{P}(\Tau_M \leq M) = 1$, we have that 
\begin{equation}\label{eq::risk_finite_tau}
	\mathbb{E} D_{\psi^*_{\mu_k}}\left(\hat{\mu}_k(\Tau_M) , \mu_k\right) \leq  \min\left\{2e\frac{1 + \log (n_k^{\eff\wedge M} / b) }{n_k^{\eff}}, \inf_{r>1} \frac{ C_{r}}{n_k^{\eff\wedge M, r}}\right\},~~\forall M \geq t_0,
\end{equation}
where $n_k^{\eff\wedge M,r}$ is the corresponding effective sample size  $\left[\mathbb{E}\left[1/N_k^r(\Tau_M)\right]\right]^{-r}$ with $n_k^{\eff\wedge M} = n_k^{\eff\wedge M,1}$. Then, we have
\begin{align*}
	\mathbb{E} D_{\psi^*_{\mu_k}}\left(\hat{\mu}_k(\Tau) , \mu_k\right) 
	& \leq \liminf_{M \rightarrow \infty} \mathbb{E} D_{\psi^*_{\mu_k}}\left(\hat{\mu}_k(\Tau_M) , \mu_k\right) \\
	& \leq \liminf_{M \rightarrow \infty} \min\left\{2e\frac{1 + \log (n_k^{\eff\wedge M} / b) }{n_k^{\eff}}, \inf_{r>1} \frac{ C_{r}}{n_k^{\eff\wedge M, r}}\right\} \\
	& \leq \min\left\{2e\frac{1 + \log (n_k^{\eff} / b) }{n_k^{\eff}}, \inf_{r>1} \frac{ C_{r}}{n_k^{\eff, r}}\right\},
\end{align*}
as desired, where the first inequality comes from  Fatous's lemma, the second one follows from the inequality \eqref{eq::risk_finite_tau} and the last one comes from the monotone convergence theorem along with the facts that $0 \leq 1 / N_k^r(\Tau_M)\leq 1/ b^r$ for all $M\geq t_0$ and that $\{1 / N_k^r(\Tau_M)\}_{M \geq t_0}$ is  a non-negative decreasing sequence converging to $1/N_k^r(\Tau)$ almost surely which also implies $n_k^{\eff\wedge M, r} \rightarrow n_k^{\eff, r}$ as $M \rightarrow \infty$. 

The proof of Theorem~\ref{thm::bound_on_Bregman} is completed. In the following subsection, we present a simple proof of Corollary~\ref{cor::bound_on_bias_l1_risk}.

\subsection{Proof of Corollary~\ref{cor::bound_on_bias_l1_risk}} \label{subSec::cor_bound_on_bias_l1_risk}
By the equation \eqref{eq::KL_equiv_psi_star}, we have that
\begin{align*}
	D_{\psi_{\mu_k}^*}\left(\hat{\mu}_k(\Tau), \mu_k\right) = \psi_{\mu_k}^*\left(\hat{\mu}_k(\Tau)\right) = \psi^*\left(\hat{\mu}_k(\Tau) - \mu_k\right).
\end{align*}
Since $\psi^*$ is convex,  applying the Jensen's inequality to the risk bound in the equation \eqref{eq::bound_on_Bregman} of Theorem~\ref{thm::bound_on_Bregman}, we get that
\begin{align*}
	\psi^*\left( \mathbb{E}\left[  \hat{\mu}_k(\Tau) - \mu_k\right] \right) 
	&\leq  \mathbb{E}\left[  \psi^*\left( \hat{\mu}_k(\Tau) - \mu_k\right) \right] \\
	& = \mathbb{E} D_{\psi_{\mu_k}^*} (\hat{\mu}_k(\Tau), \mu_k)  \leq U_{k,b}.
\end{align*}
If the bias $\mathbb{E}\left[  \hat{\mu}_k(\Tau) - \mu_k\right]$ is positive, $\psi^*$ can be replaced with $\psi_+^*$, which implies that
\[
\psi_+^*\left( \mathbb{E}\left[  \hat{\mu}_k(\Tau) - \mu_k\right] \right)  \leq U_{k,b}.
\]
Since $\psi_+^*$ is an increasing and invertible function, we get the desired upper bound on bias, namely
\[
\mathbb{E}\left[\hat{\mu}_k(\Tau) - \mu_k\right]  \leq  {\psi_+^*}^{-1}\left( U_{k,b}\right).
\]
Applying the same argument to the case of a negative bias, we arrive at the analogous lower bound
\[
- {\psi_-^*}^{-1}\left( U_{k,b}\right) \leq \mathbb{E}\left[\hat{\mu}_k(\Tau) - \mu_k\right].
\]
This completes the proof of the expression \eqref{eq::bound_on_bias}. 

If $\psi^*$ is symmetric around zero,  $\psi^*(z) = \psi_+^*(|z|)$ for all  $z \in \Lambda^*$.	Therefore, by the same steps, 
\begin{align*}
	\psi_+^*\left( \mathbb{E}\left|  \hat{\mu}_k(\Tau) - \mu_k\right| \right) 
	&\leq  \mathbb{E}\left[  \psi_+^*\left( \left|\hat{\mu}_k(\Tau) - \mu_k\right|\right) \right] \\
	& = \mathbb{E}\left[  \psi^*\left( \hat{\mu}_k(\Tau) - \mu_k\right) \right] \\
	& = \mathbb{E} D_{\psi_{\mu_k}^*} (\hat{\mu}_k(\Tau), \mu_k)  \leq U_{k,b}.
\end{align*}
Applying ${\psi_+^*}^{-1}$ to the both sides, we arrive at the bound on the expected $\ell_1$ loss given in \eqref{eq::bound_on_l1_risk} which completes the proof.

\section{Proofs of Theorem~\ref{thm::bounds_in_fully_adap} and related statements}

In this section, Theorem~\ref{thm::bounds_in_fully_adap} and related statements are proved. Recall that we assume that there exists a time $t_0$ such that, almost surely, $\Tau \geq \tau \geq t_0$ and $N_k(t_0) \geq b \geq 3$  for all $ k \in [{\K}]$.
Let $h :[\log b,\infty) \rightarrow [1, \infty)$ be a non-decreasing function such that
$\sum_{j = 1}^{\infty} \frac{1}{h(\log b+j)} \leq 1$ and 
$v \in [b, \infty) \mapsto \frac{\log h\left(\log v \right)}{v}$ is non-increasing. 
It can be easily checked that the function $h_b(x) := x^2 /\log b$  satisfies the condition above.

For ease of readability, we drop the subscript $k$ throughout this section. We first provide the proof to the adaptive deviation inequality of Lemma~\ref{lemma::adaptive_deviation_ineq}, which is a fundamental component of the proof of Theorem~\ref{thm::bounds_in_fully_adap} and related statements.

\subsection{Proof of Lemma~\ref{lemma::adaptive_deviation_ineq}} \label{subSec::proof_of_adaptive_ineq}

The proof strategy of the adaptive deviation inequality \eqref{eq::adaptive_deviation_ineq}  is based on splitting the deviation event into simpler sub-events and then find exponential bounds for the probability of each sub-event. In detail, for each $t \geq 0$ and $j \in \log b + \mathbb{N} := \{\log b +i : i \in \mathbb{N}\}$, define the events
\begin{align*}
	F_t &:= \left\{  N(t) \geq b, N(t)\psi^*(\hat{\mu}_t - \mu)  > e\left(\delta + \log h\left( \log N(t)\right)\right)  \right\}, \\
	G_t &:= \left\{ \hat{\mu}(t) \geq \mu \right\}, \text{ and}\\
	H_t^j & := \left\{ e^{j-1} \leq N(t) < e^j \right\}.
\end{align*}
To bound the probability of these events, we rely on the following result, which we establish using arguments borrowed from the proof of Theorem 11 in \cite{garivier2011kl}, See also \cite{garivier2013informational}.
\begin{lemma} \label{lemma::1}
	Let $h :[\log b,\infty) \rightarrow [1, \infty)$ be a non-decreasing function which makes $v \mapsto \frac{\log h\left(\log v \right)}{v}$ non-increasing on $[b, \infty)$.   
	Then, for any fixed $\delta >0$, there exist a deterministic $\lambda_j \geq 0$ such that
	\begin{equation}
		\left\{ F_t \cap G_t \cap H_t^j \right\} \subset  \left\{\lambda_j \left[S(t) -\mu N(t)\right]- \psi(\lambda_j) N(t) \geq \delta + \log h(j)\right\},
	\end{equation}
	and  a deterministic $\lambda_j' < 0$ such that
	\begin{equation}
		\left\{ F_t \cap G_t^c \cap H_t^j \right\} \subset  \left\{\lambda_j^\prime  \left[S(t) -\mu N(t)\right]- \psi(\lambda_j^\prime ) N(t) \geq \delta + \log h(j)\right\}.
	\end{equation}
\end{lemma}

\begin{proof}[Proof of Lemma~\ref{lemma::1}]
	On the event $F_t \cap G_t \cap H_t^j$, since $b \leq e^{j-1} \leq N(t) < e^j$ and $v \mapsto\frac{\log h\left(\log v \right)}{v}$ is non-increasing on $[b, \infty)$, we have that 
	\begin{equation} \label{eq::control_psi_star}
		\psi^* \left(\hat{\mu}(t) -\mu\right) > \frac{e \delta}{N(t)} + \frac{e \log h\left( \log N(t)\right)}{N(t)} 
		\geq \frac{\delta}{e^{j-1}} +  \frac{ \log h(j)}{e^{j-1}}  > 0.
	\end{equation}
	
	Since, by assumption, $\psi$ is a non-negative  convex function such that $\psi(0) = \psi^\prime (0) = 0$. its convex conjugate $\psi^*$ is an increasing function on $[0, \infty)$ with $\psi^*(0) = 0$. Therefore, we can pick a deterministic real number $z_j \geq 0$ such that 
	\[
	\psi^*(z_j) = \frac{\delta}{e^{j-1}} +  \frac{ \log h(j)}{e^{j-1}} .
	\]
	Note that, since  $\hat{\mu}(t)  -\mu \geq 0$ and $\psi^*$ is an increasing function on $[0, \infty)$, our choice of $z_j$ along with the inequalities in \eqref{eq::control_psi_star}  implies that $\hat{\mu}(t)  -\mu \geq z_j$, on the event $F_t \cap G_t \cap H_t^j$.
	
	Let $\lambda_j$ be the convex conjugate of $z_j$ with respect to $\psi$,  which is given by  
	\[
	\lambda_j = \argmax\limits_{\lambda \in \Lambda} \lambda z_j - \psi(\lambda) = \psi^*(z_j).
	\]
	
	Since $z_j\geq 0$,  $\lambda_j$ is also non-negative. Therefore, on the event  $F_t \cap G_t \cap H_t^j$, we have that
	\begin{align*}
		\lambda_j \left[\hat{\mu}(t) -\mu\right] - \psi(\lambda_j) & \geq \lambda_j z_j - \psi(\lambda_j) \\
		& = \psi^{*} (z_j)\\
		& = \frac{\delta}{e^{j-1}}+ \frac{\log h(j)}{e^{j-1}} \\
		& \geq  \frac{\delta}{N(t)}+ \frac{\log h(j)}{N(t)} .
	\end{align*} 	
	Re-arranging, we get that
	\[
	\lambda_j \left[S(t) -\mu N(t)\right]- \psi(\lambda_j) N(t) \geq \delta + \log h(j).
	\]
	which proves the first statement in the lemma. For the second statement, since $\psi^*$ is a decreasing function on $(-\infty, 0]$ with $\psi^*(0) = 0$ we can pick a deterministic real number $z_j' < 0$ such that 
	\[
	\psi^*(z_j') = \frac{\delta}{e^{j-1}} +  \frac{ \log h(j)}{e^{j-1}} .
	\]
	Note that, since  $\hat{\mu}(t)  -\mu < 0$ and $\psi^*$ is an decreasing function on $(-\infty, 0]$, the choice of $z_j'$ and the inequalities in \eqref{eq::control_psi_star} yield that $\hat{\mu}(t)  -\mu < z_j$, on the event $F_t \cap G_t^c \cap H_t^j$. 	Let $\lambda_j^\prime$ be the convex conjugate of $z_j^\prime$. Since $z_j^\prime < 0$, it is also the case that $\lambda_j^\prime <0$. Then, by the same argument used for $\hat{\mu}(t) -\mu \geq 0$ case,  the second statement holds which completes the proof.
\end{proof}
Now, we continue to prove Lemma~\ref{lemma::adaptive_deviation_ineq}.	In Fact~\ref{fact::KL_equiv_psi_star}, we showed that $D_{\psi_{\mu}^*}(\hat{\mu}, \mu) = \psi^* (\hat{\mu} - \mu)$. 	Thus, the adaptive deviation inequality \eqref{eq::adaptive_deviation_ineq} can be re-stated in the general setting as the inequality 
\begin{equation}\label{eq::general_adap_deviation}
	\begin{aligned}
		\mathbb{P}\left(\frac{N(\tau)}{\log h\left(\log N(\tau)\right)}\psi^*\left(\hat{\mu}(\tau) -\mu\right) \geq C_{h,b} \delta\right) 
		\leq 2\exp\left\{-\delta\right\},~~\forall \delta \geq 1,
	\end{aligned}
\end{equation}
where $C_{h,b} := e\left(1 + \frac{1}{\log h\left(\log b\right)}\right)$, which can be derived from the following more general inequality.
\begin{equation}
	\begin{aligned}
		&\mathbb{P}\left(N(\tau)\psi^*\left(\hat{\mu}(\tau) -\mu\right) \geq e\left(\delta + \log h\left( \log N(\tau)\right)\right)\right) \\
		&\leq 2\exp\left\{-\delta\right\},~~\forall \delta \geq 0.
	\end{aligned}
\end{equation}
To show the above bound, it is sufficient to show that  the inequality holds uniformly for all time. (e.g., see Lemma 3 in \cite{howard2021time}). Therefore, in this proof, we prove the following uniform concentration inequality: 
\begin{equation}\label{eq::adap_deviation_simple_form}
	\begin{aligned}
		&\mathbb{P}\left(\exists t \in \mathbb{N} : N(t) \geq b, N(t)\psi^*\left(\hat{\mu}(t) -\mu\right) \geq e\left(\delta + \log h\left( \log N(t)\right)\right)\right)\\ 
		&\leq 2\exp\left\{-\delta\right\},~~\forall \delta \geq 0.
	\end{aligned}
\end{equation}				
The event on the left-hand side of \eqref{eq::adap_deviation_simple_form} is equal to $\bigcup_{t = 1}^\infty F_t$, and its probability can be bounded as follows:
\begin{align*}
	\mathbb{P}\left( \bigcup_{t = 1}^\infty F_t \right) &= \mathbb{P} \left( \bigcup_{t = 1}^\infty \bigcup_{j \in \log b + \mathbb{N}} \left[ H_t^{j} \cap F_t \cap \left( G_t \cup G_t^c \right) \right]\right) \\
	&= \mathbb{P} \left( \left[\bigcup_{t = 1}^\infty \bigcup_{j \in \log b + \mathbb{N}} \left( H_t^{j} \cap F_t \cap  G_t \right) \right] \cup \left[ \bigcup_{t = 1}^\infty \bigcup_{j \in \log b + \mathbb{N}} \left( H_t^{j} \cap F_t \cap  G_t^c \right)\right]\right) \\
	& \leq \mathbb{P} \left( \bigcup_{t = 1}^\infty \bigcup_{j \in \log b + \mathbb{N}} \left( H_t^{j} \cap F_t \cap  G_t \right) \right) + \mathbb{P} \left( \bigcup_{t = 1}^\infty \bigcup_{j \in \log b + \mathbb{N}} \left( H_t^{j} \cap F_t \cap  G_t^c \right) \right).
\end{align*}
For each $\lambda$ and $t$, define $L_t(\lambda) := \exp\left\{ \lambda\left( S(t)- \mu N(t) \right) - N(t)\psi(\lambda)  \right\}$. We show in the proof of Lemma~\ref{lemma::deviation_ineq},  that, for each $\lambda \in \Lambda$, $\{L_t(\lambda)\}_{t \geq 0}$ is a non-negative super-martingale with $\mathbb{E}L_0(\lambda) = 1$. By Lemma~\ref{lemma::1}, we have that
\begin{align*}
	&\mathbb{P} \left( \bigcup_{t = 1}^\infty \bigcup_{j =\log b+\mathbb{N}}^\infty \left( H_t^{j} \cap F_t \cap  G_t \right) \right)  \\
	&\leq \mathbb{P}\left(\exists t \geq 0, \exists j  \in \log b + \mathbb{N} :  \lambda_j \left[S(t)-\mu N(t)\right] - \psi(\lambda_j) N(t) > \delta + \log h(j)  \right) ~~\text{(Lemma~\ref{lemma::1})}\\
	& = \mathbb{P}\left(\exists t \geq 0, \exists j  \in \log b + \mathbb{N} :  \exp\left(\lambda_j \left[S(t)-\mu N(t)\right] - \psi(\lambda_j) N(t) \right)> h(j) e^{\delta}   \right) \\
	& = \mathbb{P}\left(\exists t \geq 0, \exists j  \in \log b + \mathbb{N} :  L_t(\lambda_j)> h(j) e^{\delta}   \right) \\
	& \leq \sum_{j  \in \log b + \mathbb{N}} \mathbb{P}\left( \sup_{t \geq 0} L_t(\lambda_j)> h(j) e^{\delta}   \right) ,
\end{align*}
where the last inequality stems from the union bound. 

Since $\{L_t\}_{t \geq 0}$ is a non-negative super-martingale
with $L_0 =1$, by applying  Ville's maximal inequality \cite{villeetude}, we conclude that
\begin{align*}
	\sum_{j  \in \log b + \mathbb{N}} \mathbb{P}\left( \sup_{t \geq 0} L_t(\lambda_j)> h(j) e^{\delta}   \right)  
	\leq e^{-\delta}  \sum_{j=1}^\infty \frac{1}{h(\log b+j)} 
	\leq e^{-\delta}.
\end{align*}
Similarly, it can be shown
\[
\mathbb{P} \left( \bigcup_{t = 1}^\infty \bigcup_{j \in \log b + \mathbb{N}} \left( H_t^{j} \cap F_t \cap  G_t^c \right) \right) 
\leq e^{-\delta}.
\]
By combining  two bounds, we get that
\begin{align*}
	\mathbb{P}\left(\exists t \geq 0 : N(t) \geq b, N(t)\psi^*(\hat{\mu}_t - \mu)  > e\left(\delta + \log h\left( \log N(t)\right)\right)  \right)
	\leq 2  e^{-\delta},
\end{align*}	
which implies the desired bound on the adaptive deviation probability in \eqref{eq::general_adap_deviation}.

\subsection{Proof of Theorem~\ref{thm::bounds_in_fully_adap}} \label{subSec::Proof_of_second_them}
The proof of Theorem~\ref{thm::bounds_in_fully_adap} relies on the Donsker-Varadhan representation of the KL divergence and arguments in \cite{russo2016controlling, jiao2017dependence}. For completeness, we cite the following form of Donsker-Varadhan representation theorem (see, e.g.,\cite{donsker1983asymptotic, jiao2017dependence}): 
\begin{lemma}
	Let $P$, $Q$ be probability measures on $\mathcal{X}$ and let $\mathcal{C}$ denote the set of functions $f : \mathcal{X} \mapsto \mathbb{R}$ such that $\mathbb{E}_Q \left[e^{f(X)}\right] < \infty$. If $D_{KL}(P || Q) < \infty$ then for every $f \in \mathcal{C}$ the expectation $\mathbb{E}_P \left[f(X)\right]$ exists and furthermore 
	\begin{equation}
		D_{KL} (P || Q)  = \sup_{f \in \mathcal{C}} \mathbb{E}_P \left[f(X)\right] - \log \mathbb{E}_Q \left[e^{f(X)}\right] ,
	\end{equation}
	where the supremum is attained when $f =\log \frac{\mathrm{d}P}{\mathrm{d} Q}$.
\end{lemma}

To apply Donsker-Varadhan representation, we need the following bound on the expectation of the exponentiated stopped adaptive process, which is based on the adaptive deviation inequality in Lemma~\ref{lemma::adaptive_deviation_ineq}.
\begin{claim} \label{claim::martingale_adap_ineq}
	Under the assumptions of Theorem~\ref{thm::bound_on_Bregman}, 	for each $k \in[{\K}]$ we have that
	\begin{equation} \label{eq::martingale_adap_ineq}
		\mathbb{E} \exp\left\{\frac{1}{2C_{h,b}} \left[\frac{N_{k}(\tau)}{\log h\left(\log N_k(\tau)\right)}  D_{\psi_{\mu_k}^*}(\hat{\mu}_k, \mu_k)  \right]\right\} \leq 3.46.
	\end{equation}
\end{claim}

\begin{proof}[Proof of Claim~\ref{claim::martingale_adap_ineq}]
	Using the inequality
	\begin{equation} \label{eq::Fubini_exp} 
		\begin{aligned}
			\mathbb{E}e^X =\mathbb{E}\int_{-\infty}^{\infty}\mathbbm{1}( \delta < X)e^\delta\mathrm{d}\delta  &= \int_{-\infty}^{\infty}\mathbb{P}(X > \delta) e^\delta\mathrm{d}\delta \\
			&\leq \int_{1}^{\infty}\mathbb{P}(X > \delta) e^\delta\mathrm{d}\delta  + \int_{-\infty}^{1} e^\delta\mathrm{d}\delta \\
			&= \int_{1}^{\infty}\mathbb{P}(X > \delta) e^\delta\mathrm{d}\delta  + e,
		\end{aligned}
	\end{equation}
	we can bound the left hand side of \eqref{eq::martingale_adap_ineq} as follows:  
	\begin{align*}
		&\mathbb{E} \exp\left\{ \frac{1}{2C_{h,b}} \left[\frac{N_{k}(\tau)}{\log h\left(\log N_k(\tau)\right)}  D_{\psi_{\mu_k}^*}(\hat{\mu}_k, \mu_k)  \right]\right\} \\
		& \leq e+\int_{1}^{\infty} \mathbb{P}\left( \frac{1}{2C_{h,b}} \left[\frac{N_{k}(\tau)}{\log h\left(\log N_k(\tau)\right)}  D_{\psi_{\mu_k}^*}(\hat{\mu}_k, \mu_k)  \right] > \delta \right) e^\delta \mathrm{d} \delta \\
		& =  e + \int_{1}^{\infty} \mathbb{P}\left( \frac{N_{k}(\tau)}{\log h\left(\log N_k(\tau)\right)}  D_{\psi_{\mu_k}^*}(\hat{\mu}_k, \mu_k) >  2C_{h,b} \delta\right) e^\delta \mathrm{d} \delta.
	\end{align*}
	By applying the adaptive deviation inequality \eqref{eq::adaptive_deviation_ineq}  Lemma~\ref{lemma::adaptive_deviation_ineq}, the last term can be further bounded by $ e + \int_{1}^{\infty} 2 e^{-\delta} \mathrm{d} \delta = e + 2/e < 3.46$. The claimed result readily follows.
\end{proof}

Coming back to the proof of Theorem~\ref{thm::bounds_in_fully_adap},
for each $k \in [{\K}]$, set $   P_k = \mathcal{L}\left(\D_{\Tau} | {\kappa} = k\right)$, $Q = \mathcal{L}\left(\D_{\Tau}\right)$ and 
\[
f_k = \frac{1}{2C_{h,b}} \left[\frac{N_{k}(\tau)}{\log h\left(\log N_k(\tau)\right)}  D_{\psi_{\mu_k}^*}(\hat{\mu}_k, \mu_k)  \right].
\]
Then, from the Donsker-Varadhan representation, we can lower bound the mutual information between the adaptive query $\kappa$ and the data $\D_{\Tau}$ in the following way:
\begin{align*}
	I(\kappa ; \D_{\Tau}) &= \sum_{k =  1}^{\K} \mathbb{P}(\kappa = k)  D_{KL}\left(\mathcal{L}\left(\D_{\Tau} | {\kappa} = k\right) ||\mathcal{L}\left(\D_{\Tau}\right) \right) \\
	& \geq  \sum_{k =  1}^{\K} \mathbb{P}(\kappa = k)  \mathbb{E}_{P_k} \left[f_k\right] - \log \mathbb{E}_{Q} \left[e^{f_k}\right]   \\	
	& = \sum_{k = 1}^{\K} \mathbb{P}(\kappa = k) \left\{ \frac{1}{2C_{h,b}} \mathbb{E} \left[\frac{N_{k}(\tau)}{\log h\left(\log N_k(\tau)\right)}  D_{\psi_{\mu_k}^*}(\hat{\mu}_k, \mu_k)    \mid \kappa = k \right] \right.\\
	&~~~~~~\left. - \log \mathbb{E}\left[ \exp\left\{\frac{1}{2C_{h,b}} \left[\frac{N_{k}(\tau)}{\log h\left(\log N_k(\tau)\right)}  D_{\psi_{\mu_k}^*}(\hat{\mu}_k, \mu_k)  \right]\right\}\right]   \right\} \\
	& \geq \sum_{k = 1}^{\K} \mathbb{P}(\kappa = k)  \left\{\frac{1}{2C_{h,b}}\mathbb{E}  \left[\frac{N_{k}(\tau)}{\log h\left(\log N_k(\tau)\right)}  D_{\psi_{\mu_k}^*}(\hat{\mu}_k, \mu_k)   \mid \kappa = k \right] -\log 3.46 \right\}\\
	& \geq  \frac{1}{2C_{h,b}}\mathbb{E}  \left[\frac{N_{\kappa}(\tau)}{\log h\left(\log N_\kappa(\tau)\right)}  D_{\psi_{\mu_\kappa}^*}(\hat{\mu}_\kappa, \mu_\kappa)   \right]-1.25,
\end{align*}
where  the second inequality is due to the inequality~\eqref{eq::martingale_adap_ineq} in Claim~\ref{claim::martingale_adap_ineq}. The  risk bound \eqref{eq::adap_risk} now follows from rearranging with $h(x) = x^2 / \log b$ and $C_b := 4C_{h,b}$ which completes the proof. 

The only remaining proof in this section is that of Corollary~\ref{cor::bounds_in_fully_adap}, which we present below. 
\subsection{Proof of Corollary~\ref{cor::bounds_in_fully_adap}} \label{subSec::proof_of_bound_in_fully_adap}
For any $p, q >1$ with $\frac{1}{p} + \frac{1}{q} = 1$,  H\"older's inequality along with the bound on the adaptive risk in \eqref{eq::adap_risk}  implies that
\begin{align*}
	\left[\mathbb{E} D_{\psi_{\mu_\kappa}^*}^{1/p}(\hat{\mu}_\kappa, \mu_\kappa) \right]^p 
	&= \left[\mathbb{E} \left(\frac{N_{\kappa}(\tau)}{\log \log N_{\kappa}(\tau)}\right)^{-1/p} \left(\frac{N_{\kappa}(\tau)}{\log \log N_{\kappa}(\tau)}\right)^{1/p} D_{\psi_{\mu_\kappa}^*}^{1/p}(\hat{\mu}_\kappa, \mu_\kappa) \right]^p\\
	&\leq \left[\mathbb{E} \left(\frac{N_{\kappa}(\tau)}{\log \log N_{\kappa}(\tau)}\right)^{-q/p} \right]^{p/q}\mathbb{E}\left[  \frac{N_{\kappa}(\tau)}{\log \log N_{\kappa}(\tau)} D_{\psi_{\mu_\kappa}^*}(\hat{\mu}_\kappa, \mu_\kappa) \right] \\  
	& \leq \left[\mathbb{E} \left(\frac{N_{\kappa}(\tau)}{\log \log N_{\kappa}(\tau)}\right)^{-q/p} \right]^{p/q}C_b \left[ I(\kappa ; \D_{\Tau}) + 1.25 \right]\\
	& = \frac{C_b}{\dbtilde{n}^{\eff,q/p}} \left[ I(\kappa ; \D_{\Tau}) + 1.25 \right].
\end{align*}
By setting $r := 1/p$, we proves the  inequality  \eqref{eq::r-norm_bound}, as desired.

\section{Proofs of propositions and facts}\label{sec::proofs_of_props}
In this section, we provide formal proofs of propositions and facts which are omitted in the main text.

\subsection{Proof of Proposition~\ref{prop::consistency}} \label{Appen::prop::consistency}

We first formally describe Theorem 2.1 and 2.2 in \cite{gut2009stopped} which provide sufficient conditions for the consistency of the sample mean.
\begin{fact}[Theorem 2.1 and 2.2 in \cite{gut2009stopped}] \label{fact::thms_gut2009}
Suppose that 
\begin{equation}
    Y_n  \overset{a.s.}{\to} Y~~\text{as}~~n \to \infty~~\text{and}~~N(t) \overset{a.s.}{\to} \infty~~\text{as}~~t \to \infty.
\end{equation}
Then
\begin{equation}
      Y_{N(t)}  \overset{a.s.}{\to} Y~~\text{as}~~t \to \infty.
\end{equation}
This statement also holds if the almost sure convergence is replaced with the convergence in probability. 
\end{fact}
For the completeness of the presentation, we provide the fact as follows:
\begin{proof}[Proof of Theorem 2.1 and 2.2 in \cite{gut2009stopped}]
    Define events $E, F$ and $G$ as
    \begin{align}
        E &:= \left(Y_n \to Y~~\text{as}~~n \to \infty \right) \\
        F &:= \left(N(t) \to \infty~~\text{as}~~t \to \infty \right) \\
        G &:= \left(Y_{N(t)} \to Y~~\text{as}~~t \to \infty \right).
    \end{align}
Then, $E\cap F \subset G$ which implies $1 = \mathbb{P}(E \cap F) \leq \mathbb{P}(G)$. This proves the statement in \cref{fact::thms_gut2009} with the almost sure convergence. From the standard subsequence argument, it also proves the statement in \cref{fact::thms_gut2009} with the convergence in probability, as desired.
\end{proof}
Now, based on \cref{fact::thms_gut2009}, we prove Proposition~\ref{prop::consistency}.

For any $k \in [{\K}]$, the strong law of large numbers and Theorem 2.1. in \cite{gut2009stopped} implies that
\begin{equation} \label{eq::consistency_fixed_k}
	\begin{aligned}
		\text{if}~~N_k(\tau_t) \overset{a.s.}{\to} \infty~~\text{as}~~t \to \infty ~~ \text{then}~~ \hat{\mu}_k(\tau_t) \overset{a.s.}{\to} \mu_k~~\text{as}~~t \to \infty.
	\end{aligned}
\end{equation}
For each $k \in [{\K}]$, define events $E_k$ and $F_k$ such that
\begin{align}
	E_k &= \left( \hat{\mu}_k(\tau_t) \to \mu_k~~\text{as}~~t \to \infty  \right), \\
	F_k &= \left(N_k(\tau_t) \to \infty~~\text{as}~~t \to \infty\right).
\end{align}
The statement \eqref{eq::consistency_fixed_k} implies that $\mathbb{P}(E_k \cup F_k^c) = 1$. If not, suppose $\mathbb{P}(F_k) = 1$, then $0 < \mathbb{P}(E_k^c \cap F_k) =  \mathbb{P}(E_k^c)$ which contradicts to the statement $\mathbb{P}(E_k) =1$. Hence we also have $\mathbb{P}(D) = 1$ where $D := \bigcap_{k\in [{\K}]} E_k \cup F_k^c$.

Now, we prove that, for any random sequence $\left\{\kappa_{\tau_t} \in [{\K}]\right\}$, 
\begin{equation} \label{eq::consistency_chosen_k}
	\begin{aligned}
		\text{if}~~ N_{\kappa_{\tau_t}}(\tau_t) \overset{a.s.}{\to} \infty~~\text{as}~~t \to \infty  ~~ \text{then}~~\hat{\mu}_{\kappa_{\tau_t}}(\tau_t) - \mu_{\kappa_{\tau_t}}\overset{a.s.}{\to} 0~~\text{as}~~t \to \infty.
	\end{aligned}
\end{equation}
For notational simplicity, let $Y_k(t) := \hat{\mu}_{k}(\tau_t)$, $M_k(t) := N_k(\tau_k)$, and $C_t := \kappa_{\tau_t}$. 
First, define events $G$ and $H$ such that 
\begin{align}
	G &= \left( Y_{C_t}(t) \to \mu_{C_t}~~\text{as}~~t \to \infty  \right), \\
	H &= \left(M_{C_t}(t)) \to \infty~~\text{as}~~t \to \infty\right).
\end{align}
Note that
\begin{align*}
	& \left|Y_{C_t}(t) - \mu_{C_t}  \right| \to  0, \\
	& \Leftrightarrow \sum_{k=1}^{\K}\mathbbm{1}\left(C_t = k\right)\left|Y_k(t) - \mu_k  \right| \to  0, \\
	& \Leftrightarrow \forall k \in [{\K}],  \mathbbm{1}\left(C_t = k\right)\left|Y_k(t) - \mu_k  \right| \to 0,\\
	& \Leftrightarrow \forall k \in [{\K}],  \mathbbm{1}\left(C_t = k\right) \to 0 ~~\text{or}~~ \left|Y_k(t) - \mu_k  \right| \to 0.
\end{align*}
Hence, if $\left|Y_{C_t}(t) - \mu_{C_t}  \right| \not\to 0$, there exists $k \in [{\K}]$ such that
\[
\mathbbm{1}(C_t =k) \not\to 0 ~~\text{and}~~\left|Y_{k}(t) - \mu_{k}  \right| \not\to 0.
\]
Under the event $D$, it further implies that there exists $k \in [{\K}]$ such that 
\[
\mathbbm{1}(C_t =k) \not\to 0 ~~\text{and}~~M_{k}(t) \not\to \infty,
\]
which also implies $M_{C_t}(t) \not\to \infty$. Hence, we have $D \cap G^c \subset H^c$ which is equivalent to $H \subset D^c \cup G$. Since $\mathbb{P}(D) = 1$,  if $\mathbb{P}(H) = 1$ we have
\[
1 = \mathbb{P}(H) \leq \mathbb{P}(D^c \cup G) \leq \mathbb{P}(D^c) + \mathbb{P}(G) = \mathbb{P}(G),
\]
which proves the claimed statement~\eqref{eq::consistency_chosen_k}. From the standard subsequence argument, it also implies that
\begin{equation} \label{eq::consistency_chosen_k_in_prob}
	\begin{aligned}
		\text{if}~~  N_{\kappa_{\tau_t}}(\tau_t) \overset{p}{\to} \infty~~\text{as}~~t \to \infty ~~ \text{if}~~\hat{\mu}_{\kappa_{\tau_t}}(\tau_t) - \mu_{\kappa_{\tau_t}}\overset{p}{\to} 0~~\text{as}~~t \to \infty,
	\end{aligned}
\end{equation}
as desired.

\subsection{Proof of Proposition~\ref{prop::minimax_nonadaptive}}
\label{Appen::subSec::minimax_nonadap}

For any fixed nonadaptive sampling scheme $\nu \in \mathbb{V}$, stopping time $T \in \mathbb{T}$ satisfying $N_k(T) \geq 1$, and for Gaussian arms with mean $\mu_1, \dots, \mu_{\K}$ and variance $\sigma^2$, the likelihood function of given data $\D_T = \left\{A_1, Y_1, \dots, A_T, Y_T \right\}$ with respect to $\mu:=(\mu_1, \dots, \mu_{\K})$ is proportional to the following expression:
\begin{align*}
	P(\D_T | \mu) &\propto\prod_{t=1}^T \nu_t\left(A_t | \D_{t-1}\right) p_{A_t}\left(Y_t |\mu_{A_t} \right)  \\
	& \propto \left[\prod_{t=1}^T \nu_t\left(A_t | \D_{t-1}\right)\right] \exp\left\{-\frac{1}{2\sigma^2}\sum_{t=1}^T \left(Y_t - \mu_{A_t} \right)^2 \right\}\\
	& = \left[\prod_{t=1}^T \nu_t\left(A_t | \D_{t-1}\right)\right] \exp\left\{-\frac{1}{2\sigma^2} \sum_{t=1}^T \sum_{k=1}^{\K} \mathbbm{1}\left(A_t = k\right)\left(Y_t - \mu_{k} \right)^2 \right\} 
\end{align*}

Now, put an independent Gaussian prior with precision $\rho >0$ to each $\mu_k$, that is, $\pi(\mu_k) \propto \exp\left\{-\frac{\rho}{2}\mu_k^2  \right\},~~\forall k \in [{\K}].$ Then, the posterior distribution of $\mu$ can be expressed as follows:
\begin{align*}
	\pi\left(\mu |\D_T\right)
	& \propto \left[\prod_{t=1}^T \nu_t\left(A_t | \D_{t-1}\right)\right] \exp\left\{-\frac{1}{2\sigma^2} \sum_{t=1}^T \sum_{k=1}^{\K} \mathbbm{1}\left(A_t = k\right)\left(Y_t - \mu_{k} \right)^2  \right\} \exp\left\{ -\frac{\rho}{2}\sum_{k=1}^{\K} \mu_k^2 \right\}\\
	& \propto \prod_{k=1}^{\K} \exp\left\{-\frac{1}{2\sigma^2} \sum_{t=1}^T \mathbbm{1}\left(A_t = k\right)\left(\mu_{k}^2 - 2Y_t \mu_k\right) -\frac{\rho}{2}\mu_k^2\right\}\\
	&= \prod_{k=1}^{\K}\exp\left\{-\frac{1}{2\sigma^2}\left[\left(N_k(T) + \rho\sigma^2\right)\mu_k^2 - 2N_k(T)\overline{Y}_k(T)\mu_{k} \right]  \right\}\\
	&\propto \prod_{k=1}^{\K}\exp\left\{-\frac{N_k(T) + \rho\sigma^2}{2\sigma^2}\left(\mu_k - \frac{N_k(T)}{N_k(T) + \rho/\sigma^2}\overline{Y}_k(T) \right)^2  \right\},
\end{align*}
where $\overline{Y}_k(T)$ is the sample average of observations from $k$-th arm. Therefore, the posterior distribution of $(\mu_1, \dots, \mu_k)$ is coordinate-wisely independent given data and, for each arm, the posterior distribution is given as 
\[
\mu_k | D_T \sim N\left(\frac{N_k(T)}{N_k(T) + \rho/\sigma^2}\overline{Y}_k(T), \frac{\sigma^2}{N_k(T) + \rho\sigma^2}\right).
\]

Hence, for each $k$, $\hat{\mu}_k^B := \frac{N_k(T)}{N_k(T) + \rho/\sigma^2}\overline{Y}_k(T)$ is the Bayes estimator for $\ell_2$ loss under the Gaussian prior with precision $\rho$. Denote $P^N | \mu$ be the distribution of the data $\D_T$ under Gaussian arms with  mean $\mu = (\mu_1, \dots, \mu_{\K})$ and variance $\sigma^2$, and let $P^{N,\pi}$ be the posterior predictive distribution under a prior $\pi$ on $\mu$. Then, we have the following lower bound on the Bayes risk.
\begin{align*}
	\inf_{\tilde{\mu}_k}\mathbb{E}_{\mu\sim \pi} \mathbb{E}_{\D_T \sim P^N|\mu} N_k(T)\left(\tilde{\mu}_k -\mu_k \right)^2 
	& = \inf_{\tilde{\mu}_k}\mathbb{E}_{\D_T \sim P^{N, \pi}}  N_k(T) \mathbb{E}_{\mu\sim \pi(\cdot|\D_T)}\left(\tilde{\mu}_k -\mu_k \right)^2 \\
	& \geq \mathbb{E}_{\D_T \sim P^{N, \pi}}N_k(T) \inf_{\tilde{\mu}_k} \mathbb{E}_{\mu\sim \pi(\cdot|\D_T)} \left(\tilde{\mu}_k -\mu_k \right)^2 \\
	& =  \mathbb{E}_{\D_T \sim P^{N, \pi}}  N_k(T)\mathbb{E}_{\mu\sim \pi(\cdot|\D_T)} \left(\tilde{\mu}_k^B -\mu_k \right)^2 \\
	& \geq  \mathbb{E}_{\D_T \sim P^{N, \pi}} \left[\frac{\sigma^2 N_k(T) }{ N_k(T) + \rho\sigma^2 }\right] \\
	& = \mathbb{E}_{\D_T \sim P^{N,\pi}} \left[\frac{\sigma^2 }{ 1 + \rho\sigma^2 / N_k(T)}\right],
\end{align*}
where the last equality comes from the assumption $N_k(T) \geq 1$. 

Based on the Bayes risk calculation above, we can find a lower bound on the minimax normalized $\ell_2$ risk for each $\rho >0$ as follows:
\begin{align*}
	\inf_{\tilde{\mu}_k}\sup_{\substack{\mu_k \in \mathbb{R} \\P_k \in \mathbb{P}_k(\mu_k,\sigma_k) \\ \nu \in \mathbb{V}, T \in \mathbb{T}}} \mathbb{E}_{Q}  N_k(T)\left(\tilde{\mu}_k -\mu_k \right)^2  
	&\geq   \inf_{\tilde{\mu}_k} \mathbb{E}_{\mu\sim \pi} \mathbb{E}_{\D_T \sim P^N|\mu}  N_k(T)\left(\tilde{\mu}_k -\mu_k \right)^2  \\
	& \geq \mathbb{E}_{\D_T \sim P^{N,\pi}} \left[\frac{\sigma^2 }{ 1 + \rho\sigma^2 / N_k(T)}\right].
\end{align*}
Since we assume sampling and stopping strategies are nonadaptive, the distribution of $N_k(T)$ does not depend on $\pi$. Therefore, by the monotone convergence theorem with $\rho \searrow 0$, we have the following lower bound on the minimax normalized $\ell_2$ risk.
\begin{equation}
	\inf_{\tilde{\mu}_k}\sup_{\substack{P_k \in \mathbb{P}_k(\mu_k,\sigma_k) \\ \nu \in \mathbb{V}, T \in \mathbb{T}}} \mathbb{E}_{Q}  N_k(T)\left(\tilde{\mu}_k -\mu_k \right)^2    \geq \sigma^2.
\end{equation}

From the nonadaptivity of data collecting procedure, it can be easily shown that, for any choice of $P_k \in \mathbb{P}_k(\mu_k, \sigma_k), \nu \in \mathbb{V}, T \in \mathbb{T}$ and the corresponding $Q = Q(P_k, \nu, T)$, we have
\[
\mathbb{E}_Q N_k(T)\left(\overline{Y}_k(T) - \mu_k\right)^2 = \sigma_k^2,
\]
where $\overline{Y}_k(T)$ is the sample mean at time $T$. This observation shows that the minimax risk is equal to $\sigma_k^2$ and the sample mean estimator achieves it as claimed.

\subsection{Proof of Proposition~\ref{prop::bounds_l2_finite_moment}}
\label{Appen::subSec::bounds_l2_finite_moment}
The proof of Proposition~\ref{prop::bounds_l2_finite_moment} relies on a lower bound of $D_{f_q}$ and arguments in \cite{jiao2017dependence} which is summarized in Lemma~\ref{lemma::phi_diver_lower_bound} in Section~\ref{subSec::Proof_of_second_them_finite_moment}.

To apply the lemma, we first prove the following bound on the expectation of the $p$-norm of the normalized $\ell_2$ loss.
\begin{claim} \label{claim::p_norm_bound_nonadap}
	Under the assumptions of Proposition~\ref{prop::bounds_l2_finite_moment},	for each $k \in[{\K}]$ and for any fixed $p > 1$ we have that
	\begin{equation} \label{eq::p_norm_bound_nonadap}
		\left\|N_k(T)\left(\hat{\mu}_k(T)-\mu_k \right)^{2}\right\|_p   \leq C_{p} \left(\sigma_k^{(2p)}\right)^2,
	\end{equation}
	where $C_{p}$ is a constant depending only on $p$.
\end{claim}
The proof of the claim is based on the Marcinkiewicz-Zygmund (M-Z) inequality. We cite the following form of the inequality for completeness. 

\begin{lemma}[ \cite{marcinkiewicz1937fonctions}]
	For any $p \geq 1$, if $X_1,\dots X_n$ are independent random variables with $\mathbb{E}[X_i] = 0$ and $\mathbb{E}|X_i
	|^p < \infty$ for all $i = 1, \dots, n$ then the following inequality holds.
	\begin{equation}
		\mathbb{E}\left[\left|\sum_{i=1}^n X_i\right|^p\right] \leq B_p \mathbb{E}\left[\left(\sum_{i=1}^n |X_i|^2\right)^{p/2}\right],
	\end{equation}
	where $B_p>0$ is a constant depending only on $p$.
\end{lemma}

\begin{proof}[Proof of Claim~\ref{claim::p_norm_bound_nonadap}]
	For simple notations, let $W_t := \mathbbm{1}\left(A_t = k\right)$ and $Z_t = Y_t - \mu_k$. Then from the M-Z inequality, we have 
	\begin{align*}
		&\mathbb{E}\left|N_k(T)\left(\hat{\mu}_k(T) - \mu_k\right)^2\right|^p \\
		& =\mathbb{E}\left| \frac{1}{\sqrt{N_k(T)}}\sum_{t=1}^T W_t Z_t\right|^{2p} \\
		& = \mathbb{E}\left[\mathbb{E}\left[\left| \frac{1}{\sqrt{N_k(T)}}\sum_{t=1}^T W_t Z_t\right|^{2p} \mid \{W_t\}_{t\geq1}\right]\right]\\
		& \leq B_p \mathbb{E}\left[\mathbb{E}\left[\left( \frac{1}{N_k(T)}\sum_{t=1}^T W_t |Z_t|^2\right)^{p} \mid \{W_t\}_{t\geq1}\right]\right]~~(\text{by M-Z inequality and $W_t^2 =W_t$})\\
		& \leq B_p \mathbb{E}\left[\mathbb{E}\left[\left( \frac{1}{N_k(T)}\sum_{t=1}^T W_t |Z_t|^{2p}\right) \mid \{W_t\}_{t\geq1}\right]\right]~~(\text{by Jensen's inequality with $N_k(T) = \sum_{t=1}^{T} W_t$}) \\
		& \leq B_p \left(\sigma^{(2p)}\right)^{2p},
	\end{align*}
	which implies the claimed inequality with $C_p := (B_p)^{1/p}$.
\end{proof}
Now, we have all building blocks to complete the proof of Proposition~\ref{prop::bounds_l2_finite_moment}. For each $k \in [{\K}]$, set $   P_k = \mathcal{L}\left(\D_{T} | {\kappa} = k\right)$, $Q = \mathcal{L}\left(\D_{T}\right)$ and 
\[
f_k = \lambda N_{k}(T) \left(\hat{\mu}_k(T)-\mu_k \right)^{2}.
\]
for a $\lambda >0$. 	Then, from Lemma~\ref{lemma::phi_diver_lower_bound}, we can lower bound $I_q\left(\kappa, \D_{T}\right)$ in the following way:
\begin{align*}
	\frac{1}{q}I_q\left(\kappa, \D_{T}\right) &= \sum_{k =  1}^{\K} \mathbb{P}(\kappa = k)  \left\{ \frac{1}{q}D_{f_q}\left(\mathcal{L}\left(\D_{T} | {\kappa} = k\right) ||\mathcal{L}\left(\D_{T}\right) \right)\right\} \\
	& \geq  \sum_{k =  1}^{\K} \mathbb{P}(\kappa = k) \left\{ \mathbb{E}_{P_k} \left[f_k\right] - \mathbb{E}_{Q} \left[f_k\right] - \mathbb{E}_Q\left[\frac{\left|f_k\right|^p}{p}\right]\right\}   \\	
	& = \sum_{k = 1}^{\K} \mathbb{P}(\kappa = k)  \left\{\lambda\mathbb{E}\left[N_{k}(T) \left(\hat{\mu}_k(T)-\mu_k \right)^{2} \mid \kappa = k\right] \right. \\
	&~~~~~~~~~~ \left.-\lambda\mathbb{E}\left[N_{k}(T)\left(\hat{\mu}_k(T)-\mu_k \right)^{2} \right] -  \frac{\lambda^p}{p}\mathbb{E}\left[N_{k}(T)\left(\hat{\mu}_k(T)-\mu_k \right)^{2} \right]^p\right\}\\
	& \geq \sum_{k = 1}^{\K} \mathbb{P}(\kappa = k)  \left\{\lambda\mathbb{E}\left[N_{k}(T)\left(\hat{\mu}_k(T)-\mu_k \right)^{2} \mid \kappa = k\right] -  \left(\lambda \sigma_k^2  + (\lambda C_{p}(\sigma_k^{(2p)})^2)^{p} / p\right)   \right\}\\
	& =  \lambda\mathbb{E}\left[N_{\kappa}(T)\left(\hat{\mu}_\kappa(T)-\mu_\kappa \right)^{2}\right] -  \left(\lambda \|\sigma_\kappa\|_2^2  + \frac{\lambda^p C_{p}^p}{p}\|\sigma_\kappa^{(2p)}\|_{2p}^{2p}\right).
\end{align*}
Since this inequality holds for any $\lambda >0$, we get
\begin{align*}
	\mathbb{E}\left[N_{\kappa}(T) \left(\hat{\mu}_\kappa(T)-\mu_\kappa \right)^{2}\right] &=  \|\sigma_\kappa\|_2^2 + \inf_{\lambda >0} \frac{1}{\lambda}\left\{ \frac{I_q\left(\kappa, \D_{T}\right)}{q}  +  \frac{\lambda^p C_{p}^p}{p}\|\sigma_\kappa^{(2p)}\|_{2p}^{2p}\right\} \\
	& =  \|\sigma_\kappa\|_2^2 + C_{p}\|\sigma_\kappa^{(2p)}\|_{2p}^2  I_q^{1/q}\left(\kappa, \D_{T}\right),
\end{align*}
which completes the proof. 

\subsection{Proof of Fact~\ref{fact::KL_equiv_psi_star}} \label{subSec::prop_KL_equiv_psi_star}
Let $\theta_1$ and $\theta_0$ be natural parameters corresponding to $\mu_1$ and $\mu_0$ such that $\mu_1 = B^\prime(\theta_1)$ and  $\mu_0 = B^\prime(\theta_0)$. From well-known fact about the KL divergence in  exponential family theory, 
\begin{align*}
	\ell_{KL}(\mu_1, \mu_0) &= D_{KL}\left(\theta_1 \| \theta_0 \right) \\
	&= B^\prime(\theta_1)\left(\theta_1 - \theta_0\right) - B(\theta_1) + B(\theta_0).
\end{align*}
Since $\mu = B^\prime (\theta)$ and $\psi_{\mu}(\lambda)  := \lambda \mu + \psi(\lambda ; \theta)  = B(\lambda + \theta) - B(\theta)$, its derivative with respect to $\lambda$ is equal to $\psi_{\mu}^\prime (\lambda) = B^\prime(\lambda+\theta)$. Thus,
\begin{align*}
	D_{KL}\left(\theta_1 \| \theta_0 \right) 
	&= B^\prime(\theta_1)\left(\theta_1 - \theta_0\right) - B(\theta_1) + B(\theta_0)\\
	& = \psi^{\prime}_{\mu_0}(\theta_1 - \theta_0) \left(\theta_1 - \theta_0\right) - \psi_{\mu_0}(\theta_1 - \theta_0)\\
	&=     \psi_{\mu_0}(0)- \psi_{\mu_0}(\theta_1 - \theta_0)-\psi^{\prime}_{\mu_0}(\theta_1 - \theta_0) \left(0 - (\theta_1 - \theta_0)\right)~~\text{ (since $\psi_{\mu_0}(0) = 0$.)}\\
	&= D_{\psi_{\mu_0}}(0, \theta_1 - \theta_0) \\
	&= D_{\psi_{\mu_0}^*} \left(\psi^{\prime}_{\mu_0}(\theta_1 - \theta_0), \mu_0\right) \\
	& =  D_{\psi_{\mu_0}^*} (\mu_1, \mu_0),
\end{align*}
where the last equality comes from the fact that $\psi_{\mu_0}^\prime (\theta_1 - \theta_0) = B^\prime(\theta_1) = \mu_1$. The second-to-last equality stems instead from the duality of the Bregman divergence, which we state below.
\begin{fact}
	Let $f : \Lambda \rightarrow \mathbb{R}$ be a strictly convex and continuously differentiable on a open interval $\Lambda \subset \mathbb{R}$. For any $\theta_1, \theta_0 \in \Lambda$, let $\mu_1, \mu_2$ be the corresponding dual points satisfying $f'(\theta_j) = \mu_j$ for $j = 0, 1$. Then, we have
	\begin{equation}
		D_f(\theta_0, \theta_1) = D_{f^*}(\mu_1, \mu_0),
	\end{equation}
	where $f^*$ is the convex conjugate of $f$.
\end{fact}

Therefore, we have the first claimed equality, $\ell_{KL}(\mu_1, \mu_0)   = D_{\psi_{\mu_0}^*} (\mu_1, \mu_0)$. To show the second one, first note that $\psi^{\prime}_{\mu_0}(0) = B^\prime(\theta_0) = \mu_0$, so that $\psi_{\mu_0}^*(\mu_0) = \psi_{\mu_0}^{*\prime}(\mu_0) =0$. Therefore, we have that
\[
D_{\psi_{\mu_0}^*} (\mu_1, \mu_0) = \psi_{\mu_0}^* (\mu_1) - \psi_{\mu_0}^* (\mu_0) - \psi_{\mu_0}^{*\prime}(\mu_0) \left(\mu_1  - \mu_0\right) = \psi_{\mu_0}^* (\mu_1),
\]
which verifies the second claimed equality. The last equality can be established as follows:
\begin{align*}
	\psi_{\mu_0}^* (\mu_1) &= \sup_{\lambda} \lambda \mu_1 - \psi_{\mu_0}(\lambda) \\
	&=\sup_{\lambda} \lambda \mu_1  - \left[\lambda\mu_0 + \psi(\lambda) \right]\\
	&= \sup_{\lambda} \lambda \left(\mu_1 -\mu_0\right) -\psi(\lambda) \\
	& = 	\psi^* (\mu_1- \mu_0), 
\end{align*}
as desired.

\subsection{Proof of Proposition~\ref{prop::minimax_optimality}} \label{Append::subSec::minimax_optimal}
By following similar arguments to the ones used in the proof of Lemma~\ref{lemma::deviation_ineq}, we first show that, for any $\delta >0$, the following deviation inequality holds.
\begin{equation}\label{eq::deterministic_bound}
	\mathbb{P}\left( D_{\psi_{\mu}^*}(\hat{\mu}(n), \mu)  \geq \delta \right) \leq 2 e^{-n \delta}.	    
\end{equation}
\begin{proof}[Proof of inequality~\eqref{eq::deterministic_bound}]
	For any $\epsilon \geq 0$ and $\lambda \in [0,\lambda_{\max}  ) \subset \Lambda $, we have
	\begin{align*}
		\mathbb{P} \left(  S_k(n) /n  -\mu_k \geq \epsilon \right) 
		& = \mathbb{P}\left(  S_k(n) - n\mu_k \geq n \epsilon \right) \\
		&=  \mathbb{P}\left( e^{\lambda \left( S_k(n) - n \mu_k \right)} \geq  e^{\lambda n \epsilon} \right)\\
		& \leq e^{-\lambda n \epsilon}\mathbb{E}\left[e^{\lambda \left( S_k(n) - n \mu_k \right)} \right],
	\end{align*}
	where in the final step we have used Markov's inequality.  The last term can be  bounded as follows:
	\begin{align*}
		e^{-\lambda n \epsilon}\mathbb{E}\left[e^{\lambda \left( S_k(n) - n \mu_k \right)} \right] &=  e^{-\lambda n \epsilon}\prod_{i=1}^n\mathbb{E}\left[e^{\lambda \left( X_i -\mu_k \right)} \right] \\
		&= e^{-\lambda n \epsilon}\prod_{i=1}^n\mathbb{E}\left[e^{\psi(\lambda)} \right] \\
		& = e^{n\left(\psi(\lambda)  - \lambda \epsilon\right)}.
	\end{align*}
	Since the bound holds for any $\lambda \in [0,\lambda_{\max})$, we have the  following intermediate bound on the deviation probability:
	\begin{equation}
		\mathbb{P} \left(  S_k(n) / n  -\mu_k \geq \epsilon \right)  \leq \inf_{\lambda \in [0, \lambda_{\max})} e^{n\left(\psi(\lambda)  - \lambda \epsilon\right)}.
	\end{equation}
	Since $\epsilon \geq 0$, the convex conjugate of $\psi$ at $\epsilon$ can be written as
	\[
	\psi^*(\epsilon) = \sup_{\lambda\in \Lambda} \left\{\lambda \epsilon-  \psi(\lambda)\right\} = \sup_{\lambda\in [0, \lambda_{\max})} \left\{\lambda \epsilon-  \psi(\lambda)\right\}.
	\]
	Using this identity, the deviation probability can be further bounded as
	\begin{align*}
		\mathbb{P} \left(  S_k(n) / n  -\mu_k \geq \epsilon \right) 
		&\leq  \inf_{\lambda \in [0, \lambda_{\max})} e^{n\left(\psi(\lambda)  - \lambda \epsilon\right)}\\
		&= \exp\left(-n\sup_{\lambda \in [0, \lambda_{\max})} \left\{\lambda\epsilon - \psi(\lambda)\right\} \right) \\ 
		&=e^{-n \psi^*(\epsilon)}
	\end{align*}
	Using the same argument, it also follows that 
	\[
	\mathbb{P} \left(  S_k(n) / n  -\mu_k \leq -\epsilon \right) \leq  e^{-n \psi^*(-\epsilon)}
	\]	
	Since $\psi^*$ is a non-negative convex function with $\psi^*(0) = 0$, for any $\delta \geq0$, there exist $\epsilon_{1}, \epsilon_{2} \geq 0$ with $\psi^*(\epsilon_{1}) = \psi^*(- \epsilon_{2})  = \delta$ such that 
	\[
	\left\{z \in \mathbb{R} : \psi^*(z) \geq \delta  \right\} = \left\{z \in \mathbb{R} :  z  \geq \mu_k + \epsilon_{1}, z \leq \mu_k-\epsilon_{2}  \right\}.
	\]
	Therefore, for any $\delta \geq 0$, we conclude that
	\begin{align*}
		\mathbb{P}\left(D_{\psi^*_{\mu_k}}(\hat{\mu}_k(n) , \mu_k) \geq \delta \right) 
		&= \mathbb{P}\left(\psi^*_{\mu_k}(S_k(n) / n) \geq \delta\right) ~~\text{(By the equality~\eqref{eq::KL_equiv_psi_star} in Fact~\ref{fact::KL_equiv_psi_star}.)} \\
		& \leq   \mathbb{P}\left(S_k(n) / n)   - \mu_k\geq  \epsilon_{1} \right)  + \mathbb{P}\left(S_k(n) / n)  - \mu_k\leq - \epsilon_2 \right)  \\
		&\leq 2 e^{-n \delta},
	\end{align*}
	as desired.
\end{proof}
Now, we return to the proof of Proposition~\ref{prop::minimax_optimality}. Based on the deviation inequality~\eqref{eq::deterministic_bound}, the risk under the non-adaptive setting can be bounded as
\begin{align*}
	\mathbb{E}_{\theta} D_{\psi_{\mu}^*}(\hat{\mu}(n), \mu)
	& = \int_{0}^{\infty} 	\mathbb{P}\left( D_{\psi_{\mu}^*}(\hat{\mu}(n), \mu)  \geq \delta \right)  \mathrm{d} \delta \\
	&\leq 2\int_{0}^{\infty} e^{-n\delta} \mathrm{d} \delta \\
	& = \frac{2}{n},
\end{align*}
which completes the proof of the upper bound. 
To get a lower bound on the minimax risk, we use the following lemma.
\begin{lemma}[Modified version of Theorem 2.2 in \cite{Tsybakov:2008}] \label{lemma::LeCam_local_tri}
	Let $\left\{ P_\theta : \theta \in \Theta \right\}$ be a family of probability measures parameterized by $\theta \in \Theta$ and let $s >0$. Suppose a loss function $l : \Theta \times \Theta \mapsto [0, \infty)$ satisfies the local triangle inequality condition with positive numbers $M \leq 1$ and $\epsilon_0$. Also assume that there exist $\theta_1, \theta_0 \in \Theta$ such that  $\ell(\theta_1, \theta_0) \geq 2s / M$ for some $s \leq \epsilon_0$.  Then,  if $D_{KL}(P_{\theta_1} \| P_{\theta_0}) \leq \alpha < \infty$, we have 
	\begin{equation}
        		\inf_{\hat{\theta}}\sup_{\theta \in \Theta} \mathbb{P}_{\theta} \left( \ell(\hat{\theta}, \theta) \geq s \right) \geq \frac{1}{4} \exp(-\alpha).
	\end{equation}
\end{lemma}
\Cref{lemma::LeCam_local_tri} directly comes from Theorem 2.2 in \cite{Tsybakov:2008} except the fact that the loss function $l$ does not necessarily satisfy triangle inequality but it satisfies the local triangle inequality. This lemma can be also viewed as a simplified version of Theorem 1 in \cite{yang1999information} where the idea of deploying local triangle inequality into the minimax lower bound proof comes from. For the completeness, we prove the following inequality which demonstrates how the local triangle inequality can be integrated the original proof of Theorem 2.2 in \cite{Tsybakov:2008}.
\begin{equation}
    	\inf_{\hat{\theta}}\sup_{\theta \in \Theta} \mathbb{P}_{\theta} \left( \ell(\hat{\theta}, \theta) \geq s \right) \geq 	\inf_{\psi}\max_{j \in \{0,1\}} \mathbb{P}_{\theta_j} \left( \psi \neq j \right),
\end{equation}
under the condition of \cref{lemma::LeCam_local_tri}.
\begin{proof}
    For any estimator $\hat{\theta}$, let $\psi^*$ be the minimum distance test defined by
\begin{equation}
    \psi^* := \argmin_{j \in \{0,1\}}\ell(\hat{\theta}, \theta_j).
\end{equation}
To prove the claimed inequality, it is enough to show that if $\psi^* \neq j$ then $\ell(\hat{\theta}, \theta_j) \geq s $ for each $j = 0, 1$. First consider the case $j = 0$. In this case, the condition $\psi^* \neq 0$ implies  that $\ell(\hat{\theta}, \theta_0) \geq \ell(\hat{\theta}, \theta_1)$. For the sake of deriving a contradiction, assume $\ell(\hat{\theta}, \theta_0) < s$, which also implies $\ell(\hat{\theta}, \theta_1) < s$ Then, from the local triangle inequality condition, we have 
\begin{align*}
    2s \leq M \ell(\theta_1, \theta_0) &\leq \ell(\hat{\theta}, \theta_0) + \ell(\hat{\theta}, \theta_1) \\
    & \leq 2 \ell(\hat{\theta}, \theta_0),
\end{align*}
which leads to a contradiction. By changing roles of $\theta_0$ and $\theta_1$, we can prove the case $j = 1$, and thus the claimed inequality, as desired.
\end{proof}
To finish the proof of the lower bound part in \cref{prop::minimax_optimality},

note that, for any $\theta_1, \theta_0 \in \Theta$, the KL divergence can be written as
\[
D_{KL}(P_{\theta_1}^n || P_{\theta_0}^n) =n  D_{KL}(\theta_1 || \theta_0) = n D_{\psi_{\mu_0}^*}(\mu_1, \mu_0), 
\]
where $\mu_1$ and $\mu_0$ are corresponding mean parameters. If $n$ is large enough such that $\sqrt{\frac{M\log2}{2n}} \leq \epsilon_0$, we can always find $\theta_1, \theta_0 \in \Theta$ such that $D_{\psi_{\mu_0}^*}(\mu_1, \mu_0)  = \frac{M\log 2}{n}$.  Then, the condition in Lemma~\ref{lemma::LeCam_local_tri} can be satisfied with $\ell = \sqrt{D_{\psi_{\mu}^*}}$, $s = \sqrt{\frac{M\log 2}{2n}}$ and $\alpha = \log 2$. Therefore,
\begin{align*}
	\inf_{\hat{\mu}}\sup_{\mu} \mathbb{E}_{\theta} D_{\psi_{\mu}^*}(\hat{\mu}, \mu) 
	& \geq s^2	\inf_{\hat{\mu}}\sup_{\mu} \mathbb{P}\left(  \sqrt{D_{\psi_{\mu}^*}(\hat{\mu}, \mu)} \geq s \right) \\
	& \geq  \frac{s^2}{4} \exp(-\alpha)\\
	&= \frac{M\log 2}{16 n},
\end{align*}
as desired.

\section{Equivalence between \texorpdfstring{\MakeLowercase{$n_{t}^{\eff} \to \infty$}}{n eff go to infinity} and $N$\texorpdfstring{\MakeLowercase{$(t) \stackrel{a.s.}{\to} \infty$}}{N(t) go to infinity}}

Before we state and formally prove our claim, we first state a useful fact.

\begin{fact}[Theorem 13.7 in \cite{williams1991probability} with $X=0$]
	Let $\{X_t\}_{t \in \mathbb{N}}$ be a sequence of random variables with finite first moments. Then $\mathbb{E}\left|X_t\right|\to 0$ if and only if the following conditions are satisfied:
	\begin{enumerate}
		\item $X_t \overset{p}{\to} 0$.
		\item $\{X_t\}_{t \in \mathbb{N}}$ is uniformly integrable.
	\end{enumerate}
\end{fact}

\noindent Recall that $n_{t}^\eff = [1/N(t)]^{-1}$, we are now in place to prove the following claim.

\begin{proposition} \label{prop::equiv_n_eff_in_prob}
	As long as $N(t) \geq 1$,  we have that $n_{t}^\eff \to \infty$ as $t \to \infty$ if and only if $N(t) \stackrel{p}{\to} \infty$ $t \to \infty$.
\end{proposition}
\begin{proof}
	The assumption of $N(t) \geq 1$ ensures that the sequence $\{1/N(t)\}_{t \in \mathbb{N}}$ is uniformly integrable. Substituting $X_t = 1/N(t)$ into the aforementioned fact, we have that
	\begin{align*}
		n_{t}^\eff \to \infty ~~\text{as}~~t \to \infty
		~&\Leftrightarrow~ \mathbb{E}\left[1/N(t)\right] \to 0 ~~\text{as}~~t \to \infty\\
		~&\stackrel{\text{fact}}{\Leftrightarrow}~ 1/N(t) \overset{p}{\to} 0~~\text{as}~~t \to \infty \\
		~&\Leftrightarrow~ N(t) \overset{p}{\to} \infty~~\text{as}~~t \to \infty,
	\end{align*}
	as desired.
\end{proof}

\begin{proposition}
	If $\{N(t)\}$ is a nondecreasing sequence, we have that $N(t) \overset{p}{\to} \infty$ as $t \to \infty$ implies $N(t) \overset{a.s.}{\to} \infty$ as $t \to \infty$.
	\begin{proof}
		If $N(t) \overset{p}{\to} \infty$ as $t \to \infty$, there exists a subsequence that goes to $\infty$ almost surely. Therefore, we must have $\limsup\limits_{t \to \infty} N(t) = \infty$ almost surely. By the monotonicity of $\{N(t)\}$, we have $\limsup\limits_{t \to \infty} N(t) = \lim\limits_{t \to \infty} N(t)$ which implies that $N(t) \overset{a.s.}{\to} \infty$ as $t \to \infty$.
	\end{proof}
\end{proposition}

Returning to the MABs setting, the previous propositions show that, as long as $N_k(t) \geq b > 0$, the condition $n_{k,t}^\eff \to \infty$ implies that $N_k(t) \overset{a.s.}{\to} \infty$ since $\{N_k(t)\}$ is monotone  for each arm $k \in [\K]$.  For a sequence of chosen arms, however, if the sequence  $\{N_{\kappa_t}(T_t)\}$  is not monotone, $n_{\kappa_{t}, t}^\eff \to \infty$ does not imply $N_{\kappa_t}(\tau_t) \overset{a.s.}{\to} \infty$ as shown in the next example.
\begin{example}
	Consider a two-armed bandit; pull the first arm at time 1 and the second arm forever after. Thus, $N_1(t) = 1$ for all $t \geq 1$, and $N_2(t) = t-1$ for all $t \geq 2$ with $N_2(1)= 0$. 
	Define $\tau_t = t+1$ and let $\{\kappa_t\}$ be a sequence of random choice functions defined by a uniform random variable $U \in Unif[0,1]$ such that $\kappa_t = 1$ if $ U \in \left[ \frac{j}{2^k}, \frac{j+1}{2^k}\right]$ where $k$ and $j$ are given by $k = \left\lfloor \log_2(t) \right\rfloor$ and $t = 2^k + j$. if $ U \notin \left[ \frac{j}{2^k}, \frac{j+1}{2^k}\right]$, define $\kappa_t = 2$. It is clear $N_{\kappa_t}(\tau_t) \overset{p}{\to} \infty$, and the Proposition~\ref{prop::equiv_n_eff_in_prob} implies that  $n_{\kappa_{t},t}^\eff \to \infty$. However, for any given $U$, $N_{\kappa_t}(\tau_t) \not\to \infty$ and thus $\mathbb{P}\left(N_{\kappa_t}(\tau_t) \to \infty\right) = 0$. 
\end{example}

\section{Alternative bounds using self-normalized process} 

For sub-Gaussian arms, it is known that $\mathbb{E}\left[\exp\left\{\lambda \left(S(\Tau)-\mu N(\Tau)\right) - \frac{\sigma^2\lambda^2}{2} N(\Tau) \right\}\right] \leq 1$ for all $\lambda\in\mathbb{R}$. In this sub-Gaussian case (only), one may use the following moment bound from the literature on  self-normalized processes. 
\begin{fact}[Theorem 2.1 in \cite{de2004self}]
	If $\mathbb{E}\sqrt{N(\Tau)} < \infty$,
	\begin{equation} \label{eq::self_normal}
		\mathbb{E}\exp\left\{ \tilde{N}^\mathrm{E}(\Tau) \frac{\left(\hat{\mu}(\Tau) - \mu)\right)^2}{4\sigma^2} \right\} \leq \sqrt{2},
	\end{equation}
	where $    \tilde{N}^\mathrm{E}(\Tau) := N^2(\Tau) / \left( N(\Tau) + \left(\mathbb{E}\sqrt{N(\Tau)}]\right)^2  \right).$
	\label{fact:delapena}
\end{fact}

We can use the above fact and the Donsker-Varadhan representation to derive an alternative bound for the $l_2$ risk of the chosen sample mean at a stopping time $\Tau$ as follows: 
\begin{align*}
	I(\kappa ; \mathcal{D}) &= \sum_{k =  1}^{\K} \mathbb{P}(\kappa = k)  D_{KL}\left(\mathcal{L}\left(\mathcal{D} | {\kappa} = k\right) ||\mathcal{L}\left(\mathcal{D}\right) \right) \\
	& \geq  \sum_{k =  1}^{\K} \mathbb{P}(\kappa = k)  \mathbb{E}_{P_k} \left[f_k\right] - \log \mathbb{E}_{Q} \left[e^{f_k}\right]   \\	
	& = \sum_{k = 1}^{\K} \mathbb{P}(\kappa = k) \left\{ \mathbb{E} \left[ \tilde{N}_k^\mathrm{E}(\Tau) \frac{\left(\hat{\mu}_k(\Tau) - \mu_k)\right)^2}{4\sigma^2}    \mid \kappa = k \right] \right.\\
	&~~~~~~\left. - \log \mathbb{E}\left[ \exp\left\{ \tilde{N}_k^\mathrm{E}(\Tau) \frac{\left(\hat{\mu}(\Tau) - \mu)\right)^2}{4\sigma^2}\right\}\right]   \right\} \\
	& \geq \sum_{k = 1}^{\K} \mathbb{P}(\kappa = k)  \left\{\mathbb{E} \left[ \tilde{N}_k^\mathrm{E}(\Tau) \frac{\left(\hat{\mu}_k(\Tau) - \mu_k)\right)^2}{4\sigma^2}    \mid \kappa = k \right] - \frac{\log 2}{2} \right\}\\
	& =  \mathbb{E} \left[ \tilde{N}_\kappa^\mathrm{E}(\Tau) \frac{\left(\hat{\mu}_\kappa(\Tau) - \mu_\kappa)\right)^2}{4\sigma^2}  \right] - \frac{\log 2}{2}.
\end{align*}
By rearranging terms, we have the following bound on the $\ell_2$ risk.
\begin{equation}\label{eq::l2_bound_self_norm}
	\mathbb{E} \left[ \tilde{N}_\kappa^\mathrm{E}(\Tau) \left(\hat{\mu}_\kappa(\Tau) - \mu_\kappa)\right)^2 \right] \leq 4\sigma^2\left[I(\kappa ; \D_{\Tau})+ \frac{\log 2}{2}\right].
\end{equation}

Recall that, for sub-Gaussian arms, the bound in Theorem~\ref{thm::bounds_in_fully_adap} can be written as
\begin{equation} \label{eq::l2_bound_appen}
	\mathbb{E} \left[ \dbtilde{N}_\kappa(\tau) \left(\hat{\mu}_\kappa(\tau) - \mu_\kappa)\right)^2 \right] \leq 2C_b\sigma^2\left[I(\kappa ; \D_{\Tau})+ 1.25\right].
\end{equation}

Bounds \eqref{eq::l2_bound_self_norm} and \eqref{eq::l2_bound_appen} are matched to each other up to a constant factor. However, corresponding normalized $\ell_2$ risks in LHS shows interesting differences. First, the bound~\eqref{eq::l2_bound_self_norm} based on Fact~\ref{fact:delapena} holds only at a stopping time but our bound holds at an arbitrary random time. Second, the bound~\eqref{eq::l2_bound_self_norm} is applicable only to the sub-Gaussian case since it is non-trivial to extend the Fact~\ref{fact:delapena} to general sub-$\psi$ cases. Third, if the random sample size $N$ is highly concentrated at a constant, the normalizing factor $\tilde{N}_\kappa^{\mathrm{E}}$ tends to be larger than our normalizing factor $\dbtilde{N}_\kappa$ and thus the bound~\eqref{eq::l2_bound_self_norm} yields a tighter control on the $\ell_2$ risk. On the other hand, if the random sample size $N$ has a large variability, our normalizing factor $\dbtilde{N}_\kappa$ tends to be larger than $\tilde{N}_\kappa^{\mathrm{E}}$ since $(\mathbb{E}\sqrt{N})^2$ can be significantly larger than $N$ with a high probability. In this case, our bound~\eqref{eq::l2_bound_appen} yields a tighter control on the $\ell_2$ risk than the bound~\eqref{eq::l2_bound_self_norm}.

\end{appendices}
	
\end{document}